\documentclass[12pt,a4paper]{amsart}

\usepackage{mathrsfs}
\usepackage{tikz}
\usepackage{amsfonts}
\usepackage{amsmath}
\usepackage{amssymb}
\usepackage{enumerate}
\usepackage{hyperref}
\usepackage{latexsym}
\usepackage{xcolor}
\usepackage[shortlabels]{enumitem}
\usepackage[numbers,sort&compress]
{natbib}
\setlist[enumerate,1]{label={\upshape(\roman*)}}

 \makeatletter
    
    \newcommand{\Rmnum}[1]
    {\expandafter\@slowromancap\romannumeral #1@}
    \makeatother

\newtheorem{thm}{Theorem}[section]

\newtheorem{prop}[thm]{Proposition}
\newtheorem{lemma}[thm]{Lemma}
\newcounter{foo}[section]
\newcounter{fooo}[subsection]
\newtheorem{step}[foo]{Step}
\newtheorem{stepp}[fooo]{Step}
\newtheorem{substep}{Step}[foo]

\newtheorem{example}[thm]{Example}
\newtheorem{defin}[thm]{Definition}

\theoremstyle{definition}
\newtheorem{remark}[thm]{Remark}

\def\wz{\tilde}

\title[Semicomplete multipartite weakly distance-regular digraphs]{Semicomplete multipartite weakly distance-regular digraphs}

\date{}

\thanks{*Corresponding author}

\author[Li]{Shuang Li}
\address{Laboratory of Mathematics and Complex Systems (MOE),~School of Mathematical Sciences\\Beijing Normal University\\Beijing 100875\\China}
\email{lishuangyx@mail.bnu.edu.cn}

\author[Yang]{Yuefeng Yang*}
\address{School of Science\\China University of Geosciences\\Beijing 100083\\China}
\email{yangyf@cugb.edu.cn}

\author[Wang]{Kaishun Wang}
\address{Laboratory of Mathematics and Complex Systems (MOE),~School of Mathematical Sciences\\Beijing Normal University\\Beijing 100875\\China}
\email{wangks@bnu.edu.cn}

\begin{document}
\begin{abstract}
	  A digraph is semicomplete multipartite if its underlying graph is a complete multipartite graph. As a special case of semicomplete multipartite digraphs, J{\o}rgensen et al. \cite{JG14} initiated the study of  doubly regular team tournaments.
	  As a natural extension, we introduce  doubly regular team semicomplete multipartite digraphs and show that such digraphs fall into three types. Furthermore, we give a characterization of all semicomplete multipartite commutative weakly distance-regular digraphs.
\end{abstract}



\maketitle
\section{Introduction}
A \emph{digraph} $\Gamma$ is a pair $(V\Gamma,A\Gamma)$, where $V\Gamma$ is a finite nonempty set of vertices and $A\Gamma$ is a set of ordered pairs ({\em arcs}) $(x,y)$ with distinct vertices $x$ and $y$. A subdigraph of $\Gamma$ induced by a subset $U\subseteq V\Gamma$ is denoted by $\Gamma[U]$. For any arc $(x,y)\in A\Gamma$, if $A\Gamma$ also contains an arc $(y,x)$, then $\{(x,y),(y,x)\}$ can be viewed as an {\em edge}. We say that $\Gamma$ is an \emph{undirected graph} or a {\em graph} if $A\Gamma$ is a symmetric relation. A vertex $x$ is {\em adjacent} to $y$ if $(x,y)\in A\Gamma$. In this case, we also call $y$ an \emph{out-neighbour} of $x$, and $x$ an \emph{in-neighbour} of $y$. The set of all out-neighbours of $x$ is denoted by $N_{\Gamma}^{+}(x)$, while the set of in-neighbours is denoted by $N_{\Gamma}^{-}(x)$. If no confusion occurs, we write $N^{+}(x)$ (resp. $N^{-}(x)$) instead of $N_{\Gamma}^{+}(x)$ (resp. $N_{\Gamma}^{-}(x)$).   A digraph is said to be {\em regular of valency} $k$ if the number of in-neighbour and out-neighbour of all vertices are equal to $k$. The
\emph{adjacency matrix} $A$ of $\Gamma$ is the $|V\Gamma|\times |V\Gamma|$ matrix whose $(x,y)$-entry is $1$ if $y\in N^{+}(x)$, and
$0$ otherwise. 




 For the digraphs $\Gamma$ and $\Sigma$, the {\em lexicographic product} $\Gamma\circ\Sigma$ is the digraph with vertex set $V\Gamma\times V\Sigma$ such that $((u_1,u_2),(v_1,v_2))$ is an arc if and only if either $(u_1,v_1)\in A\Gamma$, or
$u_1=v_1$ and $(u_2,v_2)\in A\Sigma$. We say that $\Gamma\circ \Sigma$ is an {\em$m$-coclique extension} of $\Gamma$ if $\Sigma$ is an independent set of size $m$.

A \emph{path} of length $r$ from $x$ to $y$ in the digraph $\Gamma$ is a finite sequence of vertices $(x=w_{0},w_{1},\ldots,w_{r}=y)$ such that $(w_{t-1}, w_{t})\in A\Gamma$ for $1\leq t\leq r$. A digraph (resp. graph) is said to be \emph{strongly connected} (resp. \emph{connected}) if, for any vertices $x$ and $y$, there is a path from $x$ to $y$. A path $(w_{0},w_{1},\ldots,w_{r-1})$ is called a \emph{circuit} of length $r$ when $(w_{r-1},w_0)\in A\Gamma$. The \emph{girth} of $\Gamma$ is the length of a shortest circuit in $\Gamma$.  The length of a shortest path from $x$ to $y$ is called the \emph{distance} from $x$ to $y$ in $\Gamma$, denoted by $\partial_\Gamma(x,y)$. The maximum value of the distance function in $\Gamma$ is called the \emph{diameter} of $\Gamma$. Let $\wz{\partial}_{\Gamma}(x,y):=(\partial_{\Gamma}(x,y),\partial_{\Gamma}(y,x))$ be the \emph{two-way distance} from $x$ to $y$, and $\wz{\partial}(\Gamma):=\{\wz{\partial}_{\Gamma}(x,y)\mid x,y\in V\Gamma\}$ the \emph{two-way distance set} of $\Gamma$. If no confusion occurs, we write $\partial(x,y)$ (resp. $\tilde{\partial}(x,y)$) instead of $\partial_\Gamma(x,y)$ (resp. ${\tilde{\partial}}_\Gamma(x,y)$).  For any $\wz{i}:=(a,b)\in\wz{\partial}(\Gamma)$, we define $\Gamma_{\wz{i}}$ to be the set of ordered pairs $(x,y)$ with $\wz{\partial}(x,y)=\wz{i}$, and write $\Gamma_{a,b}$ instead of $\Gamma_{(a,b)}$. An arc $(x,y)$ of $\Gamma$ is of {\em type} $(1,r)$ if $\partial(y,x)=r$. Let $T$ be the set of integers $q$ such that $(1,q-1)$ is a type of $\Gamma$. That is, $T=\{q\mid (1,q-1)\in \tilde{\partial}(\Gamma)\}$.

In \cite{KSW03}, the third author and Suzuki proposed a natural directed version of a distance-regular graph (see \cite{AEB98,DKT16} for a background of the theory of distance-regular graphs) without bounded diameter, i.e., a weakly distance-regular digraph. A strongly connected digraph $\Gamma$ is said to be \emph{weakly distance-regular} if the configuration $\mathfrak{X}(\Gamma)=(V\Gamma,\{\Gamma_{\wz{i}}\}_{\wz{i}\in \wz{\partial}(\Gamma)})$ is a non-symmetric association scheme (see the definition of  association scheme in Section 2). We call $\mathfrak{X}(\Gamma)$ the \emph{attached scheme} of $\Gamma$. A weakly distance-regular digraph is \emph{commutative} if its attached scheme is commutative.
For more information about weakly distance-regular digraphs, see \cite{HS04,KSW03,KSW04,YYF16,YYF18,YF22,AM,YYF20,YL23,YYF22,QZ23}. 

For a digraph $\Gamma$, we form the {\em underlying graph} of $\Gamma$ with the same vertex set, and there is an edge between vertices $x$ and $y$ whenever $(x,y)\in A\Gamma$ or $(y,x)\in A\Gamma$. A digraph is {\em semicomplete}, if  its underlying graph is a complete graph.  A digraph $\Gamma$ is {\em semicomplete multipartite}, if  its underlying graph is a complete multipartite graph $K_{n_1,\ldots,n_m}$  with $m\geq 2$ parts and part sizes $n_1,\ldots,n_m\geq 2$.

In \cite{YYF1,YYF2}, the authors classified all commutative weakly distance-regular digraphs whose
underlying graphs are well-known distance-regular graphs, such as Hamming graphs, folded n-cubes,  Doob graphs, Johnson graphs or folded Johnson graphs.

In this paper, we continue this project, and give a characterization of all semicomplete multipartite commutative weakly distance-regular digraphs.

\begin{thm}\label{xushumain4}
A  commutative weakly distance-regular digraph is semicomplete multipartite if and only if it is isomorphic to one of the following digraphs:
\begin{enumerate}
	\item\label{main1-fu2} 
	an $n$-coclique extension of ${\rm Cay}(\mathbb{Z}_{6},\{1,2\})$, where $n\geq 1$;

\item\label{main1-fu7} 
an $l$-coclique extension of a semicomplete weakly distance-regular digraph with girth $2$ or $3$, where $l\geq 2$;	

\item\label{main1-fu0}
	an $n$-coclique extension of $\Sigma\circ C_4$, where $\Sigma$ is a semicomplete weakly distance-regular digraph with girth $3$, where $n\geq 1$;

\item\label{main1-fu5} 
an $n$-coclique extension of a doubly regular $(k,l)$-team tournament of Type  \Rmnum{2} with parameters $(\frac{(k-2)l}{4},\frac{(k-2)l}{4},\frac{l^2(k-1)}{4(l-1)})$, where $k,l\geq 2$ and $n\geq 1$ (see Definition \ref{main defin0});	

	\item\label{main1-fu1} 
a doubly regular $(k,m)$-team semicomplete multipartite digraph of  Type II with $k,m\geq 2$ (see Definition \ref{main defin}).
\end{enumerate}
\end{thm}

The remainder of this paper is organized as follows. In Section 2, we provide the required notation and concepts for association schemes. In Section 3,  as a natural extension of doubly regular team tournaments, we introduce the  definition of doubly regular team semicomplete multipartite digraphs and show that such digraphs fall into three types. In Section 4, we provide the required notation, concepts and preliminary results for a semicomplete multipartite weakly distance-regular digraph. In Section 5, we show that $|T|=1$, $T=\{3,4\}$ or $T=\{2,3\}$.
 In  Sections 6--8, we determine $\Gamma$ when $|T|=1$, $T=\{3,4\}$ or $T=\{2,3\}$, respectively. 
 In Section 9, we prove our main result based on the results in Sections 6--8.

\section{Association schemes}
In this section, we present some notation and concepts for association schemes that we shall use in the remainder of the paper.

A \emph{$d$-class association scheme} $\mathfrak{X}$ is a pair $(X,\{R_{i}\}_{i=0}^{d})$, where $X$ is a finite set, and each $R_{i}$ is a
nonempty subset of $X\times X$ satisfying the following axioms (see \cite{EB21,EB84,PHZ96,PHZ05} for a background of the theory of association schemes):
\begin{enumerate}
\item\label{as-1} $R_{0}=\{(x,x)\mid x\in X\}$ is the diagonal relation;

\item\label{as-2} $X\times X=R_{0}\cup R_{1}\cup\cdots\cup R_{d}$, $R_{i}\cap R_{j}=\emptyset~(i\neq j)$;

\item\label{as-3} for each $i$, $R_{i}^{\top}=R_{i^{*}}$ for some $0\leq i^{*}\leq d$, where $R_{i}^{\top}=\{(y,x)\mid(x,y)\in R_{i}\}$;

\item\label{as-4} for all $i,j,l$, the cardinality of the set $$P_{i,j}(x,y):=R_{i}(x)\cap R_{j^{*}}(y)$$ is constant whenever $(x,y)\in R_{l}$, where $R(x)=\{y\mid (x,y)\in R\}$ for $R\subseteq X\times X$ and $x\in X$. This constant is denoted by $p_{i,j}^{l}$.
\end{enumerate}
A $d$-class association scheme is also called an association scheme with $d$ classes (or even simply a scheme). The integers $p_{i,j}^{l}$ are called the \emph{intersection numbers} of $\mathfrak{X}$. We say that $\mathfrak{X}$ is \emph{commutative} if $p_{i,j}^{l}=p_{j,i}^{l}$ for all $i,j,l$. The subsets $R_{i}$ are called the \emph{relations} of $\mathfrak{X}$. For each $i$, the integer $k_{i}:=p_{i,i^{*}}^{0}$ is called the \emph{valency} of $R_{i}$.   A relation $R_{i}$ is called \emph{symmetric} if $i=i^{*}$, and \emph{non-symmetric} otherwise. An association scheme is called \emph{symmetric} if all relations are symmetric, and \emph{non-symmetric} otherwise.

Let $\mathfrak{X}=(X,\{R_{i}\}_{i=0}^{d})$ be an association scheme. For two nonempty subsets $E$ and $F$ of $\{R_{i}\}_{i=0}^{d}$, define
\[EF:=\{R_{l}\mid\sum_{R_{i}\in E}\sum_{R_{j}\in F}p_{i,j}^{l}\neq0\}.\]
We write $R_{i}R_{j}$ instead of $\{R_{i}\}\{R_{j}\}$,
and $R_i^2$ instead of $\{R_{i}\}\{R_{i}\}$.
If $R_{i^{*}}R_{j}\subseteq F$ for any $R_{i},R_{j}\in F$, we say that $F$ is {\em closed}.
For each $R_i$ with $0\leq i\leq d$, define $\left\langle R_i\right\rangle$ to be the smallest closed subset containing $R_i$.
 We say that $R_{j}$ {\em generates} $\mathfrak{X}$ if $\left\langle R_j\right\rangle=\{R_{i}\}_{i=0}^{d}$. We call $\mathfrak{X}$ \emph{primitive} if every non-diagonal relation generates $\mathfrak{X}$. Otherwise, $\mathfrak{X}$ is said to be \emph{imprimitive}.

Let $\mathfrak{X}=(X,\{R_{i}\}_{i=0}^{d})$ be a commutative association scheme with $|X|=n$. The \emph{adjacency matrix} $A_{i}$ of $R_{i}$ is the $n\times n$ matrix whose $(x,y)$-entry is $1$ if $(x,y)\in R_{i}$, and $0$ otherwise. By the \emph{adjacency} or \emph{Bose-Mesner algebra} $\mathfrak{U}$ of $\mathfrak{X}$ we mean the algebra generated by $A_{0},A_{1},\ldots,A_{d}$ over the complex field. Axioms \ref{as-1}--\ref{as-4} are equivalent to the following:
\[A_{0}=I,\quad \sum_{i=0}^{d}A_{i}=J,\quad A_{i}^{\top}=A_{i^{*}},
\quad A_{i}A_{j}=\sum_{l=0}^{d}p_{i,j}^{l}A_{l},\]
	where $J$ is the matrix with $1$ in every entry, and $I$ denotes the identity matrix.

We close this section with some basic properties of intersection numbers which are used frequently in this paper.

\begin{lemma}\label{jb}
	{\rm (\cite[Chapter \Rmnum{2}, Proposition 2.2]{EB84})} Let $(X,\{R_{i}\}_{i=0}^{d})$ be an association scheme. The following hold:
	\begin{enumerate}
		\item\label{jb-1} $k_{i}k_{j}=\sum_{l=0}^{d}p_{i,j}^{l}k_{l}$;
		
		\item\label{jb-2}  $p_{i,j}^{l}k_{l}=p_{l,j^{*}}^{i}k_{i}=p_{i^{*},l}^{j}k_{j}$;
		
		\item\label{jb-3}  $\sum_{j=0}^{d}p_{i,j}^{l}=k_{i}$;
		
		\item\label{jb-4} $\sum_{r=0}^{d}p_{i,l}^{r}p_{m,r}^{j}=\sum_{t=0}^{d}p_{m,i}^{t}p_{t,l}^{j}$.
	\end{enumerate}
\end{lemma}

\section{Doubly regular team semicomplete multipartite digraphs}

Let $K_{n_1,\ldots,n_m}$ be a complete multipartite graph with $m\geq 2$ parts and part sizes $n_1,\ldots,n_m\geq 2$. The parts are called the {\em partite sets}. If all the partite sets have the same size $r$, this graph is denoted by $K^{m}_{r}$.

An \emph{$(m,r)$-team semicomplete multipartite digraph} is a   digraph with $K^{m}_{r}$ as its underlying graph, where $m,r\geq 2$.
An \emph{$(m,r)$-team tournament}   is an $(m,r)$-team semicomplete multipartite digraph with girth $g\geq 3$.

\begin{defin}{\rm (\cite[Definition 3.13]{JG14})}\label{main defin0}
A regular $(m,r)$-team tournament with adjacency matrix $A$ is said to be a {\em doubly regular $(m,r)$-team tournament with parameters $(\alpha,\beta,\gamma)$} if $$A^2=\alpha A+\beta A^{\top}+\gamma(J-I-A-A^{\top}).$$ 	
\end{defin}

 J{\o}rgensen et al. \cite{JG14}  proved   that a doubly regular $(m,r)$-team tournament is of Type \Rmnum{1}, \Rmnum{2}, or \Rmnum{3}. Let $D$ be a doubly regular $(m,r)$-team tournament with parameters $(\alpha,\beta,\gamma)$, where $VD=V_1\cup \cdots \cup V_m$ is the partition of the vertex set into $m$ independent sets of size $r$.
Recall that $D$ is of {\em Type \Rmnum{2}} if $\beta-\alpha=0$, $r$ is even and $|N^+(x)\cap V_i|=\frac{r}{2}$ for all $x\in VD\setminus V_i$ and $i\in \{1,\ldots,m\}$.

Naturally, we consider the general case: $(m,r)$-team semicomplete multipartite digraphs.  For a digraph $\Gamma$, set $E\Gamma=\{(x,y)\mid (x,y),(y,x)\in A\Gamma\}$, $\overline{\Gamma}:=(V\Gamma,E\Gamma)$, $\vec{\Gamma}:=(V\Gamma,A\Gamma\backslash E\Gamma)$.



\begin{defin}\label{main defin}
Let $\Gamma$ be an $(m,r)$-team semicomplete multipartite digraph with valency $k$. Let  $A_0$ be the adjacency matrix  of  $\overline{\Gamma}$, and $A_1$ be the adjacency matrix of $\vec{\Gamma}$.  We say that $\Gamma$ is  doubly regular if there exist integers  $t$, $\alpha_{i+j}$, $\beta_{i+j}$, $\gamma_{i+j}$, $\eta_{i+j}$ for $0\leq i,j\leq 1$ such that$$
A_iA_j=t\delta_{0,i+j}I+\alpha_{i+j} A_1+\beta_{i+j} A_1^{\top}+\gamma_{i+j} A_0+\eta_{i+j} (J-I-A_1-A_1^{\top}-A_0),$$
	where $\delta$ is Kronecker's delta.
\end{defin}

Note that the equation $$A_1^2=\alpha_{2} A_1+\beta_{2} A_1^{\top}+\gamma_{2} A_0+\eta_{2} (J-I-A_1-A_1^{\top}-A_0)$$ means that the number of  paths of length $2$ in $\vec{\Gamma}$ from a vertex $x$ to a vertex $y$ is $\alpha_2$  if $(x,y)\in A\vec{\Gamma}$,
$\beta_2$ if $(y,x)\in A\vec{\Gamma}$, $\gamma_2$ if $(x,y)\in E\Gamma$, and $\eta_2$ if $(x,y)\notin A\Gamma\cup {A\Gamma}^\top$. 
The meanings of other equations  are similar. Thus,  $\alpha_0=\beta_0$, and the valency of $\overline{\Gamma}$ is $t$, while the valency of $\vec{\Gamma}$ is $k-t$.

In the remainder of this section, we always denote $\Gamma$  as a doubly regular $(m,r)$-team semicomplete multipartite digraph. Let $A_0$, $A_1$, $k$, $t$, $\alpha_i$, $\beta_i$, $\gamma_i$, $\eta_i$ for $0\leq i\leq 2$ be as in Definition \ref{main defin} and  $V\Gamma=V_1\cup V_2\cup \cdots\cup V_m$ be the partition of the vertex set into $m$ independent sets of sizes $r$. For $x\in V_i$, define $A_{j}^{+}(x):=A\vec{\Gamma}(x)\cap V_j$, $A_{j}^{-}(x):={A\vec{\Gamma}}^\top(x)\cap V_j$ and $E_{j}(x):=E\Gamma(x)\cap V_j$.

Let   $(i,j)$ be a pair with $i\ne j$.   If $|E_{j}(x)|$ does not depend on the choice of vertex $x\in V_i$, we set $e_{ij}:=|E_{j}(x)|$. If $|A_{j}^{+}(x)|$ does not depend on the choice of vertex $x\in V_i$, we set $c_{ij}:=|A_{j}^{+}(x)|$.

Next, we will prove that $\Gamma$ is of one of three types:
\begin{enumerate}
	\item[Type I.]\label{main thm-1} $\beta_1+\beta_2-\alpha_1-\alpha_2=r$ and $\Gamma$ is a $r$-coclique extension of a semicomplete digraph;

	\item[Type II.]\label{main thm-2} $\beta_1+\beta_2-\alpha_1-\alpha_2=0$ and $c_{ij}=c_{ji}=\frac{r-e_{ij}}{2}$;

	\item[Type III.]\label{main thm-3}
	$\beta_1+\beta_2-\alpha_1-\alpha_2=\frac{r}{2}$ and for each pair $(i,j)$ with $i\ne j$, the one of the following holds:
	\begin{enumerate}
		\item[\rm (i)]\label{main thm-3-3}  $V_i\times V_j\subseteq E\Gamma$;
		
		\item[\rm (ii)]\label{main thm-3-2} either $V_i$ is partitioned into two nonempty sets $V_i=V_i'\cup V_i''$ so that
		$(V_i'\times V_j)\cup (V_j\times V_i'')\subseteq A\vec{\Gamma}$, or similarly with $i$ and $j$ interchanged;
		
		\item[\rm (iii)]\label{main thm-3-1}	 $V_i,V_j$ are partitioned into two nonempty sets $V_i=V_i'\cup V_i''$ and  $V_j=V_j'\cup V_j''$ so that
		$(V_i'\times V_j')\cup (V_i''\times V_j'')\subseteq E\Gamma$ and $(V_i'\times V_j'')\cup (V_j'\times V_i'')\subseteq A\vec{\Gamma}$.
	\end{enumerate}
\end{enumerate}

\begin{lemma}\label{twopath}
Let $\{R,M\}=\{A\vec{\Gamma},E\Gamma\}$. Suppose that $x\in V_i$ and $y\in V_j$ for a pair $(i,j)$ with $i\ne j$. Then
\begin{align}
\{z\mid(x,z),(y,z)\in R\}&=R(x)\setminus (V_j\cup(A\vec{\Gamma}^{\top}\cup M)(y))=R(y)\setminus (V_i\cup(A\vec{\Gamma}^{\top}\cup M)(x)).\nonumber
\end{align}
\end{lemma}
\begin{proof}
Note that $\{z\mid(x,z),(y,z)\in R\}=R(x)\cap R(y).$
Since the underlying graph of $\Gamma$ is $K^{m}_{r}$,  one obtains $R(y)\cap V_j=R(x)\cap  V_i=\emptyset$, which implies that
\begin{align}
\{z\mid(x,z),(y,z)\in R\}&=(R(x)\setminus V_j)\cap R(y)=R(x)\cap (R(y)\setminus V_i).\nonumber
\end{align}
Since $(R(x)\setminus V_j)\times \{y\}, \{x\}\times (R(y)\setminus V_i)\subseteq A\vec{\Gamma}\cup A\vec{\Gamma}^{\top}\cup E\Gamma$, we have
\begin{align}
(R(x)\setminus V_j)\cap R(y)=R(x)\setminus (V_j\cup(A\vec{\Gamma}^{\top}\cup M)(y)),\nonumber\\
R(x)\cap (R(y)\setminus V_i)=R(y)\setminus (V_i\cup(A\vec{\Gamma}^{\top}\cup M)(x)).\nonumber
\end{align}
Thus, the desired result follows.
\end{proof}

\begin{lemma}\label{edge same}
	Let $x\in V_i$ and $y\in V_j$ for a pair $(i,j)$ with $i\ne j$. Then $e_{ij}=|E_{j}(x)|=|E_{i}(y)|=e_{ji}$.
\end{lemma}
\begin{proof}
	Since $|E\Gamma(x)\backslash V_j|=t-|E_{j}(x)|$ and $|E\Gamma(y)\backslash V_i|=t-|E_{i}(y)|$, from Lemma \ref{twopath}, one gets
	\begin{align}
	&t-|E_{j}(x)|-|\{z\mid(x,z)\in E\Gamma,~(y,z)\in A\vec{\Gamma}^{\top}\cup A\vec{\Gamma}\}|\nonumber\\=&t-|E_{i}(y)|-|\{z\mid(x,z)\in A\vec{\Gamma}^{\top}\cup A\vec{\Gamma},(y,z)\in E\Gamma\}|.\tag{3.1}\label{3.3}
	\end{align}
	Since the underlying graph of $\Gamma$ is $K^{m}_{r}$, we may assume $(x,y)\in A\Gamma$. It follows that $(x,y)\in A\vec{\Gamma}$ or $(x,y)\in E\Gamma$.
	
	Suppose $(x,y)\in A\vec{\Gamma}$. By Definition \ref{main defin}, we have
	\begin{align}
	&|\{z\mid(x,z)\in E\Gamma,(y,z)\in A\vec{\Gamma}^{\top}\cup A\vec{\Gamma}\}|\nonumber\\
	=&|\{z\mid(x,z)\in E\Gamma,(z,y)\in A\vec{\Gamma}\}|+|\{z\mid(y,z)\in A\vec{\Gamma},(z,x)\in E\Gamma\}|\nonumber\\
	=&\alpha_1+\beta_1,\nonumber\\
	&|\{z\mid(x,z)\in A\vec{\Gamma}^{\top}\cup A\vec{\Gamma},(y,z)\in E\Gamma\}|\nonumber\\
	=&|\{z\mid(x,z)\in A\vec{\Gamma},(z,y)\in E\Gamma\}|+|\{z\mid(y,z)\in E\Gamma,(z,x)\in A\vec{\Gamma}\}|\nonumber\\
	=&\alpha_1+\beta_1.\nonumber
	\end{align}
	 In view of \eqref{3.3}, one obtains
	\begin{align}\tag{3.2}\label{3.3.0}
	t-|E_{j}(x)|-\alpha_1-\beta_1=t-|E_{i}(y)|-\alpha_1-\beta_1.
	\end{align}

	Suppose  $(x,y)\in E\Gamma$. By Definition \ref{main defin}, one gets
	\begin{align}
	&|\{z\mid(x,z)\in E\Gamma,(y,z)\in A\vec{\Gamma}^{\top}\cup A\vec{\Gamma}\}|\nonumber\\
	=&|\{z\mid(x,z)\in E\Gamma,(z,y)\in A\vec{\Gamma}\}|+|\{z\mid(y,z)\in A\vec{\Gamma},(z,x)\in E\Gamma\}|\nonumber\\
	=&2\gamma_1,\nonumber\\
	&|\{z\mid(x,z)\in A\vec{\Gamma}^{\top}\cup A\vec{\Gamma},(y,z)\in E\Gamma\}|\nonumber\\
	=&|\{z\mid(x,z)\in  A\vec{\Gamma},(z,y)\in E\Gamma\}|+|\{z\mid(y,z)\in E\Gamma,(z,x)\in  A\vec{\Gamma}\}|\nonumber\\
	=&2\gamma_1.\nonumber
	\end{align}
 In view of \eqref{3.3}, one obtains
	\begin{align}\tag{3.3}\label{3.3.1}
	t-|E_{j}(x)|-2\gamma_1=t-|E_{i}(y)|-2\gamma_1.
	\end{align}

	By \eqref{3.3.0} and \eqref{3.3.1}, one has $|E_{j}(x)|=|E_{i}(y)|$. Since $x\in V_i$ was arbitrary, $|E_{j}(x)|$ does not depend on the choice of vertex $x\in V_i$, and so $|E_j(x)|=e_{ij}$. Since $y\in V_j$ was arbitrary, $|E_{i}(y)|$ does not depend on the choice of vertex $y\in V_j$, and so $|E_i(y)|=e_{ji}$. Then $e_{ij}=e_{ji}$. Thus, the desired result follows.
\end{proof}

\begin{lemma}\label{arc types}
Let $x\in V_i$ and $y\in V_j$ for a pair $(i,j)$ with $i\ne j$. The following hold:
\begin{enumerate}
		\item\label{arc types-1} If $(x,y)\in A\vec{\Gamma}$, then $|A_{j}^{+}(x)|-|A_{i}^{+}(y)|=\beta_1+\beta_2-\alpha_1-\alpha_2$;

		\item\label{arc types-2} If $(x,y)\in E\Gamma$, then $|A_{j}^{+}(x)|=|A_{i}^{+}(y)|$.
	\end{enumerate}
\end{lemma}
\begin{proof}
	By Lemma \ref{twopath}, one has
	\begin{align}
	\{z\mid (x,z),(y,z)\in A\vec{\Gamma}\}&=A\vec{\Gamma}(x)\setminus (V_j\cup(A\vec{\Gamma}^\top\cup E\Gamma)(y))\nonumber\\&=A\vec{\Gamma}(y)\setminus (V_i\cup (A\vec{\Gamma}^\top\cup E\Gamma)(x)).\nonumber
	\end{align}
	Since $|A\vec{\Gamma}(x)\backslash V_j|=k-t-|A_{j}^{+}(x)|$ and $|A\vec{\Gamma}(y)\backslash V_i|=k-t-|A_{i}^{+}(y)|$, one gets
	\begin{align}
	&|A_{j}^{+}(x)|+|\{z\mid(x,z)\in A\vec{\Gamma},(y,z)\in A\vec{\Gamma}^\top\cup E\Gamma\}|\nonumber\\
	=&|A_{i}^{+}(y)|+|\{z\mid(x,z)\in A\vec{\Gamma}^\top\cup E\Gamma,(y,z)\in A\vec{\Gamma}\}|.\label{3.4}\tag{3.4}
	\end{align}
	
(i) By Definition \ref{main defin}, we have
\begin{align}
&|\{z\mid(x,z)\in A\vec{\Gamma},(y,z)\in A\vec{\Gamma}^\top\cup E\Gamma\}|\nonumber\\
=&|\{z\mid(x,z)\in A\vec{\Gamma},(z,y)\in E\Gamma\}|+|\{z\mid(x,z)\in A\vec{\Gamma},(z,y)\in A\vec{\Gamma}\}|\nonumber\\
=&\alpha_1+\alpha_2,\nonumber\\
&|\{z\mid(x,z)\in A\vec{\Gamma}^\top\cup E\Gamma,(y,z)\in A\vec{\Gamma}\}|\nonumber\\
=&|\{z\mid(y,z)\in A\vec{\Gamma},(z,x)\in E\Gamma\}|+|\{z\mid(y,z)\in A\vec{\Gamma},(z,x)\in A\vec{\Gamma}\}|\nonumber\\
=&\beta_1+\beta_2.\nonumber
\end{align}
 \eqref{3.4} implies $|A_{j}^{+}(x)|+\alpha_1+\alpha_2=|A_{i}^{+}(y)|+\beta_1+\beta_2$, and so (i) holds.

(ii) By Definition \ref{main defin}, we have
\begin{align}
&|\{z\mid(x,z)\in A\vec{\Gamma},(y,z)\in A\vec{\Gamma}^\top\cup E\Gamma\}|\nonumber\\
=&|\{z\mid(x,z)\in A\vec{\Gamma},(z,y)\in E\Gamma\}|+|\{z\mid(x,z)\in A\vec{\Gamma},(z,y)\in A\vec{\Gamma}\}|\nonumber\\
=&\gamma_1+\gamma_2,\nonumber\\
&|\{z\mid(x,z)\in A\vec{\Gamma}^\top\cup E\Gamma,(y,z)\in A\vec{\Gamma}\}|\nonumber\\
=&|\{z\mid(y,z)\in A\vec{\Gamma},(z,x)\in E\Gamma\}|+|\{z\mid(y,z)\in A\vec{\Gamma},(z,x)\in A\vec{\Gamma}\}|\nonumber\\
=&\gamma_1+\gamma_2.\nonumber
\end{align}
 \eqref{3.4} implies  $|A_{j}^{+}(x)|+\gamma_1+\gamma_2=|A_{i}^{+}(y)|+\gamma_1+\gamma_2$, and so $|A_{j}^{+}(x)|=|A_{i}^{+}(y)|$.
\end{proof}

\begin{lemma}\label{changyong1}
Let   $(i,j)$ be a pair with $i\ne j$.	Suppose $x_1,x_2\in V_i$. If there exists   $y\in V_j$ such that $\{(x_1,y),(x_2,y)\}$ is a subset of $E\Gamma, A\vec{\Gamma}~\text{or}~{A\vec{\Gamma}}^{\top}$, then  $|A_{j}^{+}(x_1)|=|A_{j}^{+}(x_2)|$.
\end{lemma}
\begin{proof}
If $(x_1,y), (x_2,y)\in E\Gamma$, from Lemma \ref{arc types} \ref{arc types-2}, then $|A_{j}^{+}(x_1)|=|A_{i}^{+}(y)|=|A_{j}^{+}(x_2)|$, which implies $|A_{j}^{+}(x_1)|=|A_{j}^{+}(x_2)|$.	If $(x_1,y), (x_2,y)\in A\vec{\Gamma}$, from Lemma \ref{arc types} \ref{arc types-1}, then $|A_{j}^{+}(x_1)|-|A_{i}^{+}(y)|=\beta_1+\beta_2-\alpha_1-\alpha_2=|A_{j}^{+}(x_2)|-|A_{i}^{+}(y)|$, which implies $|A_{j}^{+}(x_1)|=|A_{j}^{+}(x_2)|$. If $(x_1,y), (x_2,y)\in {A\vec{\Gamma}}^{\top}$, from Lemma \ref{arc types} \ref{arc types-1}, then $|A_{i}^{+}(y)|-|A_{j}^{+}(x_1)|=\beta_1+\beta_2-\alpha_1-\alpha_2=|A_{i}^{+}(y)|-|A_{j}^{+}(x_2)|$, which implies $|A_{j}^{+}(x_1)|=|A_{j}^{+}(x_2)|$.
	 Thus, the desired result follows.
\end{proof}

\begin{lemma}\label{changyong002}
Let   $(i,j)$ be a pair with $i\ne j$.	Suppose  $x_1,x_2\in V_i$. If there exists  $y\in V_j$ such that $(x_1,y),(y,x_2)\in A\vec{\Gamma}$, then  $|A_{j}^{+}(x_1)|-|A_{j}^{+}(x_2)|=2(\beta_1+\beta_2-\alpha_1-\alpha_2).$
\end{lemma}
\begin{proof}
	If $(x_1,y), (y,x_2)\in A\vec{\Gamma}$, from Lemma \ref{arc types} \ref{arc types-1}, then $$|A_{j}^{+}(x_1)|-|A_{i}^{+}(y)|=\beta_1+\beta_2-\alpha_1-\alpha_2=|A_{i}^{+}(y)|-|A_{j}^{+}(x_2)|,$$
	which implies $|A_{j}^{+}(x_1)|-|A_{j}^{+}(x_2)|=2(\beta_1+\beta_2-\alpha_1-\alpha_2)$.
\end{proof}


\begin{lemma}\label{pair types}
	Let   $(i,j)$ be a pair with $i\ne j$. Then one of the following holds:
	\begin{enumerate}
		\item\label{pair types-1} $|A_{j}^{+}(x)|=c_{ij}$ for every $x\in V_i$;
		
		\item\label{pair types-3} $V_i$ is partitioned into two nonempty sets $V_i=V_i'\cup V_i''$ so that $(V_i'\times V_j)\cup (V_j\times V_i'')\subseteq A\vec{\Gamma}$;
		
		\item\label{pair types-2} $V_i,V_j$ are partitioned into two nonempty sets $V_i=V_i'\cup V_i''$ and  $V_j=V_j'\cup V_j''$ so that
			  $(V_i'\times V_j')\cup (V_i''\times V_j'')\subseteq E\Gamma$ and $(V_i'\times V_j'')\cup (V_j'\times V_i'')\subseteq A\vec{\Gamma}$, where $|V_j''|=\frac{r}{2}$.
	\end{enumerate}
\end{lemma}
\begin{proof}
	Assume that \ref{pair types-1} does not hold. Then $|A_{j}^{+}(x')|\ne |A_{j}^{+}(x'')|$ for some $x',x''\in V_i$.  If $\beta_1+\beta_2-\alpha_1-\alpha_2=0$, from Lemma \ref{arc types}, then $|A_{j}^{+}(x)|= |A_{i}^{+}(y)|$ for any $x\in V_i$ and $y\in V_j$, which implies $|A_{j}^{+}(x')|=|A_{j}^{+}(x'')|$, a contradiction. Thus,  $\beta_1+\beta_2-\alpha_1-\alpha_2\ne0$.
	By Lemmas \ref{edge same} and \ref{changyong1}, one has $|E_{j}(x')|=|E_{j}(x'')|=e_{ij}$ and  $E_{j}(x')\cap E_{j}(x'')=\emptyset$.
	
	\textbf{Case 1.}  $e_{ij}\ne 0$.

	Suppose that there exist two vertices $y_1,y_2\in E_{j}(x'')$ such that $(y_1,x'),(x',y_2)\in A\vec{\Gamma}$. Since $x'\in V_i$, from Lemma \ref{changyong002}, one obtains $|A_{i}^{+}(y_1)|-|A_{i}^{+}(y_2)|=2(\beta_1+\beta_2-\alpha_1-\alpha_2)$. Since $y_1,y_2\in E_{j}(x'')$, from Lemma \ref{arc types} \ref{arc types-2}, we get $|A_{i}^{+}(y_1)|=|A_{j}^{+}(x'')|=|A_{i}^{+}(y_2)|$, which implies that $\beta_1+\beta_2-\alpha_1-\alpha_2=0$, a contradiction.

Since $E_{j}(x')\cap E_{j}(x'')=\emptyset$, we may assume $E_{j}(x'')\subseteq A_{j}^{+}(x')$. Since $e_{ij}\ne 0$, we can pick a vertex $y_0\in E_{j}(x'')\subseteq A_{j}^{+}(x')$. By Lemma \ref{arc types}, one has $|A_{j}^{+}(x')|-|A_{i}^{+}(y_0)|=\beta_1+\beta_2-\alpha_1-\alpha_2$ and $|A_{j}^{+}(x'')|=|A_{i}^{+}(y_0)|$, which imply that
	\begin{align}
	|A_{j}^{+}(x')|-|A_{j}^{+}(x'')|=\beta_1+\beta_2-\alpha_1-\alpha_2.\label{key-1}\tag{3.5}
	\end{align}

	Suppose that there exists a vertex $y\in E_{j}(x')$ such that $(x'',y)\in A\vec{\Gamma}$. By Lemma \ref{arc types}, one obtains $|A_{j}^{+}(x'')|-|A_{i}^{+}(y)|=\beta_1+\beta_2-\alpha_1-\alpha_2$ and $|A_{j}^{+}(x')|=|A_{i}^{+}(y)|$, which imply $|A_{j}^{+}(x'')|-|A_{j}^{+}(x')|=\beta_1+\beta_2-\alpha_1-\alpha_2$. \eqref{key-1} implies $\beta_1+\beta_2-\alpha_1-\alpha_2=0$, a contradiction. Since $E_{j}(x')\cap E_{j}(x'')=\emptyset$, one has $E_{j}(x')\subseteq
	A_{j}^{-}(x'')$.
	
	Let $V_0=V_j\backslash (E_{j}(x')\cup E_{j}(x''))$. Since $|A_{j}^{+}(x')|\ne |A_{j}^{+}(x'')|$, from Lemma \ref{changyong1}, $V_0$ is partitioned into two sets $V_0=V_0^1\cup V_0^2$~(the sets may be empty)~so that $V_0^1\subseteq A_{j}^{+}(x')\cap A_{j}^{-}(x'')$ and $V_0^2\subseteq A_{j}^{+}(x'')\cap A_{j}^{-}(x')$.
	
	Suppose that there exists a vertex $y\in V_0^1\cup V_0^2$.
	In view of  Lemma \ref{changyong002}, one has $|A_{j}^{+}(x')|-|A_{j}^{+}(x'')|=\pm 2(\beta_1+\beta_2-\alpha_1-\alpha_2)$. \eqref{key-1} implies $\beta_1+\beta_2-\alpha_1-\alpha_2=0$, a contradiction. Thus, $V_0^1=V_0^2=\emptyset$, and so $V_j=E_{j}(x')\cup E_{j}(x'')$. Since $|E_j(x'')|=|E_j(x')|=e_{ij}$ and $E_{j}(x')\cap E_{j}(x'')=\emptyset$, we get $e_{ij}=\frac{r}{2}$. Since $E_{j}(x'')\subseteq A_{j}^{+}(x')$ and $E_{j}(x')\subseteq
	A_{j}^{-}(x'')$, one obtains $A_{j}^{+}(x')=E_{j}(x'')$ and $A_{j}^{-}(x'')=E_{j}(x')$, which imply $|A_{j}^{+}(x')|=e_{ij}=\frac{r}{2}$ and $|A_{j}^{+}(x'')|=0$. In view of \eqref{key-1}, we have $|A_{j}^{+}(x')|-|A_{j}^{+}(x'')|=\beta_1+\beta_2-\alpha_1-\alpha_2=\frac{r}{2}$.

	
	
	 Let $x\in V_i$. Suppose that there exists a vertex $y\in E_j(x')$ such that $(x,y)\in A\vec{\Gamma}$. The fact $E_j(x')=A^-_j(x'')$ implies $(y,x'')\in A\vec{\Gamma}$. Since $|A_{j}^{+}(x'')|=0$, from Lemma \ref{changyong002}, we have $|A_{j}^{+}(x)|=|A_{j}^{+}(x)|-|A_{j}^{+}(x'')|=2(\beta_1+\beta_2-\alpha_1-\alpha_2)=r$, which implies $V_j=A_{j}^{+}(x)$, and so $ E_j(x)=\emptyset$, contrary to the fact that  $|E_j(x)|=e_{ij}=\frac{r}{2}$. Since $x\in V_i$ was arbitrary, we get $E_j(x')\times  V_i\subseteq A\vec{\Gamma}\cup E\Gamma$.
	
	
	Suppose that there exists a vertex $y\in E_j(x'')$ such that $(y,x)\in A\vec{\Gamma}$.
		Since $(y,x)\in A\vec{\Gamma}$ and $(x'',y)\in  E\Gamma$, from Lemma \ref{arc types}, one has
	$\frac{r}{2}=\beta_1+\beta_2-\alpha_1-\alpha_2=|A_{i}^{+}(y)|-|A_{j}^{+}(x)|=|A_{j}^{+}(x'')|-|A_{j}^{+}(x)|=-|A_{j}^{+}(x)|$, a contradiction. Since $x\in V_i$ was arbitrary, we get $V_i\times E_j(x'')\subseteq A\vec{\Gamma}\cup E\Gamma$.
	
Let $V_i'=\{x_0\in V_i\mid A_{j}^{+}(x_0)\cap E_j(x'')\ne \emptyset\}$ and $V_i''=V_i\setminus V_i'$. Since $E_{j}(x'')= A_{j}^{+}(x')$, one has
$x'\in V_i'$ and $x''\in V_i''$, which imply that $V_i'$ and $V_i''$ are both nonempty. Let $x_0\in V_i'$. Then there exists $y\in E_j(x'')$ with $(x_0,y)\in A\vec{\Gamma}$. Since $E_{j}(x'')=A_{j}^{+}(x')$, one has $(x',y)\in A\vec{\Gamma}$. In view of Lemma \ref{changyong1}, we obtain $|A_{j}^{+}(x_0)|=|A_{j}^{+}(x')|=\frac{r}{2}$. If there exists $y_1\in E_j(x'')$ such that $(x_0,y_1)\in E\Gamma$, from Lemma \ref{arc types} \ref{arc types-2}, then $0=|A_{j}^{+}(x'')|=|A_{j}^{+}(x_0)|=\frac{r}{2}$, a contradiction. Since $V_i\times E_j(x'')\subseteq A\vec{\Gamma}\cup E\Gamma$ and $|E_j(x'')|=\frac{r}{2}$, one has $E_j(x'')=A_{j}^{+}(x_0)$. Since $V_j=E_j(x')\cup E_j(x'')$ and $e_{ij}=\frac{r}{2}$, we have $E_j(x')=E_j(x_0)$. Since $x_0\in V_i'$ was arbitrary, we obtain $V_i'\times E_j(x'')\subseteq A\vec{\Gamma}$ and $V_i'\times E_j(x')\subseteq E\Gamma$.

Since $V_i\times E_j(x'')\subseteq A\vec{\Gamma}\cup E\Gamma$, we have $V_i''\times E_j(x'')\subseteq E\Gamma$. Since $e_{ij}=\frac{r}{2}$, one gets $E_j(x_0')=E_j(x'')$ for all $x_0'\in V_i''$. Since $E_j(x')\times  V_i\subseteq A\vec{\Gamma}\cup E\Gamma$, we obtain $E_j(x')\times V_i''\subseteq A\vec{\Gamma}$. Since $|E_j(x'')|=\frac{r}{2}$, \ref{pair types-2} holds.

	\textbf{Case 2.}  $e_{ij}=0$.
	
	Since $|A_{j}^{+}(x')|\ne |A_{j}^{+}(x'')|$, from Lemma \ref{changyong1}, $V_j$ is partitioned into two sets $V_j=V_j'\cup V_j''$~(one of the sets may be empty)~so that $V_j'=A_{j}^{+}(x')\cap A_{j}^{-}(x'')$ and $V_j''=A_{j}^{+}(x'')\cap A_{j}^{-}(x')$.

\textbf{Case 2.1}  $V_j'\ne \emptyset$ and $ V_j''\ne \emptyset$.

Let $V_i'=\{x\in V_i\mid A_{j}^{+}(x)\cap V_j'\ne \emptyset\}$ and $V_i''=V_i\setminus V_i'$. Then $x'\in V_i'$ and $x''\in V_i''$. Pick a vertex $x\in V_i'$. Then there exists $y\in A_{j}^{+}(x)\cap V_j'$.
	Since $(x',y)\in  A\vec{\Gamma}$, from Lemma \ref{changyong1},  one gets $ |A_{j}^{+}(x)|=|A_{j}^{+}(x')|=| V_j'|$.
If $A_{j}^{+}(x)\cap V_j''\ne \emptyset$,  from Lemma \ref{changyong1}, then $|A_{j}^{+}(x')|=|A_{j}^{+}(x)|= |A_{j}^{+}(x'')|$ since $V_j''\subseteq A_{j}^{+}(x'')$, contrary to the fact that  $|A_{j}^{+}(x')|\ne |A_{j}^{+}(x'')|$. Thus, $A_{j}^{+}(x)\cap V_j''=\emptyset$, and so $A_{j}^{+}(x)= V_j'$. Since $e_{ij}=0$ and $x\in V_i'$ was arbitrary, we get  $(V_i'\times V_j')\cup (V_j''\times V_i')\subseteq A\vec{\Gamma}$.

Let $z\in V_i''$. Then  $A_{j}^{+}(z)\cap V_j'=\emptyset$. Suppose $A_{j}^{+}(z)\cap V_j''=\emptyset$. It follows that $V_j\subseteq A_{j}^{-}(z)$. Since $V_j'\subseteq A_{j}^{+}(x')\cap A_{j}^{-}(z)$ and $V_j''\subseteq A_{j}^{+}(x'')\cap A_{j}^{-}(z)$, from Lemma \ref{changyong002}, we get $|A_{j}^{+}(x')|-|A_{j}^{+}(z)|=2(\beta_1+\beta_2-\alpha_1-\alpha_2)$ and $|A_{j}^{+}(x'')|-|A_{j}^{+}(z)|=2(\beta_1+\beta_2-\alpha_1-\alpha_2)$, which imply that $|A_{j}^{+}(x')|=|A_{j}^{+}(x'')|$, a contradiction. Thus, $A_{j}^{+}(z)\cap V_j''\ne \emptyset$.
Since $V_j''\subseteq A_{j}^{+}(x'')$, from Lemma \ref{changyong1}, one has  $ |A_{j}^{+}(z)|=|A_{j}^{+}(x'')|=| V_j''|$. Since $A_{j}^{+}(z)\cap V_j'=\emptyset$, we get $A_{j}^{+}(z)= V_j''$. Since $e_{ij}=0$ and  $z\in V_i''$ was arbitrary, one obtains $(V_i''\times V_j'')\cup (V_j'\times V_i'')\subseteq A\vec{\Gamma}$.

Let $y'\in V_j'$ and $y''\in V_j''$. Since $(x'',y''),(y',x'')\in A\vec{\Gamma}$, from Lemma \ref{arc types} \ref{arc types-1}, one obtains
\begin{align}
|A_{j}^{+}(x'')|-|A_{i}^{+}(y'')|=\beta_1+\beta_2-\alpha_1-\alpha_2=|A_{i}^{+}(y')|-|A_{j}^{+}(x'')|.\tag{3.6}\label{lem3.7-1}
\end{align}
 Note that $x''\in V_i''$. Since $(V_i'\times V_j')\cup (V_j''\times V_i')\subseteq A\vec{\Gamma}$ and $(V_i''\times V_j'')\cup (V_j'\times V_i'')\subseteq A\vec{\Gamma}$, we have $|A_{j}^{+}(x'')|=|V_j''|$, $|A_{i}^{+}(y'')|=|V_i'|$ and $|A_{i}^{+}(y')|=|V_i''|$, which imply  $|V_j''|-|V_i'|=|V_i''|-|V_j''|$ from \eqref{lem3.7-1}. Since $|V_j''|+|V_j'|=|V_i''|+|V_i'|=r$, one gets $|V_j''|=|V_j'|=\frac{r}{2}$. Since $x'\in V_i'$, we have $|A_{j}^{+}(x'')|=|V_j''|=|V_j'|=|A_{j}^{+}(x')|$, a contradiction.

	\textbf{Case 2.2}  $V_j'=\emptyset$ or $V_j''=\emptyset$.
	
	Without loss of generality, we may assume that $V_j''=\emptyset$, and so $V_j=V_j'$. Let $V_i'=\{x\in V_i\mid A_{j}^{+}(x)\ne \emptyset\}$ and $V_i''=V_i\setminus V_i'$. Since $x'\in V_i'$ and $x''\in V_i''$, $V_i'$ and $V_i''$ are both nonempty. Let $x\in V_i$. Then there exists a vertex $y\in V_j$ such that $(x,y)\in A\vec{\Gamma}$. Since $V_j=V_j'$, we have $(x',y)\in A\vec{\Gamma}$. Lemma \ref{changyong1} implies $ |A_{j}^{+}(x)|=|A_{j}^{+}(x')|=r$. Then $A_{j}^{+}(x)=V_j$. Since $x\in V_i$ was arbitrary, we get $V_i'\times  V_j\subseteq A\vec{\Gamma}$. Since $e_{ij}=0$, one gets $V_j\times  V_i''\subseteq A\vec{\Gamma}$.  Thus, \ref{pair types-3} holds.
\end{proof}

The following theorem characterizes doubly regular team semicomplete multipartite digraphs.

\begin{thm}\label{main thm}
A doubly regular team semicomplete multipartite
digraph $\Gamma$ has Type I or  II or  III.




\end{thm}
\begin{proof}
At first, we will compute the value of $\beta_1+\beta_2-\alpha_1-\alpha_2$ case by case. Pick a pair $(i,j)$ with $i\ne j$ such that $(V_i\times V_j)\cap (A\vec{\Gamma}\cup A\vec{\Gamma}^{\top})\ne \emptyset$.

\textbf{Case 1.} $|A_{j}^{+}(x)|=c_{ij}$ for every $x\in V_i$ and $|A_{i}^{+}(y)|=c_{ji}$ for every $y\in V_j$.

Since $(V_i\times V_j)\cap (A\vec{\Gamma}\cup A\vec{\Gamma}^{\top})\ne \emptyset$, we have $c_{ij}\neq0$ or $c_{ji}\neq0$. Since $\Gamma$ is semicomplete multipartite, one gets
	 \begin{align}
	 |V_i\times V_j|&=|A\vec{\Gamma}\cap(V_i\times V_j)|+|A\vec{\Gamma}^{\top}\cap(V_i\times V_j)|+|E\Gamma\cap (V_i\times V_j)|\nonumber\\
	 &=\sum_{x\in V_i}(|A_{j}^{+}(x)|+|E_j(x)|)+\sum_{y\in V_j}|A_i^+(y)|\nonumber,
	 \end{align}
	 which implies that $r^2=rc_{ij}+re_{ij}+rc_{ji}$ from Lemma \ref{edge same}. Then $r=c_{ij}+c_{ji}+e_{ij}$.

\textbf{Case 1.1} $c_{ij}c_{ji}=0$.
	
Since the proof of two cases are similar, we may assume that $c_{ij}=0$. Then $c_{ji}\neq0$. If there exist $x\in V_i$ and $y\in V_j$ such that $(x,y)\in E\Gamma$, from Lemma \ref{arc types} \ref{arc types-2}, then $|A_{j}^{+}(x)|=|A_{i}^{+}(y)|$, which implies that $0=c_{ij}=c_{ji}$, a contradiction. Thus, $e_{ij}=0$.
Since $r=c_{ij}+c_{ji}+e_{ij}$, we obtain $c_{ji}=r$. Then $V_j\times V_i\subseteq A\vec{\Gamma}$. Lemma \ref{arc types} \ref{arc types-1} implies $\beta_1+\beta_2-\alpha_1-\alpha_2=|A_i^+(y)|-|A_j^+(x)|=c_{ji}-c_{ij}=r$ with $y\in V_j$ and $x\in V_i$.
	
\textbf{Case 1.2} $c_{ij}c_{ji}\neq0$.
	
Pick $x\in V_i$. Since $c_{ij}c_{ji}\neq0$, there exist $y\in V_j$ and $x'\in V_i$ such that $(x,y),(y,x')\in A\vec{\Gamma}$. By Lemma \ref{arc types} \ref{arc types-1}, we get  $c_{ij}-c_{ji}=|A_j^+(x)|-|A_i^+(y)|=\beta_1+\beta_2-\alpha_1-\alpha_2$ and $c_{ji}-c_{ij}=|A_i^+(y)|-|A_j^+(x')|=\beta_1+\beta_2-\alpha_1-\alpha_2$,
which imply $\beta_1+\beta_2-\alpha_1-\alpha_2=0$ and $c_{ij}=c_{ji}$. Since $r=c_{ij}+c_{ji}+e_{ij}$, one has $c_{ij}=c_{ji}=\frac{r-e_{ij}}{2}$.

\textbf{Case 2.} $|A_{j}^{+}(x)|\ne |A_{j}^{+}(x')|$ for some $x,x'\in V_i$, or $|A_{i}^{+}(y)|\ne |A_{i}^{+}(y')|$ for some $y,y'\in V_j$.

Without loss of generality, we may assume that $|A_{j}^{+}(x)|\ne |A_{j}^{+}(x')|$ for some $x,x'\in V_i$. Note that \ref{pair types-3} or \ref{pair types-2} of Lemma \ref{pair types} is satisfied. Suppose that \ref{pair types-3}  of Lemma \ref{pair types} is satisfied. Let $y\in V_j$. Then there exist $x_0,x_1\in V_i$ such that $(x_0,y),(y,x_1)\in A\vec{\Gamma}$.  By Lemma \ref{changyong002}, we have $|A_{j}^{+}(x_0)|-|A_{j}^{+}(x_1)|=2(\beta_1+\beta_2-\alpha_1-\alpha_2).$  By Lemma \ref{pair types} \ref{pair types-3}, we have $A_{j}^+(x_0)=V_j$ and $A_{j}^+(x_1)=\emptyset$, which imply $\beta_1+\beta_2-\alpha_1-\alpha_2=\frac{r}{2}$.
Suppose that \ref{pair types-2} of Lemma \ref{pair types} is satisfied. Then there exist $x_0\in V_i'$ and $y\in V_j''$ such that $(x_0,y)\in A\vec{\Gamma}$. By Lemma \ref{arc types} \ref{arc types-1}, one has $|A_{j}^{+}(x_0)|-|A_{i}^{+}(y)|=\beta_1+\beta_2-\alpha_1-\alpha_2.$ By Lemma \ref{pair types} \ref{pair types-2}, we get $|A_{j}^{+}(x_0)|=\frac{r}{2}$ and $|A_{i}^{+}(y)|=0$, which imply $\beta_1+\beta_2-\alpha_1-\alpha_2=\frac{r}{2}$.

We conclude that $\beta_1+\beta_2-\alpha_1-\alpha_2\in\{0,\frac{r}{2},r\}$.  Note that the value of $\beta_1+\beta_2-\alpha_1-\alpha_2$ is independent of the choice of a pair $(i,j)$ with $i\ne j$ satisfying $(V_i\times V_j)\cap (A\vec{\Gamma}\cup A\vec{\Gamma}^{\top})\ne \emptyset$. If $\beta_1+\beta_2-\alpha_1-\alpha_2=r$, then $V_i\times V_j\subseteq E\Gamma$ or $(i,j)$ satisfies Case 1.1 for each pair $(i,j)$ with $i\neq j$, which implies $V_i\times V_j\subseteq E\Gamma$, $V_i\times V_j\subseteq A\vec{\Gamma}$ or $V_j\times V_i\subseteq A\vec{\Gamma}$, and so $\Gamma$ has Type I since $\Gamma$ is semicomplete multipartite. If $\beta_1+\beta_2-\alpha_1-\alpha_2=0$, then $V_i\times V_j\subseteq E\Gamma$ or $(i,j)$ satisfies Case 1.2 for each pair $(i,j)$ with $i\neq j$, which implies that $\Gamma$ has Type II. If $\beta_1+\beta_2-\alpha_1-\alpha_2=\frac{r}{2}$, then $V_i\times V_j\subseteq E\Gamma$ or $(i,j)$ satisfies Case 2 for each pair $(i,j)$ with $i\neq j$, which implies that $\Gamma$ has Type III.
\end{proof}

\begin{prop}\label{main prop doubly}
	Let $(X,\{R_0,R_1,R_2,R_3,R_4\})$ be a non-symmetric impritive $4$-class association scheme with exactly one pair of non-symmetric relations $R_1^\top=R_2$. Suppose that the digraph $(X,R_1\cup  R_3)$ is strongly connected and the graph $(X,R_4)$ is not connected. Then $(X,R_1\cup  R_3)$ is a doubly regular $(m,r)$-team semicomplete multipartite digraph of Type I or II for some $m,r\in \mathbb{N}$.
\end{prop}
\begin{proof}
Since the graph $(X,R_4)$ is not connected and $R_4$ is symmetric, $(X,R_1\cup  R_3)$ is an $(m,r)$-team semicomplete multipartite digraph with $r=k_4+1$ and $m=\frac{|X|}{r}$.  Note that $$A_iA_j=p_{ij}^{0}I+p_{ij}^{1}A_1+p_{ij}^{2} A_2+p_{ij}^{3}A_3+p_{ij}^{4}A_4$$ for $(i,j)\in \{(1,1),(1,3),(3,3)\}$.
Since $A_4=J-I-A_1-A_2-A_3$ and $A_2=A_1^{\top}$, $(X,R_1\cup  R_3)$ is a doubly regular team semicomplete multipartite digraph.

By Theorem \ref{main thm}, it suffices to show that $(X,R_1\cup  R_3)$ does not have Type III. Assume the contrary, namely, $(X,R_1\cup R_3)$ has Type III. Let $V_1,V_2,\ldots,V_m$ be all partite sets of $(X,R_1\cup R_3)$.

Suppose that there exists a pair $(i,j)$ such that $V_i$ is partitioned into two nonempty sets $V_i=V_i'\cup V_i''$ so that $(V_i'\times V_j)\cup (V_j\times V_i'')\subseteq R_1$. For all $(x,y)\in R_1$ with $x\in V_i$ and $y\in V_j$,  since $y,z$ are in the same partite set if and only if $(y,z)\in R_0\cup R_4$, we have
\begin{align}
|\{z\in V_j\mid(x,z)\in R_1\}|= |\{z\mid (x,z)\in R_1,(z,y)\in R_0\cup R_4\}|=p_{10}^{1}+p_{14}^{1}.\label{3.10}\tag{3.7}
\end{align}
 Pick $x'\in V_i'$, $x''\in V_i''$ and $y'\in V_j$. Since $(V_i'\times V_j)\cup (V_j\times V_i'')\subseteq R_1$, we have $(x',y'), (y',x'')\in R_1$, $\{z\in V_j\mid (x',z)\in R_1\}=V_j$ and $\{z\in V_i\mid (y',z)\in R_1\}=V_i''$. Since vertices $u,v$ are in the same partite set if and only if $(u,v)\in R_0\cup R_4$, one gets $\{z\mid (x',z)\in R_1,(z,y')\in R_0\cup R_4\}=\{z\in V_j\mid (x',z)\in R_1\}=V_j$ and $\{z\mid (y',z)\in R_1,(z,x'')\in R_0\cup R_4\}=\{z\in V_i\mid (y',z)\in R_1\}=V_i''$, contrary to \eqref{3.10}.

	Suppose that there exists a pair $(i,j)$ such that $V_i,V_j$ are partitioned into two nonempty sets $V_i=V_i'\cup V_i''$ and  $V_j=V_j'\cup V_j''$ so that
		$(V_i'\times V_j')\cup (V_i''\times V_j'')\subseteq R_3$ and $(V_i'\times V_j'')\cup (V_j'\times V_i'')\subseteq R_1$. For all $(x,y)\in R_3$ with $x\in V_i$ and $y\in V_j,$ since  $y,z$ are in the same partite set if and only if $(y,z)\in R_0\cup R_4$, we have
\begin{align}\label{3.11}\tag{3.8}
|\{z\in V_j\mid (x,z)\in R_1\}|= |\{z\mid (x,z)\in R_1,(z,y)\in R_0\cup R_4\}|=p_{14}^{3}.
\end{align}
 Pick $x'\in V_i'$, $x''\in V_i''$, $y'\in V_j'$ and $y''\in V_j''$. Since $(V_i'\times V_j')\cup (V_i''\times V_j'')\subseteq R_3$, one obtains $(x',y'),(y'',x'')\in R_3$.
Since $(V_i'\times V_j'')\cup (V_j'\times V_i'')\subseteq R_1$, we get $\{z\in V_j\mid (x',z)\in R_1\}=V_j''$ and $\{z\in V_i\mid (y'',z)\in R_1\}=\emptyset$. Since vertices $u,v$ are in the same partite set if and only if $(u,v)\in R_0\cup R_4$, one gets $\{z\mid (x',z)\in R_1,(z,y')\in R_0\cup R_4\}=\{z\in V_j\mid (x',z)\in R_1\}=V_j''$ and $\{z\mid (y'',z)\in R_1,(z,x'')\in R_0\cup R_4\}=\{z\in V_i\mid (y'',z)\in R_1\}=\emptyset$, contrary to \eqref{3.11}.

By Theorem \ref{main thm}, we have $V_i\times V_j\subseteq E\Gamma$ for any pair $(i,j)$, a contradiction.

\end{proof}

\section{General results}

Let $\Gamma$ be a weakly distance-regular digraph.
For each $\wz{i}:=(a,b)\in\wz{\partial}(\Gamma)$,
we write $k_{a,b}$ instead of $k_{(a,b)}$.

In the remainder of this section, $\Gamma$ always denotes a semicomplete multipartite commutative weakly distance-regular digraph, we present some notation, concepts, and basic results for $\Gamma$ which we shall use in the remainder of the paper.

\begin{lemma}\label{jichu}
The underlying graph of $\Gamma$ is isomorphic to $K^{k}_{m}$ with $m\geq2$ and $k\geq 2$.
\end{lemma}
\begin{proof}
Let $V_1$ and $V_2$ be two distinct partite sets of the underlying graph of $\Gamma$. Pick a vertex $x_i\in V_i$ for each $i\in\{1,2\}$. It follows that $|V_i|=|V\Gamma|-|N^+(x_i)|-|N^-(x_i)|+k_{1,1}$ for $i\in \{1,2\}$. By the weakly distance-regularity of $\Gamma$, we have $|V_1|=|V_2|$. Since $V_1$ and $V_2$ were arbitrary, the desired result follows.
\end{proof}

\begin{lemma}\label{jichu0}
Let $x,y\in V\Gamma$. Then $y\notin N^+(x)\cup N^-(x)$ if and only if $x,y$ belong to the same partite set.
\end{lemma}
\begin{proof}
By Lemma \ref{jichu}, the desired result holds.
\end{proof}

\begin{lemma}\label{sxjl}
	Let $x,y$ be distinct vertices. Then either $\tilde{\partial}(x,y)\in \{(1,q),(q,1)\mid q+1\in T\}$ or there exists a vertex $z\in P_{\wz{i},\wz{j}}(x,y)$ with $\wz{i},\wz{j}\in \{(1,q),(q,1)\mid q+1\in T\}$.
\end{lemma}
\begin{proof}
	 If $x,y$ belong to distinct partite sets,  then $\tilde{\partial}(x,y)\in \{(1,q),(q,1)\mid q+1\in T\}$ by Lemma \ref{jichu0}. If $x,y$ belong to the same partite set, from Lemma \ref{jichu}, then there exists a vertex $z\in P_{\wz{i},\wz{j}}(x,y)$ with $\wz{i},\wz{j}\in \{(1,q),(q,1)\mid q+1\in T\}$. Thus, the desired result follows.
\end{proof}

\begin{lemma}\label{pssize}
	Let $x\in V\Gamma$ and $V$ be a partite set with $x\notin V$. Then
	\begin{align}
	|V|=\sum_{q\in T}|(\Gamma_{1,q-1}(x)\cup\Gamma_{q-1,1}(x))\cap V|.\nonumber
	\end{align}
\end{lemma}
\begin{proof}
	Pick a vertex $y\in V$. Since $x\notin V$, from Lemma \ref{jichu0}, one has $(x,y)\in \Gamma_{1,q-1}\cup\Gamma_{q-1,1}$ for some $q\in T$. Thus, the desired result holds.
\end{proof}

The commutativity of $\Gamma$ will be used frequently in the sequel, so we no longer refer to it for the sake of simplicity.

\begin{lemma}\label{jiaojik}
	Let $x\in V\Gamma$, $\wz{i}\in \{(1,p-1),(p-1,1)\}$ and $\wz{j}\in \{(1,q-1),(q-1,1)\}$ with $p,q\in T$. Suppose $\{\wz{i},\wz{j}\}\ne \{(1,1)\}$ and $p_{\wz{i},\wz{j}}^{(h,l)}=0$ for all $h,l>1$. If $\Gamma_{\wz{i}}(x) \cap V\ne \emptyset$ for some partite set $V$, then $\Gamma_{\wz{j}^{*}}(x) \cap V=\emptyset$.
\end{lemma}
\begin{proof}
	Let $z\in \Gamma_{\wz{i}}(x) \cap V$. If $y\in \Gamma_{\wz{j}^{*}}(x) \cap V$, from $p_{\wz{i},\wz{j}}^{(h,l)}=0$ for all $h,l>1$, then $(y,z)$ or $(z,y)\in A\Gamma$ since $x\in P_{\wz{j},\wz{i}}(y,z)$ with $\{\wz{i},\wz{j}\}\ne \{(1,1)\}$, contrary to Lemma \ref{jichu0}.
\end{proof}

\begin{lemma}\label{jichu2}
Let $y\notin N^+(x)\cup N^-(x)$ with $x,y\in V\Gamma$. 
Then $w\in N^+(x)\cup N^-(x)$ if and only if $w\in N^+(y)\cup N^-(y)$. Moreover, if $\partial(x,y)\geq3$, then $N^+(x)=N^+(y)$ and $N^-(x)=N^-(y)$.
\end{lemma}
\begin{proof}
By Lemma \ref{jichu0}, $x,y$ belong to the same partite set. Since $\Gamma$ is semicomplete multipartite, $w\in N^+(x)\cup N^-(x)$ if and only if $x,w$ belong to the distinct partite sets, if and only if $y,w$ belong to the distinct partite sets. Since $y,w$ belong to the distinct partite sets if and only if $w\in N^+(y)\cup N^-(y)$, the first statement is valid.

Now suppose $\partial(x,y)\geq 3$. Pick vertices  $u\in  N^{+}(x)$ and $v\in N^{-}(y)$. By the first statement, we have $u\in  N^{+}(y)\cup N^{-}(y)$ and $v\in N^{+}(x)\cup N^{-}(x)$. The fact $3\leq \partial(x,y)\leq\min\{1+\partial(u,y),1+\partial(x,v)\}$ implies $u\in N^{+}(y)$ and $v\in N^{-}(x)$. Since $u\in  N^{+}(x)$ and $v\in N^{-}(y)$ were arbitrary, one has $N^{+}(x)\subseteq N^{+}(y)$ and $N^{-}(y)\subseteq N^{-}(x)$. The fact $|N^{+}(y)|=|N^{+}(x)|=|N^{-}(x)|=|N^{-}(y)|$ implies $N^+(x)=N^+(y)$ and $N^-(x)=N^-(y)$. Thus, the second statement is also valid.
\end{proof}

\begin{lemma}\label{changyong}
If $(2,s)\in\tilde{\partial}(\Gamma)$, then $s\in\{1,2\}$.
\end{lemma}
\begin{proof}
Assume the contrary, namely, $s\geq 3$. Since $(2,s)\in \tilde{\partial}(\Gamma)$, we have  $\Gamma_{2,s}\in \Gamma_{1,q}\Gamma_{1,p}$ for some $q+1,p+1\in T$. Let $(x,z)\in\Gamma_{2,s}$ and $y\in P_{(1,q),(1,p)}(x,z)$.  Then $z\notin N^{+}(x)\cup N^{-}(x)$. Since $\partial(z,x)\geq3$, from Lemma \ref{jichu2}, one gets $N^{-}(x)=N^{-}(z)$ and $N^{+}(z)=N^{+}(x)$. Since $y\in N^-(z)$ and $y\in N^+(x)$, we have $y\in N^{-}(x)$ and $y\in N^{+}(z)$. It follows that $(x,y),(y,z)\in\Gamma_{1,1}$, contrary to the fact that $\partial(z,x)\geq 3$.
\end{proof}

\begin{lemma}\label{mixed arcs}
If $q\in T$, then $q\leq 4$.
\end{lemma}
\begin{proof}
 Let $(x,y)\in\Gamma_{1,q-1}$. If $\partial(x,z)=1$ for all $z\in N^+(y)$, then $N^+(y)\cup\{y\}\subseteq N^+(x)$, contrary to the fact that $\Gamma$ is weakly distance-regular. Then there exists $z\in N^+(y)$ such that $\partial(x,z)=2$. By Lemma \ref{changyong}, we have $\partial(z,x)=1$ or $2$. Since $q-1=\partial(y,x)\leq \partial(y,z)+\partial(z,x)\leq 3$, one has $q\leq 4$.
\end{proof}

\begin{lemma}\label{zuidaarc}
	If $2\in  T$, then $(s,t)\notin \tilde{\partial}(\Gamma)$ for  all $s,t\geq 3$.
\end{lemma}
\begin{proof}
	Suppose for the contrary that $(s,t)\in \tilde{\partial}(\Gamma)$ for  some $s,t\geq 3$. Let $(x,z)\in\Gamma_{s,t}$. Then $z\notin N^{+}(x)\cup N^{-}(x)$. Pick a vertex $w\in\Gamma_{1,1}(x)$. By Lemma \ref{jichu2}, we obtain $(z,w)\in A\Gamma$ or $(w,z)\in A\Gamma$, which implies that $\partial(z,x)=t\leq2$ or $\partial(x,z)=s\leq2$, a contradiction.
\end{proof}


\begin{lemma}\label{jichu3}
If $i,j,s,t\geq 2$, then $\Gamma_{1,q}\notin \Gamma_{i,j}\Gamma_{s,t}$ for all $q+1\in T$.
\end{lemma}
\begin{proof}
Assume the contrary, namely, $\Gamma_{1,q}\in \Gamma_{i,j}\Gamma_{s,t}$  for some $q+1\in T$. Let $(x,z)\in \Gamma_{1,q}$ and $y\in P_{(i,j),(s,t)}(x,z)$.  Since $y\notin N^{+}(x)\cup N^{-}(x)$ and $z\in N^{+}(x)\cup N^{-}(x)$, from Lemma \ref{jichu2}, we obtain $z\in N^{+}(y)\cup N^{-}(y)$, contrary to the fact that $(z,y)\in\Gamma_{s,t}$ with $s,t\geq 2$.
\end{proof}

\begin{lemma}\label{tf22}
	Let $q\geq 1$. If $\Gamma_{2,2}\in \Gamma_{1,1}\Gamma_{1,q}$, then $\Gamma_{2,2}\in \Gamma_{1,q}^2$.
\end{lemma}
\begin{proof}
	If $q=1$, then the desired result follows. Now consider the case that $q>1$. Let $(x,y)\in \Gamma_{1,1}$. By Lemma \ref{jichu0},  $x,y$ belong to two distinct partite sets. Since $\Gamma_{2,2}\in \Gamma_{1,1}\Gamma_{1,q}$,  from Lemma \ref{jb} \ref{jb-2}, one has  $p_{(2,2),(q,1)}^{(1,1)}=p_{(2,2),(1,q)}^{(1,1)}\ne 0$, which implies that there exist  distinct vertices $z_1,z_2$ such that $z_1\in P_{(2,2),(q,1)}(x,y)$ and $z_2\in P_{(2,2),(1,q)}(x,y)$. Since $z_1,z_2\in \Gamma_{2,2}(x)$, from Lemma \ref{jichu0},  $x,z_1,z_2$ belong to the same partite set. Since $y\in P_{(1,q),(1,q)}(z_2,z_1)$, by   Lemma \ref{changyong}, one gets $(z_1,z_2)\in \Gamma_{2,2}$, and so  $\Gamma_{2,2}\in \Gamma_{1,q}^2$.
\end{proof}

\begin{lemma}\label{one}
	If $q\in T$, then $\Gamma_{1,q-1}^2\ne \{\Gamma_{1,q-1}\}$.
\end{lemma}
\begin{proof}
	 Suppose for the contrary that  $\Gamma_{1,q-1}^2=\{\Gamma_{1,q-1}\}$. Let $(x,y)\in \Gamma_{1,q-1}$. Since  $\Gamma_{1,q-1}^2=\{\Gamma_{1,q-1}\}$, we have $(x,z)\in \Gamma_{1,q-1}$  for all  $z\in \Gamma_{1,q-1}(y)$. Hence, $\Gamma_{1,q-1}(y)\cup \{y\} \subseteq \Gamma_{1,q-1}(x)$, a contradiction.
\end{proof}

\begin{lemma}\label{changyong2}
Let $s\geq3$. If $\tilde{\partial}(\Gamma)\backslash\{(1,q),(q,1)\mid q+1\in T\}\subseteq\{(0,0),(2,2),(s,s)\}$, then $\Gamma_{s,s}^2\subseteq\{\Gamma_{0,0},\Gamma_{s,s}\}$.
\end{lemma}
\begin{proof}
By Lemma \ref{jichu3}, we have $\Gamma_{s,s}^2\subseteq\{\Gamma_{0,0},\Gamma_{2,2},\Gamma_{s,s}\}$. It suffices to show that $\Gamma_{2,2}\notin\Gamma_{s,s}^2$. Suppose for the contrary that $\Gamma_{2,2}\in\Gamma_{s,s}^2$. Let $(x,z)\in\Gamma_{2,2}$ and $y\in P_{(s,s),(s,s)}(x,z)$. Then $y\notin N^+(x)\cup N^-(x)$. Pick a path $(x,w,z)$. Since $\partial(x,y)\geq 3$, from Lemma \ref{jichu2}, we obtain $w\in N^{+}(x)=N^{+}(y)$, which implies that $(y,w,z)$ is a path. It follows that $s=\partial(y,z)\leq 2$, a contradiction.
\end{proof}

For a nonempty closed subset $F$ of $\{\Gamma_{\wz i}\}_{\wz i\in \tilde{\partial}(\Gamma)}$,
let
$$V\Gamma/F :=\{F(x)\mid x\in V\Gamma\} ~{\rm and}~ \Gamma_{\tilde{i}}^F :=\{(F(x),F(y))\mid y\in F\Gamma_{\tilde{i}}F(x)\},$$
where $F(x):=\{y\in V\Gamma\mid (x,y)\in \underset{f\in F}{\cup}f\}$. The digraph $(V\Gamma/F,\underset{(1,s)\in \tilde{\partial}(\Gamma)}{\cup}\Gamma_{1,s}^F)$ is said to be the  {\em quotient digraph} of $\Gamma$ over $F$, denoted by $\Gamma/F$.

\begin{lemma}\label{tdxw1}
Let  $s\geq3$.	If $\tilde{\partial}(\Gamma)\backslash\{(1,q),(q,1)\mid q+1\in T\}\subseteq\{(0,0),(2,2),(s,s)\}$,
then $\Gamma\simeq \Sigma\circ \overline{{K}}_{k_{s,s}+1}$, where $\Sigma$ is a semicomplete multipartite commutative weakly distance-regular digraph with $\tilde{\partial}(\Sigma)= \tilde{\partial}(\Gamma)\backslash \{(s,s)\}$. 
\end{lemma}
\begin{proof}
	Let $F=\{\Gamma_{0,0},\Gamma_{s,s}\}$. By Lemma \ref{changyong2}, 	$F$ is closed. We claim that $(x',y')\in A\Gamma$ for $x'\in F(x)$ and $y'\in F(y)$ when $(x,y)\in A\Gamma$. Without loss of generality, we may assume $x'\neq x$. Since $F=\{\Gamma_{0,0},\Gamma_{s,s}\}$, we have $(x,x')\in\Gamma_{s,s}$. It follows that $x'\notin N^+(x)\cup  N^-(x)$. Since $\partial(x,x')=s\geq3$, from Lemma \ref{jichu2}, we have $y\in N^+(x)=N^+(x')$.  If $y=y'$, then our claim is valid. If $y\neq y'$, then $(y,y')\in\Gamma_{s,s}$, and so $x'\in N^-(y)=N^-(y')$ by Lemma \ref{jichu2} again. Thus, our claim is valid.

Since $F=\{\Gamma_{0,0},\Gamma_{s,s}\}$, one gets $|F(w)|=k_{s,s}+1$ for all  $w\in V\Gamma$. Denote $\Sigma:=\Gamma/F$. By the claim,  $\Gamma$ is isomorphic to $\Sigma\circ \overline{{K}}_{k_{s,s}+1}$.

Let $x,y\in V\Gamma$ and  $y\notin F(x)$. Pick a shortest path $(x=x_0,x_1,\ldots,x_l=y)$ from $x$ to $y$ in $\Gamma$. Since $F=\{\Gamma_{0,0},\Gamma_{s,s}\}$,  $(F(x_0),F(x_1),\ldots,F(x_{l}))$ is a path in $\Sigma$, and so $\partial_{\Sigma}(F(x),F(y))\leq\partial_{\Gamma}(x,y)$.

Pick a shortest path $(F(x)=F(y_0),F(y_1),\ldots,F(y_{h})=F(y))$ from $F(x)$ to $F(y)$ in $\Sigma$. It follows that there exists $y_i'\in F(y_i)$ for each $i\in\{0,1,\ldots,h\}$ such that $(y_0',y_1',\ldots,y_h')$ is a path in $\Gamma$. By the claim,  $(x,y_1',\ldots,y_{h-1}',y)$ is a path in $\Gamma$. It follows that $\partial_{\Gamma}(x,y)\leq\partial_{\Sigma}(F(x),F(y))$. Thus, $\partial_{\Gamma}(x,y)=\partial_{\Sigma}(F(x),F(y))$ and $\tilde{\partial}(\Sigma)= \tilde{\partial}(\Gamma)\backslash \{(s,s)\}$.

Let $\wz{h}\in\wz{\partial}(\Sigma)$ and $(F(u),F(v))\in\Sigma_{\wz{h}}$. Then $(u,v)\in \Gamma_{\wz{h}}$. For $\wz{i},\wz{j}\in\wz{\partial}(\Sigma)$, we have
$$P_{\wz{i},\wz{j}}(u,v)=\cup_{w\in P_{\wz{i},\wz{j}}(u,v)}F(w)
=\cup_{F(w)\in \Sigma_{\wz{i}}(F(u))\cap\Sigma_{\wz{j}^*}(F(v))}F(w).$$
Since  $|F(w)|=k_{s,s}+1$ for all  $w\in V\Gamma$, we have $$|\Sigma_{\wz{i}}(F(u))\cap\Sigma_{\wz{j}^*}(F(v))|=\frac{p_{\wz{i},\wz{j}}^{\wz{h}}}{k_{s,s}+1}.$$ Since $\Gamma$ is a commutative weakly distance-regular digraph, $\Sigma$ is a  commutative weakly distance-regular digraph. Since $\Gamma$ is  semicomplete multipartite, from Lemma \ref{jichu0}, $\Sigma$ is  semicomplete multipartite.
\end{proof}

\begin{lemma}\label{le216fj}
	Let $x\in V\Gamma$,  $p_{(1,p-1),(1,q-1)}^{(2,2)}>0$ and  $\{\wz{i},\wz{j}\}= \{(1,p-1),(q-1,1)\}$ with $p,q\in T$. Suppose that there are exactly $t$ partite sets having nonempty intersection with $\Gamma_{\wz{i}}(x)$. Then  $|\Gamma_{\wz{j}}(x)\cap V|=p_{\wz{j},(2,2)}^{\wz{i}}+p_{\wz{j},(0,0)}^{\wz{i}}=\frac{k_{\wz{j}}}{t}$, where $V$ is a partite set with $\Gamma_{\wz{i}}(x)\cap V\ne \emptyset$.
\end{lemma}
\begin{proof}
Let $V_0,V_1,\ldots,V_t$ be distinct partite sets such that  $x\in V_0$ and $\Gamma_{\tilde{i}}(x)\cap V_r\neq\emptyset$ for $1\leq r\leq t$. Pick  a vertex $x_r\in \Gamma_{\wz{i}}(x)\cap V_{r}$ for each $r\in \{1,\ldots,t\}.$

By Lemma \ref{jb} \ref{jb-2} and Lemma \ref{changyong}, one has $\{(2,2)\}=\{(h,l)\mid p_{\tilde{j},(h,l)}^{\tilde{i}}>0~\textrm{with}~h,l>1\}$. In view of Lemma \ref{jichu0}, $y\in P_{\wz{j},(2,2)}(x,x_r)\cup P_{\wz{j},(0,0)}(x,x_r)$ if and only if $y\in \Gamma_{\wz{j}}(x)\cap V_{r}$ for each $r\in \{1,\ldots,t\}$, which  implies
\begin{align}\label{2.11}\tag{4.1}
|\Gamma_{\wz{j}}(x)\cap V_{r}|=p_{\wz{j},(2,2)}^{\wz{i}}+p_{\wz{j},(0,0)}^{\wz{i}}.
\end{align}

Let $v\in\Gamma_{\tilde{j}}(x)\cap V'$ for some partite set $V'$. Since $\wz{j}\in \{(1,p-1),(q-1,1)\}$, we have $V'\neq V_0$. Since $p_{(1,p-1),(1.q-1)}^{(2,2)}>0$, from Lemma \ref{jb} \ref{jb-2}, one has  $p_{(2,2),\wz{i}^*}^{\wz{j}^*}>0$, which implies that  there exists $z\in P_{(2,2),\wz{i}^*}(v,x)$. It follows  from Lemma \ref{jichu0} that $z\in V'$, and so $z\in\Gamma_{\wz{i}}(x)\cap V'$. It follows that $V'\in\{V_1,V_2,\ldots,V_t\}$, and so $v\in\cup_{1\leq r\leq t}(\Gamma_{\wz{j}}(x)\cap V_{r})$. Since $v\in\Gamma_{\tilde{j}}(x)$ was arbitrary, we get $\cup_{1\leq r\leq t}(\Gamma_{\wz{j}}(x)\cap V_{r})=\Gamma_{\tilde{j}}(x)$. The desired result follows from \eqref{2.11}.
\end{proof}

\begin{lemma}\label{main3''}
Let $x\in V\Gamma$ and $\wz{i}\in \{(1,q-1),(q-1,1)\mid q\in T\}$. Suppose $|\Gamma_{\wz{i}}(x)\cap V|\ne \emptyset$ for some partite set $V$ of $\Gamma$. Then
	$\sum_{h,l>1}p_{\wz{i},(h,l)}^{\wz{i}}=|\Gamma_{\wz{i}}(x)\cap V|-1$.
\end{lemma}
\begin{proof}
Pick a vertex $x_1\in \Gamma_{\wz{i}}(x)\cap V$. Since  $$\sum_{h,l>1}p_{\wz{i},(h,l)}^{\wz{i}}=\sum_{h,l>1}|P_{\wz{i},(h,l)}(x,x_1)|,$$ from Lemma \ref{jichu0}, we have $\sum_{h,l>1}|P_{\wz{i},(h,l)}(x,x_1)|+1=|\Gamma_{\wz{i}}(x)\cap V|$. Thus, the desired result holds.
\end{proof}

\section{Possibility of $T$}
In  this section, $\Gamma$ always denotes  a semicomplete multipartite commutative weakly distance-regular digraph. The main result of this section is as follows.

\begin{prop}\label{g234mprop}
Let $T=\{q\mid (1,q-1)\in \tilde{\partial}(\Gamma)\}$. Then
\[T=\{3\},\quad T=\{4\},\quad T=\{3,4\},
\quad or \quad T=\{2,3\}.\]
\end{prop}

In view of Lemma \ref{mixed arcs}, we only need to prove that $T\ne \{2,4\}$ and $T\ne \{2,3,4\}$.
Next we divide this proof into two steps.


\begin{step}
Show that	$T\ne \{2,4\}$.
\end{step}
By way of contradiction, we may assume that $T=\{2,4\}$. Lemmas \ref{changyong} and \ref{zuidaarc} imply that
\begin{align}
\Gamma_{1,q}\Gamma_{1,p}&\subseteq\{\Gamma_{1,3},\Gamma_{2,2}\},\tag{5.1}\label{g24-1}\\
\Gamma_{\tilde{i}}\Gamma_{\tilde{i}^*}&\subseteq \{\Gamma_{0,0},\Gamma_{1,1}, \Gamma_{1,3},\Gamma_{3,1},\Gamma_{2,2}\},\label{g24-2}\tag{5.2}
\end{align}
where  $(p,q)\neq(1,1)$ and $\tilde{i}\in\{(1,3),(3,1)\}$.
Next, we reach a contradiction step by step. Pick a vertex $x\in V\Gamma$. Assume that there are exactly $t$ partite sets having nonempty intersection with $\Gamma_{1,3}(x)$.

\begin{substep}\label{g24gs}
 $p_{(3,1),(2,2)}^{(1,3)}=p_{(1,3),(2,2)}^{(1,3)}+1=\frac{k_{1,3}}{t}$.
\end{substep}
By \eqref{g24-1} and Lemma \ref{one}, one gets $\Gamma_{2,2}\in \Gamma_{1,3}^2$.   Let  $V$ be a partite set with $\Gamma_{1,3}(x)\cap V\ne \emptyset$. By Lemma \ref{le216fj}, one has  $|\Gamma_{3,1}(x)\cap V|=p_{(3,1),(2,2)}^{(1,3)}=\frac{k_{1,3}}{t}$. In view of  \eqref{g24-2}, Lemma \ref{jb} \ref{jb-2} and Lemma \ref{main3''}, we get $p_{(1,3),(2,2)}^{(1,3)}=p_{(3,1),(2,2)}^{(3,1)}=\frac{k_{1,3}}{t}-1$. The desired result follows.

\begin{substep}\label{g24zzjiegou}
	We reach a contradiction.
\end{substep}

We claim that 		$\Gamma_{2,2}\in \Gamma_{1,1}\Gamma_{1,3}$.
Suppose for the contrary that $\Gamma_{2,2}\notin \Gamma_{1,1}\Gamma_{1,3}$.  By  \eqref{g24-1}, one has $ \Gamma_{1,1}\Gamma_{1,3}=\{\Gamma_{1,3}\}$.
Let $(x,y)\in \Gamma_{1,3}$ and $(y,z)\in \Gamma_{1,1}$. Then $(x,z)\in \Gamma_{1,3}$. By Step \ref{g24gs}, one has  $p_{(3,1),(2,2)}^{(1,3)}\ne 0$, which implies that there exists a vertex $w\in P_{(3,1),(2,2)}(x,y)$. It follows  that $w\notin N^+(y)\cup N^-(y)$. Since $(y,z)\in \Gamma_{1,1}$, from Lemma \ref{jichu2}, we have $z\in  N^+(w)\cup N^-(w)$. Since $(w,x,z)$ is a path consisting of arcs of type $(1,3)$, from \eqref{g24-1}, one obtains $(w,z)\in \Gamma_{1,3}$. Since $ \Gamma_{1,1}\Gamma_{1,3}=\{\Gamma_{1,3}\}$ and  $z\in P_{(1,3),(1,1)}(w,y)$, we get $(w,y)\in \Gamma_{1,3}$, a contradiction.
 Thus, our claim is valid.

	By setting $i=(1,1)$ and $l=j=m^*=(1,3)$ in Lemma \ref{jb} \ref{jb-4}, from \eqref{g24-1} and  Lemma \ref{jb} \ref{jb-2}, we get
	$$p_{(1,1),(1,3)}^{(1,3)}p_{(1,3),(3,1)}^{(1,3)}+p_{(1,1),(1,3)}^{(2,2)}p_{(2,2),(3,1)}^{(1,3)}=p_{(1,1),(3,1)}^{(3,1)}p_{(1,3),(3,1)}^{(1,3)}+p_{(1,1),(3,1)}^{(2,2)}p_{(1,3),(2,2)}^{(1,3)},$$
which implies $p_{(1,1),(1,3)}^{(2,2)}p_{(2,2),(3,1)}^{(1,3)}=p_{(1,1),(3,1)}^{(2,2)}p_{(1,3),(2,2)}^{(1,3)}.$
	By the  claim, we get $p_{(1,1),(1,3)}^{(2,2)}=p_{(1,1),(3,1)}^{(2,2)}\ne 0$. It follows from Step \ref{g24gs} that $\frac{k_{1,3}}{t}=p_{(3,1),(2,2)}^{(1,3)}=p_{(1,3),(2,2)}^{(1,3)}=\frac{k_{1,3}}{t}-1$, a contradiction.


\begin{step}
	Show that	 $T\ne\{2,3,4\}$.
\end{step}

By way of contradiction, we may assume that $T=\{2,3,4\}$. Lemmas \ref{changyong} and \ref{zuidaarc} imply that
\begin{align}
	\Gamma_{1,1}^2&\subseteq \{\Gamma_{0,0},\Gamma_{1,1},\Gamma_{1,2},\Gamma_{2,1},\Gamma_{2,2}\},\label{g234-1'}\tag{5.3}\\
		\Gamma_{1,2}^2&\subseteq \{\Gamma_{1,1},\Gamma_{1,2},\Gamma_{2,1},\Gamma_{1,3},\Gamma_{2,2}\},\label{g234-1''}\tag{5.4}\\
		 \Gamma_{1,3}\Gamma_{3,1}&\subseteq \{\Gamma_{0,0},\Gamma_{1,1},\Gamma_{1,2},\Gamma_{2,1},\Gamma_{1,3},\Gamma_{3,1},\Gamma_{2,2}\},\label{g234-4}\tag{5.5}\\
		 \Gamma_{1,3}\Gamma_{2,1}&\subseteq \{ \Gamma_{1,1},\Gamma_{1,2},\Gamma_{2,1},\Gamma_{1,3},\Gamma_{3,1},\Gamma_{2,2}\},\label{g234-3}\tag{5.6}\\
\Gamma_{1,q}\Gamma_{1,3}&\subseteq \{ \Gamma_{1,2},\Gamma_{1,3},\Gamma_{2,2}\},\label{g234-1}\tag{5.7}
	\end{align}
where $q\in \{1,2,3\}$.

Pick a vertex $x\in V\Gamma$. Assume that there are exactly $t$ partite sets having nonempty intersection with $\Gamma_{3,1}(x)$.
Next, we reach a contradiction step by step.



\begin{substep}\label{t234xj2}
Show that	 $\Gamma_{2,2}\in \Gamma_{1,2}\Gamma_{1,3}$.
\end{substep}

Assume the contrary, namely, $\Gamma_{2,2}\notin \Gamma_{1,2}\Gamma_{1,3}$.
By  \eqref{g234-3}, \eqref{g234-1} and Lemma \ref{jb} \ref{jb-2}, \ref{jb-3}, we get
\begin{align}
k_{1,2}&=p_{(1,2),(1,1)}^{(1,3)}+p_{(1,2),(1,2)}^{(1,3)}+p_{(1,2),(2,1)}^{(1,3)}+p_{(1,2),(1,3)}^{(1,3)}+p_{(1,2),(3,1)}^{(1,3)}+p_{(1,2),(2,2)}^{(1,3)},\nonumber\\
k_{1,2}&=p_{(2,1),(1,2)}^{(1,3)}+p_{(2,1),(1,3)}^{(1,3)}=p_{(1,2),(2,1)}^{(1,3)}+p_{(1,2),(1,3)}^{(1,3)},\nonumber
\end{align}
 which imply  $p_{(1,2),(1,1)}^{(1,3)}=p_{(1,2),(3,1)}^{(1,3)}=0$. By Lemma \ref{jb} \ref{jb-2}, one has $p_{(1,3),(1,1)}^{(1,2)}=p_{(1,3),(1,3)}^{(1,2)}=0$. By  \eqref{g234-1}, we obtain $\Gamma_{1,3}\Gamma_{1,1},\Gamma_{1,3}^2\subseteq \{\Gamma_{1,3},\Gamma_{2,2}\}$.

By Lemma \ref{one}, one gets $\Gamma_{2,2}\in \Gamma_{1,3}^2$. Pick a partite set $V$ such that $\Gamma_{3,1}(x)\cap V\ne \emptyset$. By Lemma \ref{le216fj}, one obtains $|\Gamma_{1,3}(x)\cap V|=p_{(1,3),(2,2)}^{(3,1)}=\frac{k_{1,3}}{t}$. In view of \eqref{g234-4},  Lemma \ref{jb} \ref{jb-2} and Lemma \ref{main3''}, we have $p_{(1,3),(2,2)}^{(1,3)}=\frac{k_{1,3}}{t}-1$.

Since $\Gamma_{1,3}^2\subseteq\{\Gamma_{1,3},\Gamma_{2,2}\}$, from Lemma \ref{jb} \ref{jb-2} and \ref{jb-3}, one has $k_{1,3}=p_{(3,1),(1,3)}^{(1,3)}+p_{(3,1),(2,2)}^{(1,3)}.$ Since $p_{(3,1),(2,2)}^{(1,3)}=p_{(1,3),(2,2)}^{(3,1)}=\frac{k_{1,3}}{t}$, we get
\begin{align}
(1-\frac{1}{t})k_{1,3}=p_{(3,1),(1,3)}^{(1,3)}=p_{(1,3),(1,3)}^{(1,3)}.\label{step-1}\tag{5.8}
\end{align}
	 By \eqref{g234-4} and Lemma \ref{jb} \ref{jb-2}, \ref{jb-3}, one has
	\begin{align}
k_{1,3}=p_{(1,3),(1,1)}^{(1,3)}+p_{(1,3),(1,2)}^{(1,3)}+p_{(1,3),(2,1)}^{(1,3)}+p_{(1,3),(1,3)}^{(1,3)}+p_{(1,3),(3,1)}^{(1,3)}+p_{(1,3),(2,2)}^{(1,3)}+1.\nonumber
\end{align}
By substituting $p_{(1,3),(2,2)}^{(1,3)}=\frac{k_{1,3}}{t}-1$ into the above equation, one has
\begin{align}
(1-\frac{1}{t})k_{1,3}=p_{(1,3),(1,1)}^{(1,3)}+p_{(1,3),(1,2)}^{(1,3)}+p_{(1,3),(2,1)}^{(1,3)}+p_{(1,3),(1,3)}^{(1,3)}+p_{(1,3),(3,1)}^{(1,3)}.\label{step-2}\tag{5.9}
	\end{align}
By \eqref{step-1} and \eqref{step-2}, we get $p_{(1,3),(1,1)}^{(1,3)}=0$, which implies  $\Gamma_{1,3}\Gamma_{1,1}=\{\Gamma_{2,2}\}$.
	
	In view of Lemma \ref{jb} \ref{jb-2}, we have $p_{(1,3),(3,1)}^{(2,2)}=\frac{k_{1,3}}{k_{2,2}}p_{(2,2),(1,3)}^{(1,3)}$ and $p_{(2,2),(1,1)}^{(1,3)}=\frac{k_{2,2}}{k_{1,3}}p_{(1,3),(1,1)}^{(2,2)}$. By setting $i=j=(1,3)$, $l=(1,1)$ and $m=(3,1)$ in Lemma \ref{jb} \ref{jb-4}, one obtains
	\begin{align}
	p_{(1,3),(1,1)}^{(2,2)}p_{(2,2),(3,1)}^{(1,3)}=p_{(1,3),(3,1)}^{(2,2)}p_{(2,2),(1,1)}^{(1,3)}=p_{(2,2),(1,3)}^{(1,3)}p_{(1,3),(1,1)}^{(2,2)}.\nonumber
	\end{align}
Since $p_{(1,3),(1,1)}^{(2,2)}\ne 0$, one has $\frac{k_{1,3}}{t}=p_{(2,2),(3,1)}^{(1,3)}=p_{(2,2),(1,3)}^{(1,3)}=\frac{k_{1,3}}{t}-1$, a contradiction.

\begin{substep}\label{G1213fj}
Show that $p_{(1,3),(2,2)}^{(1,3)}=\frac{k_{1,3}}{t}-1$ and 	  $|\Gamma_{3,1}(x)\cap V|=p_{(3,1),(2,2)}^{(1,2)}=\frac{k_{1,3}}{t}$, where  $V$ is   a partite set with $\Gamma_{3,1}(x)\cap V\ne \emptyset$.
\end{substep}
Pick a partite set $V$ such that $\Gamma_{3,1}(x)\cap V\ne \emptyset$. By Step \ref{t234xj2}, we get $p_{(1,3),(1,2)}^{(2,2)}>0$. In view of Lemma \ref{le216fj}, we get $|\Gamma_{1,2}(x)\cap V|=p_{(1,2),(2,2)}^{(3,1)}=\frac{k_{1,2}}{t}$. Note that there are exactly $t$ partite sets having nonempty intersection with $\Gamma_{3,1}(x)$. Since $V$ was arbitrary, there are exactly $t$ partite sets having nonempty intersection with $\Gamma_{1,2}(x)$. Since $p_{(1,3),(1,2)}^{(2,2)}>0$, from Lemma \ref{le216fj}, one obtains $|\Gamma_{3,1}(x)\cap V| =p_{(3,1),(2,2)}^{(1,2)}=\frac{k_{1,3}}{t}$. In view of \eqref{g234-4}, Lemma \ref{jb} \ref{jb-2} and Lemma \ref{main3''}, we have $p_{(1,3),(2,2)}^{(1,3)}=p_{(3,1),(2,2)}^{(3,1)}=\frac{k_{1,3}}{t}-1$.

\begin{substep}\label{t234fenxi}
 Show that $\Gamma_{2,2}\notin \Gamma_{1,3}\Gamma_{1,1}\cup \Gamma_{1,2}\Gamma_{1,1}\cup \Gamma_{1,2}^2$ if	$\Gamma_{2,2}\notin \Gamma_{1,3}^2$.
\end{substep}
Since $\Gamma_{2,2}\notin \Gamma_{1,3}^2$, from Lemma \ref{tf22}, we have $\Gamma_{2,2}\notin \Gamma_{1,3}\Gamma_{1,1}$. We only need to prove that $\Gamma_{2,2}\notin  \Gamma_{1,2}\Gamma_{1,1}\cup \Gamma_{1,2}^2$. By Step \ref{t234xj2}, one gets $p_{(1,2),(1,3)}^{(2,2)}\ne 0$. Let $(y,z)\in \Gamma_{2,2}$ and $x'\in P_{(1,2),(1,3)}(y,z)$.

 Suppose  $\Gamma_{2,2}\in  \Gamma_{1,2}^2$.  By Lemma \ref{jb} \ref{jb-2}, one obtains $p_{(2,2),(2,1)}^{(1,2)}\ne 0$, which implies that there exists a vertex $u\in P_{(2,2),(2,1)}(y,x')$. Since $p_{(1,2),(1,3)}^{(2,2)}\ne 0$, from Lemma \ref{jb} \ref{jb-2}, we have $p_{(3,1),(2,2)}^{(1,2)}\ne 0$, which implies that there exists a vertex $v\in P_{(3,1),(2,2)}(x',u)$. Since $(y,z), (y,u),(u,v)\in \Gamma_{2,2}$, from Lemma \ref{jichu0}, $v,z$ belong to the same partite set. Since $x'\in P_{(1,3),(1,3)}(v,z)$, from \eqref{g234-1}, one has $\Gamma_{2,2}\in  \Gamma_{1,3}^2$, a contradiction. Thus, $\Gamma_{2,2}\notin  \Gamma_{1,2}^2$. By Lemma \ref{tf22}, we have $\Gamma_{2,2}\notin \Gamma_{1,2}\Gamma_{1,1}$. The desired result holds.

\begin{substep}\label{t234xj6}
Show that	 $\Gamma_{2,2}\in \Gamma_{1,3}^2$.
\end{substep}

Suppose for the contrary that $\Gamma_{2,2}\notin \Gamma_{1,3}^2$. Step \ref{t234fenxi} implies $\Gamma_{2,2}\notin \Gamma_{1,3}\Gamma_{1,1}$. By \eqref{g234-1}, we get   $\Gamma_{1,3}^2,\Gamma_{1,3}\Gamma_{1,1}\subseteq \{ \Gamma_{1,2},\Gamma_{1,3}\}$.
 Pick a partite set $V$ such that $\Gamma_{3,1}(x)\cap V\ne \emptyset$. In view of Lemma \ref{le216fj} and Step \ref{t234xj2}, one obtains $|\Gamma_{1,2}(x)\cap V|=p_{(1,2),(2,2)}^{(3,1)}=\frac{k_{1,2}}{t}$. By Step \ref{G1213fj}, we have  $|\Gamma_{3,1}(x)\cap V|=p_{(3,1),(2,2)}^{(1,2)}=\frac{k_{1,3}}{t}$ and $p_{(1,3),(2,2)}^{(1,3)}=\frac{k_{1,3}}{t}-1$.

Since
$\Gamma_{1,3}^2,\Gamma_{1,3}\Gamma_{1,1}\subseteq \{ \Gamma_{1,2},\Gamma_{1,3}\}$, we get $\Gamma_{3,1}^2,\Gamma_{3,1}\Gamma_{1,1}\subseteq \{ \Gamma_{2,1},\Gamma_{3,1}\}$. Since $\Gamma_{3,1}(x)\cap V\ne \emptyset$,  from Lemma \ref{jiaojik}, one has  $\Gamma_{1,3}(x)\cap V=\Gamma_{1,1}(x)\cap V=\emptyset$.  In view of \eqref{g234-1''} and Step \ref{t234fenxi}, one gets  $\Gamma_{1,2}^2\subseteq \{\Gamma_{1,1},\Gamma_{1,2}, \Gamma_{2,1},\Gamma_{1,3}\}$. Since  $\Gamma_{1,2}(x)\cap V\ne \emptyset$, from Lemma \ref{jiaojik}, we obtain $\Gamma_{2,1}(x)\cap V=\emptyset$. Lemma \ref{pssize} implies $|V|=|\Gamma_{3,1}(x)\cap V|+|\Gamma_{1,2}(x)\cap V|=\frac{k_{1,2}+k_{1,3}}{t}$. By Lemma \ref{jichu0}, one has $|V|=1+\sum_{h,l>1}k_{h,l}$. In view of Lemmas \ref{changyong} and \ref{zuidaarc}, we obtain $|V|=\frac{k_{1,2}+k_{1,3}}{t}=k_{2,2}+1$.

Assume that there are exactly $s$ partite sets having nonempty intersection with $\Gamma_{1,1}(x)$. Pick a partite set $\tilde{V}$ such that $\Gamma_{1,1}(x)\cap\tilde{V}\ne \emptyset$. Since $\Gamma_{2,2}\notin \Gamma_{1,q}\Gamma_{1,1}$ for $q\in\{2,3\}$ from Step \ref{t234fenxi},  we get $\Gamma_{1,q}\Gamma_{1,1}\cup \Gamma_{q,1}\Gamma_{1,1}\subseteq\{\Gamma_{1,1},\Gamma_{1,2},\Gamma_{2,1},\Gamma_{1,3},\Gamma_{3,1}\}$ by Lemma \ref{changyong}. By Lemma \ref{jiaojik}, one has $\Gamma_{1,q}(x)\cap \tilde{V}=\Gamma_{q,1}(x)\cap \tilde{V}=\emptyset$ for all $q\in \{2,3\}$. Lemma \ref{pssize} implies  $\tilde{V}\subseteq \Gamma_{1,1}(x)$. Since $\tilde{V}$ was arbitrary, we have $|\tilde{V}|=\frac{k_{1,1}}{s}$. By  Lemma \ref{jichu}, one gets $|V|=|\tilde{V}|$, and so $\frac{k_{1,1}}{s}=\frac{k_{1,2}+k_{1,3}}{t}=k_{2,2}+1$.

Since  $\Gamma_{2,2}\notin \Gamma_{1,3}^2\cup \Gamma_{1,3}\Gamma_{1,1}$, from Lemma \ref{jb} \ref{jb-2}, one has $\Gamma_{3,1},\Gamma_{1,1}\notin \Gamma_{1,3}\Gamma_{2,2}$. By the first statement in Lemma \ref{jichu2}, we have  $\Gamma_{1,3}\Gamma_{2,2}\subseteq\{\Gamma_{1,2},\Gamma_{2,1},\Gamma_{1,3}\}$.
In view of Lemma \ref{jb}  \ref{jb-3}, we get
\begin{align}
\frac{k_{1,2}+k_{1,3}}{t}-1=k_{2,2}=p_{(1,3),(2,2)}^{(1,3)}+p_{(1,2),(2,2)}^{(1,3)}+p_{(2,1),(2,2)}^{(1,3)}.\nonumber
\end{align}
By substituting $p_{(1,3),(2,2)}^{(1,3)}=\frac{k_{1,3}}{t}-1$ and  $p_{(2,1),(2,2)}^{(1,3)}=p_{(1,2),(2,2)}^{(3,1)}=\frac{k_{1,2}}{t}$ into the above equation,  one obtains  $p_{(1,2),(2,2)}^{(1,3)}=0$. By Lemma \ref{jb} \ref{jb-2}, we get $\Gamma_{1,3}\Gamma_{2,2}\subseteq\{\Gamma_{2,1},\Gamma_{1,3}\}$.



By setting $i=j=(1,3)$, $l=(1,a)$ and $m=(2,2)$ for  $a\in \{2,3\}$ in Lemma \ref{jb} \ref{jb-4}, one obtains
$$p_{(1,3),(1,a)}^{(1,3)}p_{(1,3),(2,2)}^{(1,3)}+p_{(1,3),(1,a)}^{(2,1)}p_{(2,1),(2,2)}^{(1,3)}=p_{(1,3),(2,2)}^{(2,1)}p_{(2,1),(1,a)}^{(1,3)}+p_{(1,3),(2,2)}^{(1,3)}p_{(1,3),(1,a)}^{(1,3)}$$
from Lemma \ref{jb} \ref{jb-2}. Since $p_{(1,3),(1,a)}^{(2,1)}=0$, we get
$p_{(1,3),(2,2)}^{(2,1)}p_{(2,1),(1,a)}^{(1,3)}=0$. Since $p_{(1,3),(2,2)}^{(2,1)}=p_{(3,1),(2,2)}^{(1,2)}\ne 0$, we have $p_{(2,1),(1,a)}^{(1,3)}=0$. By Lemma \ref{jb} \ref{jb-2}, one gets $p_{(1,2),(1,3)}^{(1,a)}=0$ for  $a\in \{2,3\}$. In view of  \eqref{g234-1}, we have  $\Gamma_{1,2}\Gamma_{1,3}=\{ \Gamma_{2,2}\}$. By Lemma \ref{jb} \ref{jb-2} and \ref{jb-3}, we have $k_{1,2}=p_{(2,1),(2,2)}^{(1,3)}=\frac{k_{1,2}}{t}$, and so $t=1$. Then
$\frac{k_{1,1}}{s}=k_{1,2}+k_{1,3}$.

Suppose that $\Gamma_{1,3}\in \Gamma_{1,3}\Gamma_{1,1}$. Let $(y,w)\in \Gamma_{1,3}$ and $x'\in P_{(1,3),(1,1)}(y,w)$. Since $p_{(2,1),(2,2)}^{(1,3)}\ne 0$,  there exists a vertex $z\in P_{(2,1),(2,2)}(y,x')$. It follows that $z\notin N^+(x')\cup N^-(x')$. By Lemma \ref{jichu2}, one gets $(z,w)$ or $(w,z)\in A\Gamma$. Since $y\in P_{(1,2),(1,3)}(z,w)$ and $\Gamma_{1,2}\Gamma_{1,3}=\{\Gamma_{2,2}\}$, we have $(z,w)\in \Gamma_{2,2}$, a contradiction. Thus, $\Gamma_{1,3}\notin \Gamma_{1,3}\Gamma_{1,1}$, and so $\Gamma_{1,3}\Gamma_{1,1}=\{ \Gamma_{1,2}\}$.

In view of Lemma \ref{jb} \ref{jb-2} and \ref{jb-3}, one obtains $k_{1,1}=p_{(1,1),(1,2)}^{(1,3)}$. By \eqref{g234-3} and  Lemma \ref{jb} \ref{jb-3}, one obtains
\begin{align}
k_{1,2}&=p_{(1,2),(1,1)}^{(1,3)}+p_{(1,2),(1,2)}^{(1,3)}+p_{(1,2),(2,1)}^{(1,3)}+p_{(1,2),(1,3)}^{(1,3)}+p_{(1,2),(3,1)}^{(1,3)}+p_{(1,2),(2,2)}^{(1,3)} \geq k_{1,1}.\nonumber
\end{align}
Since $\frac{k_{1,1}}{s}=k_{1,2}+k_{1,3}$, we have $\frac{k_{1,1}}{s}\geq  k_{1,1}+k_{1,3}>k_{1,1}$, a contradiction.

\begin{substep}\label{g234jiegou}
	 Show that  $p_{(1,2),(2,2)}^{(1,3)}=\frac{k_{1,2}}{t}$.
\end{substep}
 Pick a partite set $V$ such that $\Gamma_{3,1}(x)\cap V\ne \emptyset$.  By Step \ref{G1213fj}, one has  $|\Gamma_{3,1}(x)\cap V|=p_{(3,1),(2,2)}^{(1,2)}=\frac{k_{1,3}}{t}$ and  $p_{(1,3),(2,2)}^{(1,3)}=\frac{k_{1,3}}{t}-1$.

 By Steps \ref{t234xj2} and  \ref{t234xj6}, one has $p_{(1,2),(1,3)}^{(2,2)}>0$ and $p_{(1,3),(1,3)}^{(2,2)}>0$. Lemma \ref{le216fj} implies $|\Gamma_{1,2}(x)\cap V|=p_{(1,2),(2,2)}^{(3,1)}=\frac{k_{1,2}}{t}$ and $|\Gamma_{1,3}(x)\cap V|=p_{(1,3),(2,2)}^{(3,1)}=\frac{k_{1,3}}{t}$.
 Since $V$ was arbitrary,  there are exactly $t$ partite sets having nonempty intersection with $\Gamma_{1,3}(x)$.  In view of Lemma \ref{le216fj} and Step \ref{t234xj2}, one has  $|\Gamma_{2,1}(x)\cap V| =p_{(2,1),(2,2)}^{(1,3)}=\frac{k_{1,2}}{t}$. By Lemma \ref{pssize}, we obtain
 \begin{align}
 |V|=\frac{2(k_{1,2}+k_{1,3})}{t}+|\Gamma_{1,1}(x)\cap V|.\label{step-3}\tag{5.10}
 \end{align}

We claim that $|V|=\frac{2(k_{1,2}+k_{1,3})}{t}+p_{(1,1),(2,2)}^{(3,1)}$. Suppose $\Gamma_{2,2}\notin \Gamma_{1,3}\Gamma_{1,1}$. Lemma \ref{jb} \ref{jb-2} implies $p_{(1,1),(2,2)}^{(3,1)}=0$. By Lemma \ref{changyong}, one gets $\Gamma_{1,3}\Gamma_{1,1}\subseteq\{\Gamma_{1,2},\Gamma_{1,3}\}$.
Since $\Gamma_{1,3}(x)\cap V\ne \emptyset$, from Lemma \ref{jiaojik}, we have $\Gamma_{1,1}(x)\cap V=\emptyset$. \eqref{step-3} implies  that the claim is valid. Suppose $\Gamma_{2,2}\in \Gamma_{1,3}\Gamma_{1,1}$. By Lemma \ref{le216fj}, we get
 $|\Gamma_{1,1}(x)\cap V|=p_{(1,1),(2,2)}^{(3,1)}=\frac{k_{1,1}}{t}$.  Combining with \eqref{step-3}, our claim is valid.

 By Lemma \ref{jichu0}, one has $|V|=1+\sum_{h,l>1}k_{h,l}$. In view of Lemmas \ref{changyong} and \ref{zuidaarc}, we obtain $|V|=k_{2,2}+1$. By the  claim, one has
 \begin{align}
 k_{2,2}=\frac{2(k_{1,2}+k_{1,3})}{t}+p_{(1,1),(2,2)}^{(3,1)}-1.\label{step6-1}\tag{5.11}
 \end{align}
  By the first statement in Lemma \ref{jichu2}, we have  $\Gamma_{1,3}\Gamma_{2,2}\subseteq\{\Gamma_{1,1},\Gamma_{1,2},\Gamma_{2,1},\Gamma_{1,3},\Gamma_{3,1}\}$.
In view of Lemma \ref{jb} \ref{jb-2} and  \ref{jb-3}, one has
\begin{align}
k_{2,2}&=p_{(1,1),(2,2)}^{(1,3)}+p_{(1,2),(2,2)}^{(1,3)}+p_{(2,1),(2,2)}^{(1,3)}+p_{(1,3),(2,2)}^{(1,3)}+p_{(3,1),(2,2)}^{(1,3)}.\nonumber
\end{align}
By substituting $p_{(1,3),(2,2)}^{(1,3)}=\frac{k_{1,3}}{t}-1$, $p_{(2,1),(2,2)}^{(1,3)}=p_{(1,2),(2,2)}^{(3,1)}=\frac{k_{1,2}}{t}$ and  $p_{(3,1),(2,2)}^{(1,3)}=p_{(1,3),(2,2)}^{(3,1)}=\frac{k_{1,3}}{t}$ into the above equation,  we obtain
\begin{align}
k_{2,2}=p_{(1,1),(2,2)}^{(1,3)}+p_{(1,2),(2,2)}^{(1,3)}+\frac{k_{1,2}+2k_{1,3}}{t}-1.\label{step6-2}\tag{5.12}
\end{align}
By \eqref{step6-1} and \eqref{step6-2}, we get $p_{(1,2),(2,2)}^{(1,3)}=\frac{k_{1,2}}{t}$.


\begin{substep}
We reach a contradiction.
\end{substep}

 By \eqref{g234-3} and  Lemma \ref{jb}  \ref{jb-2},  \ref{jb-3}, we have $$k_{1,2}=p_{(1,2),(1,1)}^{(1,3)}+p_{(1,2),(1,2)}^{(1,3)}+p_{(1,2),(2,1)}^{(1,3)}+p_{(1,2),(1,3)}^{(1,3)}+p_{(1,2),(3,1)}^{(1,3)}+p_{(1,2),(2,2)}^{(1,3)}.$$
By Step \ref{g234jiegou}, one has $p_{(1,2),(2,2)}^{(1,3)}=\frac{k_{1,2}}{t}$, which implies
\begin{align}
(1-\frac{1}{t})k_{1,2}=p_{(1,2),(1,1)}^{(1,3)}+p_{(1,2),(1,2)}^{(1,3)}+p_{(1,2),(2,1)}^{(1,3)}+p_{(1,2),(1,3)}^{(1,3)}+p_{(1,2),(3,1)}^{(1,3)}.\label{key-1-3}\tag{5.13}
\end{align}
In view of \eqref{g234-1} and Lemma \ref{jb} \ref{jb-2}, \ref{jb-3},
$k_{1,2}=p_{(2,1),(1,2)}^{(1,3)}+p_{(2,1),(1,3)}^{(1,3)}+p_{(2,1),(2,2)}^{(1,3)}.$
  Lemma \ref{le216fj} and Step \ref{t234xj2}  imply $p_{(2,1),(2,2)}^{(1,3)}=p_{(1,2),(2,2)}^{(3,1)}=\frac{k_{1,2}}{t}$.  It follows  that
\begin{align}
(1-\frac{1}{t})k_{1,2}=p_{(2,1),(1,2)}^{(1,3)}+p_{(2,1),(1,3)}^{(1,3)}=p_{(1,2),(2,1)}^{(1,3)}+p_{(1,2),(1,3)}^{(1,3)}.\label{key-2}\tag{5.14}
\end{align}	
In view of \eqref{key-1-3} and \eqref{key-2}, one has $p_{(1,2),(1,1)}^{(1,3)}=p_{(1,2),(3,1)}^{(1,3)}=0$. By Lemma \ref{jb} \ref{jb-2} and  \eqref{g234-1}, we get $\Gamma_{1,3}\Gamma_{1,1},\Gamma_{1,3}^2\subseteq\{ \Gamma_{1,3},\Gamma_{2,2}\}$.

By \eqref{g234-4} and Lemma \ref{jb}  \ref{jb-2}, \ref{jb-3}, we get
$$k_{1,3}=p_{(1,3),(1,1)}^{(1,3)}+p_{(1,3),(1,2)}^{(1,3)}+p_{(1,3),(2,1)}^{(1,3)}+p_{(1,3),(1,3)}^{(1,3)}+p_{(1,3),(3,1)}^{(1,3)}+p_{(1,3),(2,2)}^{(1,3)}+1.$$
In view of Step \ref{G1213fj}, one has
$p_{(1,3),(2,2)}^{(1,3)}=\frac{k_{1,3}}{t}-1$, which implies
\begin{align}
(1-\frac{1}{t})k_{1,3}=p_{(1,3),(1,1)}^{(1,3)}+p_{(1,3),(1,2)}^{(1,3)}+p_{(1,3),(2,1)}^{(1,3)}+p_{(1,3),(1,3)}^{(1,3)}+p_{(1,3),(3,1)}^{(1,3)}.\label{key-1-1}\tag{5.15}
\end{align}
Since $\Gamma_{1,3}^2\subseteq\{ \Gamma_{1,3},\Gamma_{2,2}\}$, from   Lemma \ref{jb} \ref{jb-2} and \ref{jb-3}, we have
$k_{1,3}=p_{(3,1),(1,3)}^{(1,3)}+p_{(3,1),(2,2)}^{(1,3)}.$
By Lemma \ref{le216fj} and Step \ref{t234xj6}, one has $p_{(3,1),(2,2)}^{(1,3)}=p_{(1,3),(2,2)}^{(3,1)}=\frac{k_{1,3}}{t}$, which implies
\begin{align}
(1-\frac{1}{t})k_{1,3}=p_{(3,1),(1,3)}^{(1,3)}=p_{(1,3),(1,3)}^{(1,3)}.\label{key-1-2}\tag{5.16}
\end{align}
By \eqref{key-1-1} and \eqref{key-1-2}, one obtains   $p_{(1,3),(1,1)}^{(1,3)}=0$ and $p_{(1,3),(1,3)}^{(1,3)}=p_{(1,3),(3,1)}^{(1,3)}=0$, which imply   $\Gamma_{1,3}\Gamma_{1,1}=\{\Gamma_{2,2}\}$ and $t=1$.

Pick a partite set $V$ such that $\Gamma_{3,1}(x)\cap V\ne \emptyset$. In view of Lemma \ref{le216fj}, we obtain $|\Gamma_{1,1}(x)\cap V| =k_{1,1}$.
  Lemma \ref{le216fj} and Steps \ref{t234xj2}, \ref{t234xj6} imply $|\Gamma_{1,2}(x)\cap V| =k_{1,2}$ and $|\Gamma_{1,3}(x)\cap V| =k_{1,3}$. By Step \ref{G1213fj}, one has $|\Gamma_{3,1}(x)\cap V| =k_{1,3}$. Then $\Gamma_{1,1}(x)\cup \Gamma_{1,2}(x)\cup \Gamma_{1,3}(x)\cup \Gamma_{3,1}(x)\subseteq V$. It follows that $2k_{1,3}+k_{1,2}+k_{1,1}\leq |V|$. Since $3\in T$, there exists a circuit of length $3$ containing an arc of type $(1,2)$. By Lemma \ref{jichu0},  $\Gamma$ has at least three distinct partite sets. Let $V'$ be a partite set with $x\notin V'$ and $V'\ne V$. Since $\Gamma_{1,1}(x)\cup \Gamma_{1,2}(x)\cup \Gamma_{1,3}(x)\cup \Gamma_{3,1}(x)\subseteq V$, from Lemma \ref{jichu0}, one has $V'\subseteq \Gamma_{2,1}(x)$. Then $|V'|\leq k_{1,2}$. In view of Lemma \ref{jichu}, we get $2k_{1,3}+k_{1,2}+k_{1,1}\leq |V|=|V'|\leq k_{1,2}$, a contradiction.

\section{$|T|=1$}
In  this section, $\Gamma$ always denotes  a semicomplete multipartite commutative weakly distance-regular digraph with $|T|=1$. By Lemma \ref{jichu},
the underlying graph of $\Gamma$ is isomorphic to $K^{k}_{m}$ with $m\geq2$ and $k\geq 2$.


		


\begin{prop}\label{Tjishu1}
	  The digraph $\Gamma$ is  isomorphic to  one of the digraphs in Theorem \ref{xushumain4} \ref{main1-fu7} and \ref{main1-fu5}.
\end{prop}
\begin{proof}
By Proposition \ref{g234mprop}, we have $T=\{q+1\}$ with $q\in \{2,3\}$. We claim that $s=1$ if $(1,3),(3,s)\in \tilde{\partial}(\Gamma)$ with $s\geq 1$. Suppose $s\geq 2$. Let $(x,y)\in \Gamma_{3,s}$. Then $y\notin N^+(x)\cup N^-(x)$. Pick   a shortest path $(x,x_1,x_2,y)$.
Since $\partial(x,y)=3$, from Lemma \ref{jichu2}, one has $x_1\in N^+(x)=N^+(y)$, which implies that $(y,x_1,x_2)$ is a circuit consisting of arcs of type $(1,3)$, a contradiction. Thus, our claim is valid.
Combining with Lemma \ref{changyong}, we have
	\begin{align}
	\Gamma_{1,q}^2&\subseteq\{\Gamma_{1,q},\Gamma_{2,q-1},\Gamma_{2,2}\},\label{t3-1}\tag{6.1}\\ \Gamma_{1,q}\Gamma_{q,1}&\subseteq \{\Gamma_{0,0}, \Gamma_{1,q},\Gamma_{q,1},\Gamma_{2,2},\Gamma_{q+1,q+1}\}.\label{t3-2}\tag{6.2}
\end{align}
It follows from Lemma \ref{sxjl} that
	\begin{align}
\tilde{\partial}(\Gamma)\subseteq\{(0,0),(1,q),(q,1),(2,2),(q+1,q+1)\}.\label{t1sxjl}\tag{6.3}
\end{align}

		\textbf{Case 1.} $(2,2)\notin \tilde{\partial}(\Gamma)$.
	
 In view of Lemma \ref{one} and \eqref{t3-1}, one has  $T=\{3\}$.	Since the underlying graph of $\Gamma$ is isomorphic to $K^{k}_{m}$, from  \eqref{t1sxjl}, we get $\tilde{\partial}(\Gamma)=\{(0,0),(1,2),(2,1),(3,3)\}$ with  $k_{1,2}=\frac{(k-1)m}{2}$ and $k_{3,3}=m-1$.
	Let $F=\{\Gamma_{0,0},\Gamma_{3,3}\}$. By Lemma \ref{changyong2}, 	$F$ is closed. Denote $\Sigma:=\Gamma/F$.
	By Lemma \ref{tdxw1},  $\Gamma$ is isomorphic to
	$\Sigma\circ \overline{{K}}_{m}$, where $\Sigma$ is a semicomplete weakly distance-regular digraph with $\tilde{\partial}(\Sigma) =\{(0,0),(1,2),(2,1)\}$. Thus, $\Gamma$ is isomorphic to one of
	the digraphs in Theorem \ref{xushumain4}  \ref{main1-fu7}.
	
		\textbf{Case 2.} $(2,2)\in \tilde{\partial}(\Gamma)$.
	
	Let $t=0$ if $(q+1,q+1)\notin \tilde{\partial}(\Gamma)$, and $t=q+1$ if $(q+1,q+1)\in \tilde{\partial}(\Gamma)$. By \eqref{t1sxjl}, one has $\tilde{\partial}(\Gamma)=\{(0,0),(1,q),(q,1),(2,2),(t,t)\}$.
	Let  $F=\{\Gamma_{0,0},\Gamma_{t,t}\}$. By Lemma \ref{changyong2}, 	$F$ is closed. Denote  $\Sigma:=\Gamma/F$. By Lemma \ref{tdxw1},
	$\Gamma$ is isomorphic to $\Sigma\circ \overline{{K}}_{k_{q+1,q+1}+1}$ and $\Sigma$ is a commutative weakly distance-regular digraph with
	\begin{align}
\tilde{\partial}(\Sigma)=\{(0,0),(1,q),(q,1),(2,2)\}.\label{3-2'}\tag{6.4}
	\end{align}
	 To distinguish the notations between $\Gamma$ and $\Sigma$, let $\overline{p}_{\wz{i},\wz{j}}^{\wz{h}}, \overline{k}_{\wz{i}}$ denote the symbols belonging to $\Sigma$ with $\wz{i},\wz{j},\wz{h}\in\wz{\partial}(\Sigma)$ and $x,y\in V\Sigma$. Since the underlying graph of $\Gamma$ is isomorphic to $K^{k}_{m}$,  the underlying graph of  $\Sigma$ is isomorphic to $K^{k}_{m'}$ with  $m=m'(k_{q+1,q+1}+1)$. \eqref{3-2'} implies $\overline{k}_{1,q}=\frac{(k-1)m'}{2}$ and $\overline{k}_{2,2}=m'-1$.
	
	Since $(V\Sigma,\{\Sigma_{0,0},\Sigma_{1,q},\Sigma_{q,1},\Sigma_{2,2}\})$ is a 3-class non-symmetric association scheme,  from \cite[Proposition 5.3]{JG14}, $\Sigma$ is a doubly regular $(k,m')$-team tournament with parameters $(\overline{p}_{(1,q),(1,q)}^{(1,q)},\overline{p}_{(1,q),(1,q)}^{(q,1)},\overline{p}_{(1,q),(1,q)}^{(2,2)})$ of Type \Rmnum{1} or \Rmnum{2}.
	
	Suppose that $\Sigma$ has Type \Rmnum{1}. Then $\overline{p}_{(1,q),(1,q)}^{(q,1)}-\overline{p}_{(1,q),(1,q)}^{(1,q)}=m'$.  By \eqref{t3-1}, we have  $\overline{p}_{(1,3),(1,3)}^{(3,1)}=0$, which implies that $q=2$.  In view of  \eqref{3-2'} and Lemma \ref{jb} \ref{jb-2}, \ref{jb-3}, one obtains
	\begin{align}
	\overline{k}_{1,2}&=\overline{p}_{(1,2),(2,1)}^{(1,2)}+\overline{p}_{(2,1),(2,1)}^{(1,2)}+\overline{p}_{(2,2),(2,1)}^{(1,2)}=\overline{p}_{(1,2),(1,2)}^{(1,2)}+\overline{p}_{(1,2),(1,2)}^{(2,1)}+\overline{p}_{(2,2),(2,1)}^{(1,2)}.\label{3-1}\tag{6.5}\\ \overline{k}_{1,2}-1&=\overline{p}_{(1,2),(1,2)}^{(1,2)}+\overline{p}_{(2,1),(1,2)}^{(1,2)}+\overline{p}_{(2,2),(1,2)}^{(1,2)}=2\overline{p}_{(1,2),(1,2)}^{(1,2)}+\overline{p}_{(2,2),(1,2)}^{(1,2)}.\label{3-2}\tag{6.6}
	\end{align}
	Since $\Sigma_{1,2}\notin\Sigma_{2,2}^2$ from Lemma \ref{jichu3}, by Lemma \ref{jb} \ref{jb-3}, we have $m'-1=\overline{k}_{2,2}=\overline{p}_{(2,2),(1,2)}^{(1,2)}+\overline{p}_{(2,2),(2,1)}^{(1,2)}$. \eqref{3-1} and \eqref{3-2} imply $$2\overline{k}_{1,2}-1=3\overline{p}_{(1,2),(1,2)}^{(1,2)}+\overline{p}_{(1,2),(1,2)}^{(2,1)}+m'-1.$$ By substituting $\overline{k}_{1,2}=\frac{(k-1)m'}{2}$ and $\overline{p}_{(1,2),(1,2)}^{(2,1)}=\overline{p}_{(1,2),(1,2)}^{(1,2)}+m'$ into the above equation, one has $\overline{p}_{(1,2),(1,2)}^{(1,2)}=\frac{(k-3)m'}{4}$, and so $\overline{p}_{(1,2),(1,2)}^{(2,1)}=\frac{(k+1)m'}{4}$. \eqref{3-1} implies $\overline{p}_{(2,2),(2,1)}^{(1,2)}=0$. By Lemma \ref{jb} \ref{jb-2}, we get $\overline{p}_{(1,2),(1,2)}^{(2,2)}=0$. Since $\Sigma$ is a doubly regular $(k,m')$-team tournament and $q=2$, we have $\overline{p}_{(1,2),(1,2)}^{(2,2)}\ne0$, a contradiction. By \cite[Theorem 4.4]{JG14}, $\Sigma$ is a doubly regular $(k,m')$-team tournament of Type \Rmnum{2} with parameters $(\frac{(k-2)m'}{4},\frac{(k-2)m'}{4},\frac{(k-1)m'^2}{4(m'-1)})$.
	
	Suppose $T=\{3\}$. If $\Gamma_{3,3}\notin \Gamma_{1,2}\Gamma_{2,1}$, then $t=0$ and $m'=m$, which imply that $\Gamma$ is  isomorphic to  the digraph in Theorem \ref{xushumain4} \ref{main1-fu5} with $n=1$. If $\Gamma_{3,3}\in \Gamma_{1,2}\Gamma_{2,1}$, then $t=3$ and $m=m'(k_{3,3}+1)$, which imply that  $\Gamma$ is isomorphic to one of the digraphs in Theorem \ref{xushumain4}  \ref{main1-fu5} with $n\geq 2$.

	Suppose $T=\{4\}$. By \cite[Theorem 4.4]{JG14}, one has $\overline{p}_{(1,3),(1,3)}^{(1,3)}=\overline{p}_{(1,3),(1,3)}^{(3,1)}=\frac{(k-2)m'}{4}=0$ and $m'-1\mid k-1$, which imply that $k=2$ and $m'=2$. It follows that   $\overline{k}_{1,3}=\frac{(k-1)m'}{2}=1$. Then $\Sigma$ is isomorphic to $C_4$.
	If $\Gamma_{4,4}\notin \Gamma_{1,3}\Gamma_{3,1}$, then $t=0$ and $m'=m$, which imply that $\Gamma$ is isomorphic to  the digraph in Theorem \ref{xushumain4} \ref{main1-fu5} with $n=1$ and $k=l=2$.	
	If $\Gamma_{4,4}\in \Gamma_{1,3}\Gamma_{3,1}$, then $t=3$ and  $m=m'(k_{3,3}+1)$, which imply that
	$\Gamma$ is isomorphic to one of the digraphs in Theorem \ref{xushumain4} \ref{main1-fu5} with $n\geq 2$ and $k=l=2$.
\end{proof}


\section{$T=\{3,4\}$}
In  this section, $\Gamma$ always denotes  a semicomplete multipartite commutative weakly distance-regular digraph with $T=\{3,4\}$. We shall prove that $\Gamma$ is isomorphic to one of the digraphs in Theorem \ref{xushumain4} \ref{main1-fu2} and \ref{main1-fu0}.

\begin{lemma}\label{changyong3}
	We have $\tilde{\partial}(\Gamma)\subseteq\{(0,0),(1,2),(2,1),(1,3),(3,1),(2,2),(3,3)\}$.
\end{lemma}
\begin{proof}
	Let $(s,t)\in \tilde{\partial}(\Gamma)$  with $t\geq s\geq 3$. By Lemma \ref{changyong}, it suffices to show that $s=t=3$. Let $(x,y)\in\Gamma_{s,t}$. Then $y\notin N^{+}(x)\cup N^{-}(x)$. Pick a vertex $w\in\Gamma_{1,2}(x)$. Since $\partial(x,y)=s\geq 3$, from Lemma \ref{jichu2}, we obtain $w\in N^{+}(x)=N^{+}(y)$, which implies $t=\partial(y,x)\leq 1+\partial(w,x)=3$. Thus, the desired result follows.
\end{proof}

\begin{lemma}\label{jl}
	Suppose that $(3,3)\notin \tilde{\partial}(\Gamma)$. Then the following hold:
	\begin{enumerate}
		\item\label{jl-1'}  $\Gamma_{1,2}^2\subseteq \{ \Gamma_{1,2},\Gamma_{2,1},\Gamma_{1,3},\Gamma_{2,2}\}$;
		
		\item\label{jl-2'} $\Gamma_{1,q}\Gamma_{q,1}\subseteq \{\Gamma_{0,0}, \Gamma_{1,2},\Gamma_{2,1},\Gamma_{1,3},\Gamma_{3,1},\Gamma_{2,2}\}$;
		
		\item\label{jl-3'} $\Gamma_{1,q}\Gamma_{1,3}\subseteq \{ \Gamma_{1,2},\Gamma_{1,3},\Gamma_{2,2}\}$;

		\item\label{jl-6'} $\Gamma_{1,p}\Gamma_{q,1}\subseteq \{ \Gamma_{1,2},\Gamma_{2,1},\Gamma_{1,3},\Gamma_{3,1},\Gamma_{2,2}\}$.
	\end{enumerate}
Here, $\{p,q\}=\{2,3\}$.
\end{lemma}
\begin{proof}
	(i)--(iv) are immediate from Lemma \ref{changyong3}.
\end{proof}

\begin{lemma}\label{t34xj2*dier}
	Suppose $(3,3)\notin \tilde{\partial}(\Gamma)$. If $\Gamma_{2,2}\notin \Gamma_{1,2}\Gamma_{1,3}$,  then $\Gamma$ is isomorphic to one of the digraphs in Theorem \ref{xushumain4} \ref{main1-fu0} with $n=1$.
\end{lemma}
\begin{proof}
 By  Lemma \ref{jb} \ref{jb-2},\ref{jb-3} and Lemma \ref{jl}  \ref{jl-3'}, \ref{jl-6'}, we get
\begin{align}
k_{1,2}&=p_{(2,1),(1,2)}^{(1,3)}+p_{(2,1),(1,3)}^{(1,3)}=p_{(1,2),(2,1)}^{(1,3)}+p_{(1,2),(1,3)}^{(1,3)},\nonumber\\
k_{1,2}&=p_{(1,2),(1,2)}^{(1,3)}+p_{(1,2),(2,1)}^{(1,3)}+p_{(1,2),(1,3)}^{(1,3)}+p_{(1,2),(3,1)}^{(1,3)}+p_{(1,2),(2,2)}^{(1,3)},\nonumber
\end{align}
which imply  that
 $p_{(1,2),(1,2)}^{(1,3)}=p_{(1,2),(3,1)}^{(1,3)}=p_{(1,2),(2,2)}^{(1,3)}=0$. Lemma \ref{jb} \ref{jb-2} implies $\Gamma_{1,3}\Gamma_{2,1}\subseteq \{\Gamma_{2,1},\Gamma_{1,3}\}$. Since $p_{(1,2),(1,2)}^{(1,3)}=0$, from Lemma \ref{jl} \ref{jl-1'}, one obtains $\Gamma_{1,2}^2\subseteq \{ \Gamma_{1,2},\Gamma_{2,1},\Gamma_{2,2}\}$.

 Pick a vertex $x\in V\Gamma$. Assume that there are exactly $t$ partite sets having nonempty intersection with $\Gamma_{1,3}(x)$. Set $F=\langle  \Gamma_{1,3}\rangle$.
Let   $\Sigma:=\Gamma/F$ and $\Delta$ be a digraph with the vertex set $F(x)$ such that $(y,z)\in A\Delta$ whenever $(y,z)\in\Gamma_{1,3}$. Next, we prove this lemma  step by step.

\begin{step}\label{22bu1213-2}
Show that $\Gamma_{3,1}(x)\cup\Gamma_{1,3}(x)\subseteq V$ for some partite set $V$ and  $\Gamma_{1,3}\Gamma_{\wz{i}}\subseteq\{\Gamma_{0,0},\Gamma_{2,2}\}$ for $\wz{i} \in \{(1,3),(3,1)\}$.
\end{step}
Since $p_{(1,2),(3,1)}^{(1,3)}=0$, from  Lemma \ref{jb} \ref{jb-2}, one has $p_{(1,3),(1,3)}^{(1,2)}=0$. It follows from Lemma \ref{jl}  \ref{jl-3'} that $\Gamma_{1,3}^2\subseteq \{ \Gamma_{1,3},\Gamma_{2,2}\}$. By Lemma \ref{one}, we have $\Gamma_{2,2}\in \Gamma_{1,3}^2$. Pick a partite set $V$ such that $V\cap \Gamma_{1,3}(x)\ne \emptyset$. By Lemma \ref{le216fj}, one gets $|\Gamma_{3,1}(x)\cap V|=p_{(3,1),(2,2)}^{(1,3)}=\frac{k_{1,3}}{t}$. Since $(3,3)\notin \tilde{\partial}(\Gamma)$, from Lemmas  \ref{main3''} and \ref{changyong3}, we obtain $p_{(1,3),(2,2)}^{(1,3)}=p_{(3,1),(2,2)}^{(3,1)}=\frac{k_{1,3}}{t}-1$.

	Since $\Gamma_{1,3}^2\subseteq \{ \Gamma_{1,3},\Gamma_{2,2}\}$, from  Lemma \ref{jb} \ref{jb-2},  \ref{jb-3} and Lemma \ref{jl} \ref{jl-2'}, we get
\begin{align}	k_{1,3}&=p_{(3,1),(1,3)}^{(1,3)}+p_{(3,1),(2,2)}^{(1,3)}=p_{(1,3),(1,3)}^{(1,3)}+p_{(3,1),(2,2)}^{(1,3)},\nonumber\\ k_{1,3}&=p_{(1,3),(1,2)}^{(1,3)}+p_{(1,3),(2,1)}^{(1,3)}+p_{(1,3),(1,3)}^{(1,3)}+p_{(1,3),(3,1)}^{(1,3)}+p_{(1,3),(2,2)}^{(1,3)}+1.\nonumber
	\end{align}
	By substituting $p_{(3,1),(2,2)}^{(1,3)}=\frac{k_{1,3}}{t}$ and $p_{(1,3),(2,2)}^{(1,3)}=\frac{k_{1,3}}{t}-1$
	into the above equations, one has
\begin{align}	(1-\frac{1}{t})k_{1,3}=p_{(1,3),(1,3)}^{(1,3)}=p_{(1,3),(1,2)}^{(1,3)}+p_{(1,3),(2,1)}^{(1,3)}+p_{(1,3),(1,3)}^{(1,3)}+p_{(1,3),(3,1)}^{(1,3)},\nonumber
	\end{align}
which implies that $p_{(1,3),(1,2)}^{(1,3)}=p_{(1,3),(3,1)}^{(1,3)}=0$. By Lemma \ref{jb} \ref{jb-2}, we obtain $p_{(1,3),(3,1)}^{(1,2)}=p_{(1,3),(1,3)}^{(1,3)}=0$.  It follows from Lemma \ref{jl} \ref{jl-2'} and \ref{jl-3'} that $\Gamma_{1,3}^2=\{ \Gamma_{2,2}\}$ and $\Gamma_{1,3}\Gamma_{3,1}\subseteq \{\Gamma_{0,0},\Gamma_{2,2}\}$. Since $(1-\frac{1}{t})k_{1,3}=p_{(1,3),(1,3)}^{(1,3)}=0$, one has $t=1$. It follows that $\Gamma_{1,3}(x)\subseteq V$. Since $|\Gamma_{3,1}(x)\cap V|=\frac{k_{1,3}}{t}=k_{1,3}$, we get $\Gamma_{1,3}(x)\cup\Gamma_{3,1}(x)\subseteq V$. Thus, the desired result follows.

\begin{step}\label{22bu1213-3}
Show that $\Gamma_{1,2}\Gamma_{\wz{i}}=\{ \Gamma_{1,2}\}$ with $\wz{i} \in \{(1,3),(3,1)\}$.
\end{step}
By  Step \ref{22bu1213-2}, we get $\Gamma_{1,2}\notin \Gamma_{1,3}\Gamma_{3,1}$, which implies that $p_{(1,3),(3,1)}^{(1,2)}=0$. In view of Lemma \ref{jb} \ref{jb-2}, one obtains $p_{(1,3),(2,1)}^{(1,3)}=p_{(1,3),(1,2)}^{(1,3)}=0$. It follows that $\Gamma_{1,3}\notin \Gamma_{1,3}\Gamma_{2,1}$ and  $\Gamma_{1,3}\notin \Gamma_{1,2}\Gamma_{1,3}$. Since $\Gamma_{2,2}\notin \Gamma_{1,2}\Gamma_{1,3}$ and $\Gamma_{1,3}\Gamma_{2,1}\subseteq \{\Gamma_{2,1},\Gamma_{1,3}\}$, from   Lemma \ref{jl} \ref{jl-3'}, we have $\Gamma_{1,2}\Gamma_{1,3}=\{ \Gamma_{1,2}\}$ and $\Gamma_{1,3}\Gamma_{2,1}=\{\Gamma_{2,1}\}$. By the commutativity of $\Gamma$,  the desired result follows.

\begin{step}\label{t34xj2*diyi}
Show that $\Gamma_{2,2}\notin \Gamma_{1,2}^2$.
\end{step}
	Suppose for the contrary that $\Gamma_{2,2}\in \Gamma_{1,2}^2$.
	Let $(x,y)\in \Gamma_{2,2}$ and $z\in P_{(1,2),(1,2)}(x,y)$. Pick a vertex $w\in \Gamma_{1,3}(y)$. Since $y\in P_{(1,2),(1,3)}(z,w)$, from Step \ref{22bu1213-3}, we get $(z,w)\in \Gamma_{1,2}$. Since $z\in P_{(1,2),(1,2)}(x,w)$ and $\Gamma_{1,2}^2\subseteq \{ \Gamma_{1,2},\Gamma_{2,1},\Gamma_{2,2}\}$, one obtains $(x,w)\in \Gamma_{1,2}\cup\Gamma_{2,1}\cup\Gamma_{2,2}$.  Since $y\notin N^+(x)\cup N^-(x)$, from Lemma \ref{jichu2}, one has $(x,w)$ or $(w,x)\in A\Gamma$, which implies that $(x,w)\in \Gamma_{1,2}\cup\Gamma_{2,1}$. Since $w\in  P_{(1,3),(1,2)}(y,x)\cup P_{(1,3),(2,1)}(y,x)$, we get $\Gamma_{2,2}\in \Gamma_{1,3}\Gamma_{1,2}\cup \Gamma_{1,3}\Gamma_{2,1}$, contrary to Step \ref{22bu1213-3}.

\begin{step}\label{claim0}
Show that $F=\{\Gamma_{0,0},\Gamma_{1,3},\Gamma_{3,1},\Gamma_{2,2}\}$.
\end{step}	
By Step \ref{22bu1213-2}, one has $\Gamma_{1,3}^2=\{\Gamma_{2,2}\}$. By the first statement in Lemma \ref{jichu2}, we have  $\Gamma_{1,3}\Gamma_{2,2}\subseteq\{\Gamma_{1,2},\Gamma_{2,1},\Gamma_{1,3},\Gamma_{3,1}\}$. Lemma \ref{jb} \ref{jb-2} and Step \ref{22bu1213-3} imply $\Gamma_{1,2},\Gamma_{2,1}\notin \Gamma_{1,3}\Gamma_{2,2}$. It follows that $\Gamma_{1,3}\Gamma_{2,2}\subseteq\{\Gamma_{1,3},\Gamma_{3,1}\}$. Then $\Gamma_{1,3}^3\subseteq\{\Gamma_{1,3},\Gamma_{3,1}\}$. Step  \ref{22bu1213-2} implies $\Gamma_{1,3}\Gamma_{3,1}\subseteq\{\Gamma_{0,0},\Gamma_{2,2}\}$. Then $\Gamma_{1,3}^i\subseteq\{\Gamma_{0,0},\Gamma_{1,3},\Gamma_{3,1},\Gamma_{2,2}\}$ for all $i\geq 1$. Thus, $F=\left\langle \Gamma_{1,3}\right\rangle=\{\Gamma_{0,0},\Gamma_{1,3},\Gamma_{3,1},\Gamma_{2,2}\}$.

\begin{step}\label{claim1}
Show that $\Gamma$ is isomorphic to $\Sigma\circ \Delta$.
\end{step}	
 By the first statement in Lemma \ref{jichu2}, we have  $\Gamma_{1,2}\Gamma_{2,2}\subseteq\{\Gamma_{1,2},\Gamma_{2,1},\Gamma_{1,3},\Gamma_{3,1}\}$. Lemma \ref{jb} \ref{jb-2} and Step \ref{22bu1213-3} imply $\Gamma_{1,3},\Gamma_{3,1}\notin \Gamma_{1,2}\Gamma_{2,2}$. By Lemma \ref{jb} \ref{jb-2} and Step \ref{t34xj2*diyi}, one gets $\Gamma_{2,1}\notin \Gamma_{1,2}\Gamma_{2,2}$. Thus, $\Gamma_{1,2}\Gamma_{2,2}=\{\Gamma_{1,2}\}$.

By Step \ref{claim0}, it sufficies to show that $(x',y')\in\Gamma_{1,2}$ for $x'\in F(x)$ and $y'\in F(y)$ when $(x,y)\in\Gamma_{1,2}$. Without loss of generality, we may assume $x'\neq x$.
Since $F=\{\Gamma_{0,0},\Gamma_{1,3},\Gamma_{3,1},\Gamma_{2,2}\}$ from Step \ref{claim0}, we have $(x,x')\in\Gamma_{1,3}\cup\Gamma_{3,1}\cup\Gamma_{2,2}$ and $(y,y')\in\Gamma_{0,0}\cup\Gamma_{1,3}\cup\Gamma_{3,1}\cup\Gamma_{2,2}$.
If $(x',x)\in\Gamma_{1,3}\cup\Gamma_{3,1}$, from Step   \ref{22bu1213-3}, then  $(x',y)\in \Gamma_{1,2}$.
	If $(x',x)\in\Gamma_{2,2}$,  then $(x',y)\in \Gamma_{1,2}$ since $\Gamma_{1,2}\Gamma_{2,2}=\{\Gamma_{1,2}\}$.  Therefore,  $(x',y)\in \Gamma_{1,2}$.

If $(y,y')\in\Gamma_{0,0}$, then $(x',y')\in\Gamma_{1,2}$. If  $(y,y')\in\Gamma_{1,3}\cup\Gamma_{3,1}$,  from Step  \ref{22bu1213-3}, then  $(x',y')\in \Gamma_{1,2}$.  If $(y,y')\in\Gamma_{2,2}$,  then $(x',y')\in \Gamma_{1,2}$ since $\Gamma_{1,2}\Gamma_{2,2}=\{\Gamma_{1,2}\}$. Thus, the desired result follows.


	
\begin{step}\label{claim2}
Show that  $\Delta\simeq C_4$.
\end{step}
We claim that $\partial_{\Delta}(x_0,x_1)=\partial_{\Gamma}(x_0,x_1)$ for any $x_0,x_1\in V\Delta$. By Step \ref{22bu1213-2}, we have  $p_{(1,3),(1,3)}^{(2,2)}\ne 0$. Since 	$\partial_{\Delta}(x_0,x_1)\geq\partial_{\Gamma}(x_0,x_1)$,  it suffices to show  that $\partial_{\Gamma}(x_0,x_1)\geq\partial_{\Delta}(x_0,x_1)$.  By Lemma \ref{changyong3}, $\partial_{\Gamma}(x_0,x_1)\leq 3$.

Suppose $\partial_{\Gamma}(x_0,x_1)=1$. Since $x_0,x_1\in F(x)$ and $F=\{\Gamma_{0,0},\Gamma_{1,3},\Gamma_{3,1},\Gamma_{2,2}\}$ from Step \ref{claim0},   we get $\wz{\partial}_{\Gamma}(x_0,x_1)=(1,3)$, which implies that $(x_0,x_1)\in A\Delta$. It follows that $\partial_{\Delta}(x_0,x_1)=1$.

Suppose $\partial_{\Gamma}(x_0,x_1)=2$. Since $x_0,x_1\in F(x)$ and $F=\{\Gamma_{0,0},\Gamma_{1,3},\Gamma_{3,1},\Gamma_{2,2}\}$ from Step \ref{claim0},  one obtains $\wz{\partial}_{\Gamma}(x_0,x_1)=(2,2)$. Since  $p_{(1,3),(1,3)}^{(2,2)}\ne 0$,  there exists a path $(x_0,v,x_1)$ consisting of arcs of type $(1,3)$. It follows that $\partial_{\Delta}(x_0,x_1)\leq 2$.

Suppose $\partial_{\Gamma}(x_0,x_1)=3$. Since $x_0,x_1\in F(x)$ and $F=\{\Gamma_{0,0},\Gamma_{1,3},\Gamma_{3,1},\Gamma_{2,2}\}$ from Step \ref{claim0},  one has  $\wz{\partial}_{\Gamma}(x_0,x_1)=(3,1)$. Since $p_{(1,3),(1,3)}^{(2,2)}\ne 0$, from Lemma \ref{jb} \ref{jb-2}, we get  $p_{(1,3),(2,2)}^{(3,1)}=p_{(2,2),(3,1)}^{(1,3)}\ne 0$,  which implies that there exists a vertex  $z\in P_{(1,3),(2,2)}(x_0,x_1)$.  Since $p_{(1,3),(1,3)}^{(2,2)}\ne 0$ again, there exists a path $(z,z_1,x_1)$ consisting of arcs of type $(1,3)$. It follows that $(x_0,z,z_1,x_1)$ is a path consisting of arcs of type $(1,3)$. Then   $\partial_{\Delta}(x_0,x_1)\leq 3$. Thus, our claim is valid.

  In view of  Step \ref{claim0} and the  claim, $\Delta$ is  a semicomplete multipartite weakly distance-regular digraph with $\tilde{\partial}(\Delta)=\{(0,0),(1,3),(3,1),(2,2)\}$. By  Step \ref{22bu1213-2}, we have  $\Gamma_{1,3}(x)\cup\Gamma_{3,1}(x)\subseteq V$ for some partite set $V$. Then $\Delta$ has two partite sets.
	By \cite[Proposition 5.3]{JG14}, $\Delta$ is a doubly regular $(2,m')$-team tournament  of Type  \Rmnum{2} for some positive integer $m'$. By \cite[Theorem 4.4]{JG14}, we have $m'-1\mid 1$, which implies that $m'=2$. Thus, $\Delta\simeq C_4$.

\begin{step}\label{xinjia'}
Show that $\Sigma$ is a  semicomplete weakly distance-regular digraph with girth $3$.
\end{step}
  Suppose that there exist  $F(x)$ and $F(y)$ such that $F(x)\ne F(y)$ and $(x,y)\notin A\Gamma\cup A\Gamma^\top$. Since $(3,3)\notin \tilde{\partial}(\Gamma)$, from Lemma \ref{changyong3}, we have $(x,y)\in \Gamma_{2,2}$. In view of Step \ref{22bu1213-2}, one gets $p_{(1,3),(1,3)}^{(2,2)}\ne 0$, which implies that there exists $z\in P_{(1,3),(1,3)}(x,y)$. Then $y\in F(x)$, and so $F(x)=F(y)$, a contradiction. Then $(x,y)\in A\Gamma\cup A\Gamma^\top$, which implies that $\Sigma$ is a semicomplete digraph. By Step \ref{claim0}, we get $(x,y)\in \Gamma_{1,2}\cup\Gamma_{2,1}$.

Let $x,y\in V\Gamma$ and $y\notin F(x)$. By Lemma \ref{changyong3} and Step \ref{claim0}, one obtains $\partial_{\Gamma}(x,y)=1$ or $2$. If $\partial_{\Gamma}(x,y)=1$, then $\wz{\partial}_{\Gamma}(x,y)=(1,2)$, and so $\partial_{\Sigma}(x,y)=1$. Now suppose $\partial_{\Gamma}(x,y)=2$. Then $\wz{\partial}_{\Gamma}(x,y)=(2,1)$. Pick a shortest path $(x,z,y)$ from $x$ to $y$ in $\Gamma$. Since $T=\{3,4\}$, we have $(x,z),(z,y)\in \Gamma_{1,2}$, which implies that  $(F(x),F(z),F(y))$ is a path in $\Sigma$. Then $\partial_{\Sigma}(F(x),F(y))\leq 2$. Thus, $\partial_{\Sigma}(F(x),F(y))\leq \partial_{\Gamma}(x,y)$.

Pick a shortest path $(F(x)=F(y_0),F(y_1),\ldots,F(y_{h})=F(y))$ from $F(x)$ to $F(y)$ in $\Sigma$. It follows that there exists $y_i'\in F(y_i)$ for each $i\in\{0,1,\ldots,h\}$ such that $(y_0',y_1',\ldots,y_h')$ is a path in $\Gamma$. By Step \ref{claim1},  $(x,y_1',\ldots,y_{h-1}',y)$ is a path in $\Gamma$. It follows that $\partial_{\Gamma}(x,y)\leq\partial_{\Sigma}(F(x),F(y))$. Thus, $\partial_{\Gamma}(x,y)=\partial_{\Sigma}(F(x),F(y))$ and $\tilde{\partial}(\Sigma)=\{(0,0),(1,2),(2,1)\}$.

Let $\wz{h}\in\wz{\partial}(\Sigma)$ and $(F(u),F(v))\in\Sigma_{\wz{h}}$. Then $(u,v)\in \Gamma_{\wz{h}}$. For $\wz{i},\wz{j}\in\wz{\partial}(\Sigma)$, we have
$$P_{\wz{i},\wz{j}}(u,v)=\cup_{w\in P_{\wz{i},\wz{j}}(u,v)}F(w)
=\cup_{F(w)\in \Sigma_{\wz{i}}(F(u))\cap\Sigma_{\wz{j}^*}(F(v))}F(w).$$
Since  $|F(w)|=2k_{1,3}+k_{2,2}+1$ for all  $w\in V\Gamma$, we have $$|\Sigma_{\wz{i}}(F(u))\cap\Sigma_{\wz{j}^*}(F(v))|=\frac{p_{\wz{i},\wz{j}}^{\wz{h}}}{2k_{1,3}+k_{2,2}+1}.$$
Since $\Gamma$ is a commutative weakly distance-regular digraph, $\Sigma$ is a  commutative weakly distance-regular digraph. Thus, the desired result follows.
\end{proof}



\begin{lemma}\label{zuihouxli}
	Suppose $(3,3)\notin \tilde{\partial}(\Gamma)$. If $\Gamma_{2,2}\in \Gamma_{1,2}\Gamma_{1,3}$,  then $\Gamma$ is isomorphic to  the digraph in Theorem \ref{xushumain4} \ref{main1-fu2} with $n=1$.
\end{lemma}
\begin{proof}
Pick a vertex $x\in V\Gamma$. Assume that there are exactly $t$ partite sets having nonempty intersection with $\Gamma_{1,3}(x)$.
  Next, we prove this lemma step by step.
\begin{stepp}\label{jiegou}
 	Show that $\Gamma_{2,2}\notin  \Gamma_{1,3}^2$.
\end{stepp}
	Suppose  for the contrary that $\Gamma_{2,2}\in \Gamma_{1,3}^2$.  Pick a partite set $V$ such that $\Gamma_{1,3}(x)\cap V\ne \emptyset$. Since $\Gamma_{2,2}\in \Gamma_{1,3}^2$,  from Lemma \ref{le216fj}, we have
\begin{align}
|\Gamma_{3,1}(x)\cap V|=p_{(3,1),(2,2)}^{(1,3)}=\frac{k_{1,3}}{t}.\label{5-1}\tag{7.1}
	\end{align}
Since $V$ was arbitrary,  there are exactly $t$ partite sets having nonempty intersection with $\Gamma_{3,1}(x)$. By Lemma \ref{le216fj} again, one gets  $|\Gamma_{1,3}(x)\cap V|=p_{(1,3),(2,2)}^{(3,1)}=p_{(3,1),(2,2)}^{(1,3)}=\frac{k_{1,3}}{t}$. Since $\Gamma_{2,2}\in \Gamma_{1,2}\Gamma_{1,3}$, from  Lemma \ref{le216fj}, we obtain   $|\Gamma_{2,1}(x)\cap V|=p_{(2,1),(2,2)}^{(1,3)}=\frac{k_{1,2}}{t}$ and $|\Gamma_{1,2}(x)\cap V|=p_{(1,2),(2,2)}^{(3,1)}=\frac{k_{1,2}}{t}$. Thus,
\begin{align}
|\Gamma_{1,q}(x)\cap V|=|\Gamma_{q,1}(x)\cap V|=\frac{k_{1,q}}{t}~\text{with}~ q\in \{2,3\}.\label{5-1'}\tag{7.2}
	\end{align}
	By Lemma \ref{pssize}, one has
	$|V|=\frac{2(k_{1,2}+k_{1,3})}{t}.$
In view of Lemmas \ref{jichu0} and \ref{changyong3}, we obtain $|V|=k_{2,2}+1$. Then $k_{2,2}+1=\frac{2(k_{1,2}+k_{1,3})}{t}$.
 By the first statement in Lemma \ref{jichu2}, we have  $\Gamma_{1,3}\Gamma_{2,2}\subseteq\{\Gamma_{1,2},\Gamma_{2,1},\Gamma_{1,3},\Gamma_{3,1}\}$.
 In view of Lemma \ref{jb} \ref{jb-2} and  \ref{jb-3}, one gets
	\begin{align} \frac{2(k_{1,2}+k_{1,3})}{t}-1=k_{2,2}=p_{(1,3),(2,2)}^{(1,3)}+p_{(3,1),(2,2)}^{(1,3)}+p_{(1,2),(2,2)}^{(1,3)}+p_{(2,1),(2,2)}^{(1,3)}.\label{ls}\tag{7.3}
	\end{align}
Since $|\Gamma_{1,3}(x)\cap V|=\frac{k_{1,3}}{t}$, from Lemmas \ref{main3''} and \ref{changyong3}, one has $p_{(1,3),(2,2)}^{(1,3)}=\frac{k_{1,3}}{t}-1$.		By substituting $p_{(1,3),(2,2)}^{(1,3)}=\frac{k_{1,3}}{t}-1$, $p_{(3,1),(2,2)}^{(1,3)}=\frac{k_{1,3}}{t}$ and $p_{(2,1),(2,2)}^{(1,3)}=\frac{k_{1,2}}{t}$
	into \eqref{ls}, we get $p_{(1,2),(2,2)}^{(1,3)}=\frac{k_{1,2}}{t}$.

	In view of Lemma \ref{jl} \ref{jl-3'} and Lemma \ref{jb} \ref{jb-2}, \ref{jb-3}, one has
	$$k_{1,2}=p_{(2,1),(1,2)}^{(1,3)}+p_{(2,1),(1,3)}^{(1,3)}+p_{(2,1),(2,2)}^{(1,3)}=p_{(1,2),(2,1)}^{(1,3)}+p_{(1,2),(1,3)}^{(1,3)}+p_{(2,1),(2,2)}^{(1,3)}.$$ Since $p_{(2,1),(2,2)}^{(1,3)}=\frac{k_{1,2}}{t}$, we get
	\begin{align}
	(1-\frac{1}{t})k_{1,2}=p_{(1,2),(2,1)}^{(1,3)}+p_{(1,2),(1,3)}^{(1,3)}.\label{jiegou-1}\tag{7.4}
	\end{align}
	In view of Lemma \ref{jl} \ref{jl-6'}, one obtains
	$k_{1,2}=p_{(1,2),(1,3)}^{(1,3)}+p_{(1,2),(2,1)}^{(1,3)}+p_{(1,2),(3,1)}^{(1,3)}+p_{(1,2),(1,2)}^{(1,3)}+p_{(1,2),(2,2)}^{(1,3)}.$
Since $p_{(1,2),(2,2)}^{(1,3)}=\frac{k_{1,2}}{t}$, we get
	\begin{align}
	(1-\frac{1}{t})k_{1,2}=p_{(1,2),(1,3)}^{(1,3)}+p_{(1,2),(2,1)}^{(1,3)}+p_{(1,2),(3,1)}^{(1,3)}+p_{(1,2),(1,2)}^{(1,3)}.\label{jiegou-2}\tag{7.5}
	\end{align}
	By \eqref{jiegou-1} and \eqref{jiegou-2}, one has  $p_{(1,2),(3,1)}^{(1,3)}=0$. It follows from Lemma \ref{jb} \ref{jb-2}  that $p_{(1,3),(1,3)}^{(1,2)}=0$. By Lemma \ref{jl} \ref{jl-3'},  we have $\Gamma_{1,3}^2\subseteq \{ \Gamma_{1,3},\Gamma_{2,2}\}$.
	
	By   Lemma \ref{jb} \ref{jb-2} and \ref{jb-3}, one gets $k_{1,3}=p_{(3,1),(1,3)}^{(1,3)}+p_{(3,1),(2,2)}^{(1,3)}=p_{(1,3),(1,3)}^{(1,3)}+p_{(3,1),(2,2)}^{(1,3)}$. \eqref{5-1} implies $(1-\frac{1}{t})k_{1,3}=p_{(3,1),(1,3)}^{(1,3)}=p_{(1,3),(1,3)}^{(1,3)}$.
	In view of Lemma \ref{jl} \ref{jl-2'}, we have $$k_{1,3}=p_{(1,3),(1,2)}^{(1,3)}+p_{(1,3),(2,1)}^{(1,3)}+p_{(1,3),(1,3)}^{(1,3)}+p_{(1,3),(3,1)}^{(1,3)}+p_{(1,3),(2,2)}^{(1,3)}+1.$$ Since $p_{(1,3),(2,2)}^{(1,3)}=\frac{k_{1,3}}{t}-1$, one obtains
	\begin{align}
	(1-\frac{1}{t})k_{1,3}&=p_{(1,3),(1,2)}^{(1,3)}+p_{(1,3),(2,1)}^{(1,3)}+p_{(1,3),(1,3)}^{(1,3)}+p_{(1,3),(3,1)}^{(1,3)}\geq2(1-\frac{1}{t})k_{1,3}.\nonumber
	\end{align}
Then $t=1$ and $V$ is the unique partite set such that $\Gamma_{1,3}(x)\cap V\ne \emptyset$. In view of  \eqref{5-1'}, we have  $|\Gamma_{1,q}(x)\cap V|=|\Gamma_{q,1}(x)\cap V|=k_{1,q}$ for $q\in \{2,3\}$. Then
 $ \Gamma_{1,2}(x)\cup \Gamma_{2,1}(x)\cup \Gamma_{1,3}(x)\cup \Gamma_{3,1}(x)\subseteq V$.  Since $3\in T$, there exists a circuit of length $3$ containing an arc of type $(1,2)$. By Lemma \ref{jichu0},  $\Gamma$ has at least three distinct partite sets. Let $V'$ be a partite set with $x\notin V'$ and $V'\ne V$. Since $ \Gamma_{1,2}(x)\cup \Gamma_{2,1}(x)\cup \Gamma_{1,3}(x)\cup \Gamma_{3,1}(x)\subseteq V$, from Lemma \ref{jichu0}, one has $V'=\emptyset$, a contradiction.

\begin{stepp}\label{chaifen}
	Show that $p_{(2,1),(2,2)}^{(1,3)}=\frac{k_{1,2}}{t}$, $p_{(1,3),(2,2)}^{(1,3)}=\frac{k_{1,3}}{t}-1$ and $p_{(1,2),(2,2)}^{(1,3)}=0$.
\end{stepp}
	Pick a partite set $V$ such that $\Gamma_{1,3}(x)\cap V\ne \emptyset$.
Since $\Gamma_{2,2}\in \Gamma_{1,2}\Gamma_{1,3}$, from Lemma \ref{le216fj}, one has $|\Gamma_{2,1}(x)\cap V|=p_{(2,1),(2,2)}^{(1,3)}=\frac{k_{1,2}}{t}$. Since $V$ was arbitrary and $|\Gamma_{2,1}(x)\cap V|=\frac{k_{1,2}}{t}$, there are exactly $t$ partite sets having nonempty intersection with $\Gamma_{2,1}(x)$.  Since $\Gamma_{2,2}\in \Gamma_{1,2}\Gamma_{1,3}$, from Lemma \ref{le216fj}, one gets $|\Gamma_{1,3}(x)\cap V|=p_{(1,3),(2,2)}^{(2,1)}=\frac{k_{1,3}}{t}$. By Lemmas \ref{main3''} and  \ref{changyong3}, we get $p_{(1,3),(2,2)}^{(1,3)}=\frac{k_{1,3}}{t}-1$.  
	
	By Lemma \ref{jl} \ref{jl-3'} and Step \ref{jiegou}, we get $\Gamma_{1,3}^2\subseteq \{\Gamma_{1,2},\Gamma_{1,3}\}$. Since $\Gamma_{1,3}(x)\cap V\ne \emptyset$, from Lemma \ref{jiaojik},  one has $\Gamma_{3,1}(x)\cap V=\emptyset$. Assume that there are exactly $l$ partite sets of $\Gamma$ having nonempty intersection with $\Gamma_{3,1}(x)$. Pick a partite set $V'$ such that $\Gamma_{3,1}(x)\cap V'\ne \emptyset$.  Since $\Gamma_{2,2}\in \Gamma_{1,2}\Gamma_{1,3}$, from Lemma \ref{le216fj}, we obtain  $|\Gamma_{1,2}(x)\cap V'|=p_{(1,2),(2,2)}^{(3,1)}=\frac{k_{1,2}}{l}$.
	Since $V'$ was arbitrary,  there are exactly $l$ partite sets having nonempty intersection with $\Gamma_{1,2}(x)$. Since $\Gamma_{3,1}(x)\cap V=\emptyset$, we get $\Gamma_{1,2}(x)\cap V=\emptyset$.
	By Lemma \ref{pssize}, one has
	$|V|=\frac{k_{1,2}+k_{1,3}}{t}.$ In view of Lemmas \ref{jichu0} and  \ref{changyong3}, we obtain $|V|=k_{2,2}+1$. Then $k_{2,2}+1=\frac{k_{1,2}+k_{1,3}}{t}$.  In view of Lemma \ref{jb} \ref{jb-2} and Step \ref{jiegou}, one gets $p_{(1,3),(2,2)}^{(3,1)}=p_{(3,1),(3,1)}^{(2,2)}=p_{(1,3),(1,3)}^{(2,2)}=0$.  By the first statement in Lemma \ref{jichu2}, we have $\Gamma_{1,3}\Gamma_{2,2}\subseteq\{\Gamma_{1,2},\Gamma_{2,1},\Gamma_{1,3}\}$.
	In view of  Lemma \ref{jb} \ref{jb-2} and	\ref{jb-3}, one has
	\begin{align}
	\frac{k_{1,2}+k_{1,3}}{t}-1=k_{2,2}=p_{(2,2),(1,2)}^{(1,3)}+p_{(2,2),(2,1)}^{(1,3)}+p_{(2,2),(1,3)}^{(1,3)}.\nonumber
	\end{align}
 Since $p_{(2,1),(2,2)}^{(1,3)}=\frac{k_{1,2}}{t}$ and $p_{(1,3),(2,2)}^{(1,3)}=\frac{k_{1,3}}{t}-1$, we get  $p_{(1,2),(2,2)}^{(1,3)}=p_{(2,2),(1,2)}^{(1,3)}=0$. 

\begin{stepp}\label{chaifen2}
	Show that $t=1$.
\end{stepp}
	In view of Lemma \ref{jb} \ref{jb-2} and Steps \ref{jiegou}, \ref{chaifen}, one gets $p_{(1,3),(2,2)}^{(3,1)}=p_{(1,3),(2,2)}^{(1,2)}=0$.  By the first statement in Lemma \ref{jichu2}, we have $\Gamma_{1,3}\Gamma_{2,2}\subseteq\{\Gamma_{2,1},\Gamma_{1,3}\}$.
	
	 By setting $i=l=j=(1,3)$ and $m=(2,2)$ in Lemma \ref{jb} \ref{jb-4},  one obtains
	\begin{align}
	&p_{(1,3),(1,3)}^{(1,3)}p_{(1,3),(2,2)}^{(1,3)}=p_{(1,3),(2,2)}^{(1,3)}p_{(1,3),(1,3)}^{(1,3)}+p_{(1,3),(2,2)}^{(2,1)}p_{(2,1),(1,3)}^{(1,3)}\nonumber
	\end{align}
	from Lemma \ref{jb} \ref{jb-2} and Lemma \ref{jl} \ref{jl-3'}, which implies that $p_{(1,3),(2,2)}^{(2,1)}p_{(2,1),(1,3)}^{(1,3)}=0$. By Lemma \ref{jb} \ref{jb-2} and Step \ref{chaifen}, we get $p_{(1,3),(2,2)}^{(2,1)}=\frac{k_{1,3}}{t}$. It follows that $p_{(2,1),(1,3)}^{(1,3)}=0$. It follows from Lemma \ref{jb} \ref{jb-2} that  $p_{(1,2),(1,3)}^{(1,3)}=p_{(1,3),(3,1)}^{(2,1)}=p_{(3,1),(1,3)}^{(1,2)}=0$. Lemma \ref{jl} \ref{jl-2'} and \ref{jl-3'} imply $\Gamma_{1,3}\Gamma_{3,1}\subseteq \{\Gamma_{0,0}, \Gamma_{1,3},\Gamma_{3,1},\Gamma_{2,2}\}$ and $\Gamma_{1,2}\Gamma_{1,3}\subseteq \{\Gamma_{1,2},\Gamma_{2,2}\}$.
	
	By Step \ref{chaifen}  and the first statement in Lemma \ref{jichu2}, we have  $\Gamma_{1,2}\Gamma_{2,2}\subseteq\{\Gamma_{1,2},\Gamma_{2,1},\Gamma_{3,1}\}$.	By setting $i=l^*=(1,3)$,  $j=(1,2)$ and $m=(2,2)$ in Lemma \ref{jb} \ref{jb-4}, one obtains	$p_{(1,3),(3,1)}^{(3,1)}p_{(3,1),(2,2)}^{(1,2)}=0$
	since $\Gamma_{1,3}\Gamma_{2,2}\subseteq\{\Gamma_{2,1},\Gamma_{1,3}\}$ and $\Gamma_{1,2}\Gamma_{1,3}\subseteq \{\Gamma_{1,2},\Gamma_{2,2}\}$. Since $p_{(3,1),(2,2)}^{(1,2)}=p_{(1,3),(2,2)}^{(2,1)}=\frac{k_{1,3}}{t}$, we have $p_{(1,3),(3,1)}^{(3,1)}=0$, which implies $\Gamma_{1,3}\Gamma_{3,1}\subseteq \{\Gamma_{0,0},\Gamma_{2,2}\}$. By Lemma \ref{jb} \ref{jb-2}, \ref{jb-3} and Step \ref{chaifen}, one gets $k_{1,3}=p_{(1,3),(0,0)}^{(1,3)}+p_{(1,3),(2,2)}^{(1,3)}=\frac{k_{1,3}}{t}$, and so $t=1$.

\begin{stepp}\label{lemma7.6s}
  	Show that $\Gamma_{1,3}^2=\{\Gamma_{1,2}\}$ and $\Gamma_{1,2}\Gamma_{1,3}=\{\Gamma_{2,2}\}$.
\end{stepp}

 By Steps \ref{chaifen}  and \ref{chaifen2}, one has $p_{(1,3),(2,2)}^{(1,3)}=k_{1,3}-1$. In view of Lemma \ref{jb} \ref{jb-2}, \ref{jb-3} and Lemma   \ref{jl} \ref{jl-2'}, we get $$k_{1,3}=2p_{(1,3),(1,2)}^{(1,3)}+2p_{(1,3),(1,3)}^{(1,3)}+p_{(1,3),(2,2)}^{(1,3)}+1,$$ which implies that $p_{(1,3),(1,2)}^{(1,3)}=p_{(1,3),(1,3)}^{(1,3)}=0$.
By Lemma \ref{jl} \ref{jl-3'} and Step \ref{jiegou}, we have  $\Gamma_{1,3}^2=\{\Gamma_{1,2}\}$.


	
Since $p_{(1,3),(1,2)}^{(1,3)}=0$, from  Lemma \ref{jl} \ref{jl-3'}, we have $\Gamma_{1,2}\Gamma_{1,3}\subseteq\{\Gamma_{1,2},\Gamma_{2,2}\}$. In view of  Lemma \ref{jb} \ref{jb-2} and \ref{jb-3}, we get $k_{1,2}=p_{(2,1),(1,2)}^{(1,3)}+p_{(2,1),(2,2)}^{(1,3)}$. By Steps \ref{chaifen}  and  \ref{chaifen2}, one has  $p_{(2,1),(2,2)}^{(1,3)}=k_{1,2}$, which implies $p_{(1,2),(1,3)}^{(1,2)}=p_{(2,1),(1,2)}^{(1,3)}=0$. Thus, $\Gamma_{1,2}\Gamma_{1,3}=\{\Gamma_{2,2}\}$.

\begin{stepp}
Show that $\Gamma$ is isomorphic to  the digraph in Theorem \ref{xushumain4} \ref{main1-fu2} with $n=1$.
\end{stepp}
In view of Lemma \ref{jb} \ref{jb-2} and Steps \ref{jiegou},  \ref{chaifen}, one gets $p_{(1,3),(2,2)}^{(3,1)}=p_{(1,3),(2,2)}^{(1,2)}=0$.  By the first statement in Lemma \ref{jichu2}, we have $\Gamma_{1,3}\Gamma_{2,2}\subseteq\{\Gamma_{2,1},\Gamma_{1,3}\}$. In view of  Lemma \ref{jb} \ref{jb-2} and	\ref{jb-3}, one has
$k_{2,2}=p_{(2,2),(2,1)}^{(1,3)}+p_{(2,2),(1,3)}^{(1,3)}.$
Steps \ref{chaifen} and \ref{chaifen2} imply that
\begin{align}
k_{1,2}+k_{1,3}-1=k_{2,2}.\label{lemma7.6s-7}\tag{7.6}
\end{align}

	By Step \ref{chaifen}, we have $p_{(1,2),(2,2)}^{(1,3)}=0$. In view of  Lemma \ref{jb} \ref{jb-2}, \ref{jb-3} and Lemma \ref{jl} \ref{jl-6'}, one obtains
	\begin{align}
	k_{1,2}=p_{(1,2),(1,2)}^{(1,3)}+p_{(1,2),(2,1)}^{(1,3)}+p_{(1,2),(1,3)}^{(1,3)}+p_{(1,2),(3,1)}^{(1,3)}.\label{1614}\tag{7.7}
	\end{align}
	 In view of Step \ref{lemma7.6s} and Lemma \ref{jb} \ref{jb-2}, \ref{jb-3}, we have $k_{1,3}=p_{(3,1),(1,2)}^{(1,3)}$ and $p_{(1,2),(2,1)}^{(1,3)}=p_{(1,2),(1,3)}^{(1,2)}=p_{(1,2),(1,3)}^{(1,3)}=0$. By substituting $k_{1,3}=p_{(3,1),(1,2)}^{(1,3)}$  and $p_{(1,2),(2,1)}^{(1,3)}=p_{(1,2),(1,3)}^{(1,3)}=0$
	into \eqref{1614}, 	
	one has
	\begin{align}
	p_{(1,2),(1,2)}^{(1,3)}=k_{1,2}-k_{1,3}.\label{lemma7.6s-4}\tag{7.8}
	\end{align}

	By setting $i=l=j=(1,3)$ and $m=(1,2)$ in Lemma \ref{jb} \ref{jb-4}, from  Step \ref{lemma7.6s}, we get
\begin{align}
p_{(1,3),(1,3)}^{(1,2)}p_{(1,2),(1,2)}^{(1,3)}=p_{(1,2),(1,3)}^{(2,2)}p_{(2,2),(1,3)}^{(1,3)}.\label{t34xjia-1}\tag{7.9}
	\end{align}
In view of Lemma \ref{jb} \ref{jb-1} and Step \ref{lemma7.6s}, one gets $p_{(1,3),(1,3)}^{(1,2)}=\frac{k_{1,3}^2}{k_{1,2}}$ and  $p_{(1,2),(1,3)}^{(2,2)}=\frac{k_{1,2}k_{1,3}}{k_{2,2}}.$ By Steps \ref{chaifen}  and  \ref{chaifen2}, we have $p_{(1,3),(2,2)}^{(1,3)}=k_{1,3}-1$.
\eqref{lemma7.6s-4} and \eqref{t34xjia-1} imply $$\frac{k_{1,3}(k_{1,2}-k_{1,3})}{k_{1,2}}=\frac{k_{1,2}(k_{1,3}-1)}{k_{2,2}}.$$ By \eqref{lemma7.6s-7}, we obtain $$k_{1,3}(k_{1,2}-k_{1,3})(k_{1,2}+k_{1,3}-1)=k_{1,2}^2(k_{1,3}-1),$$ and so $k_{1,3}^3-k_{1,2}^2-k_{1,3}^2+k_{1,2}k_{1,3}=0.$
Let $s=\frac{k_{1,2}}{k_{1,3}}$. It follows that $k_{1,3}=s^2-s+1$.
By  \eqref{lemma7.6s-4},	we get $s\geq 1$. Then there exist two coprime positive integers $p$ and $q$ such that $s=\frac{p}{q}$ with $p\geq q\geq 1$. Since $p_{(1,3),(1,3)}^{(1,2)}=\frac{k_{1,3}^2}{k_{1,2}}$, one gets $p_{(1,3),(1,3)}^{(1,2)}=\frac{k_{1,3}}{s}=\frac{s^2-s+1}{s}\in \mathbb{Z}$, and so $\frac{p^2-pq+q^2}{pq}\in \mathbb{Z}$. Then $p\mid q^2$. Since $p$ and $q$ are coprime, one obtains  $p=1$, which implies $q=1$. Thus, $s=1$ and $k_{1,2}=k_{1,3}=1$. Since $k_{2,2}=k_{1,2}+k_{1,3}-1$, we have $k_{2,2}=1$. Lemma \ref{changyong3} implies $|V\Gamma|=6$. By \cite[Theorem 1.1]{KSW03}, the desired result follows.
\end{proof}

\begin{prop}\label{t34xjia}
	The digraph $\Gamma$ is isomorphic to one of the digraphs in Theorem \ref{xushumain4} \ref{main1-fu2} and \ref{main1-fu0}.
\end{prop}
\begin{proof}
If $(3,3)\notin \tilde{\partial}(\Gamma)$, from Lemmas \ref{t34xj2*dier} and \ref{zuihouxli}, then $\Gamma$ is isomorphic to one of the digraphs in Theorem \ref{xushumain4} \ref{main1-fu2}   and \ref{main1-fu0} with $n=1$.
Now suppose $(3,3)\in \tilde{\partial}(\Gamma)$.
			By Lemma \ref{changyong3}, we get $\tilde{\partial}(\Gamma)\subseteq\{(0,0),(1,2),(2,1),(1,3),(3,1),(2,2),(3,3)\}$.
			Let  $F=\{\Gamma_{0,0},\Gamma_{3,3}\}$.  By Lemma \ref{changyong2},  $F$ is closed. Denoted  $\hat{\Gamma}:=\Gamma/F$.
			By Lemma \ref{tdxw1},  $\Gamma$ is  isomorphic to $\hat{\Gamma}\circ \overline{{K}}_{k_{3,3}+1}$, where $\hat{\Gamma}$ is a semicomplete multipartite commutative  weakly distance-regular digraph with $\tilde{\partial}(\hat{\Gamma})= \tilde{\partial}(\Gamma)\backslash \{(3,3)\}$. By Lemmas \ref{t34xj2*dier} and \ref{zuihouxli},
				$\Gamma$ is isomorphic to one of the digraphs in Theorem \ref{xushumain4} \ref{main1-fu2}   and \ref{main1-fu0} with $n\geq 2$.
\end{proof}

\section{$T=\{2,3\}$}
In  this section, $\Gamma$ always denotes  a semicomplete multipartite commutative weakly distance-regular digraph with $T=\{2,3\}$.  We shall prove that $\Gamma$ is  isomorphic to  one of the digraphs in Theorem \ref{xushumain4} \ref{main1-fu7} and \ref{main1-fu1}.

\begin{lemma}\label{g21}
We have $\tilde{\partial}(\Gamma)=\{(0,0),(1,1),(1,2),(2,1),(2,2)\}$.
%
\end{lemma}
\begin{proof}
	By Lemmas \ref{changyong} and \ref{zuidaarc}, the desired result follows.
\end{proof}

Pick a vertex  $x\in V\Gamma$.  Assume that there are exactly $t$ partite sets having nonempty intersection with $\Gamma_{1,2}(x)$.

\begin{lemma}\label{g2jiegou}
	Let  $\Gamma_{2,2}\in \Gamma_{1,2}\Gamma_{1,1}$.  Then the following hold:
	\begin{enumerate}
		\item\label{g2jiegou-6}	$p_{(1,2),(2,2)}^{(1,2)}=\frac{k_{1,2}}{t}-1$;
		
			\item\label{g2jiegou-7}	$p_{(2,1),(2,2)}^{(1,2)}=\frac{k_{1,2}}{t}$.
	\end{enumerate}
\end{lemma}
\begin{proof}
 Pick a partite set $V$ such that $\Gamma_{1,2}(x)\cap V\ne \emptyset$.   	Since $\Gamma_{2,2}\in \Gamma_{1,2}\Gamma_{1,1}$, from Lemma \ref{le216fj}, we have $|\Gamma_{1,1}(x)\cap V|=p_{(1,1),(2,2)}^{(1,2)}=\frac{k_{1,1}}{t}$.
	
	Since $V$ was arbitrary,  there are exactly $t$ partite sets having nonempty intersection with $\Gamma_{1,1}(x)$. By Lemma \ref{le216fj}, one obtains  $|\Gamma_{2,1}(x)\cap V|=p_{(2,1),(2,2)}^{(1,1)}=p_{(1,2),(2,2)}^{(1,1)}=|\Gamma_{1,2}(x)\cap V|=\frac{k_{1,2}}{t}$.
	
 By Lemma \ref{pssize}, we have $|V|=\frac{2k_{1,2}+k_{1,1}}{t}$. Lemma \ref{jichu0} implies $|V|=1+\sum_{h,l>1}k_{h,l}$. In view of Lemma \ref{g21}, we obtain $|V|=\frac{2k_{1,2}+k_{1,1}}{t}=k_{2,2}+1$.
	
	(i) Since $|\Gamma_{1,2}(x)\cap V|=\frac{k_{1,2}}{t}$,  from Lemmas \ref{main3''} and  \ref{g21}, we get $p_{(1,2),(2,2)}^{(1,2)}=\frac{k_{1,2}}{t}-1$. Thus, \ref{g2jiegou-6} holds.
	
	(ii) By the first statement in Lemma \ref{jichu2}, we have  $\Gamma_{1,2}\Gamma_{2,2}\subseteq\{\Gamma_{1,1},\Gamma_{1,2},\Gamma_{2,1}\}$. In view of Lemma \ref{jb} \ref{jb-2} and \ref{jb-3}, one has
	\begin{align}
	\frac{2k_{1,2}+k_{1,1}}{t}-1=k_{2,2}=p_{(2,2),(1,2)}^{(1,2)}+p_{(2,2),(2,1)}^{(1,2)}+p_{(2,2),(1,1)}^{(1,2)}.\label{2.3.1}\tag{8.1}
	\end{align}
 In view of  \ref{g2jiegou-6}, we get $p_{(2,2),(1,2)}^{(1,2)}=\frac{k_{1,2}}{t}-1$.
	By substituting $p_{(2,2),(1,2)}^{(1,2)}=\frac{k_{1,2}}{t}-1$ and $p_{(2,2),(1,1)}^{(1,2)}=\frac{k_{1,1}}{t}$
	into \eqref{2.3.1}, one obtains $p_{(2,1),(2,2)}^{(1,2)}=\frac{k_{1,2}}{t}$. Thus, \ref{g2jiegou-7} holds.
\end{proof}

\begin{lemma}\label{g2zzjiegou}
	Let  $\Gamma_{2,2}\in \Gamma_{1,2}\Gamma_{1,1}$. Then $\Gamma$ is  isomorphic to  one of the digraphs in Theorem \ref{xushumain4} \ref{main1-fu1}.
\end{lemma}
\begin{proof}
	In view of Lemma \ref{jb} \ref{jb-2}, \ref{jb-3} and Lemma \ref{g21}, we have
\begin{align}
k_{1,2}&=p_{(2,1),(1,1)}^{(1,2)}+p_{(2,1),(1,2)}^{(1,2)}+p_{(2,1),(2,1)}^{(1,2)}+p_{(2,1),(2,2)}^{(1,2)},\label{8-2}\tag{8.2}\\
k_{1,2}&=p_{(1,2),(1,1)}^{(1,2)}+p_{(1,2),(1,2)}^{(1,2)}+p_{(1,2),(2,1)}^{(1,2)}+p_{(1,2),(2,2)}^{(1,2)}+1.\label{8-3}\tag{8.3}
\end{align}
	By Lemma  \ref{g2jiegou} \ref{g2jiegou-7}, one gets $p_{(2,1),(2,2)}^{(1,2)}=\frac{k_{1,2}}{t}$, which implies that
	\begin{align} p_{(2,1),(1,1)}^{(1,2)}+p_{(2,1),(2,1)}^{(1,2)}=(1-\frac{1}{t})k_{1,2}-p_{(2,1),(1,2)}^{(1,2)}\label{8-3'}\tag{8.4}
	\end{align}
from \eqref{8-2}.
By	Lemma \ref{g2jiegou} \ref{g2jiegou-6}, one has $p_{(1,2),(2,2)}^{(1,2)}=\frac{k_{1,2}}{t}-1$, which implies
	\begin{align} p_{(1,2),(1,1)}^{(1,2)}+p_{(1,2),(1,2)}^{(1,2)}=(1-\frac{1}{t})k_{1,2}-p_{(1,2),(2,1)}^{(1,2)}\label{8-3''}\tag{8.5}
	\end{align}
from \eqref{8-3}. By Lemma \ref{jb} \ref{jb-2} and \eqref{8-3'}, \eqref{8-3''}, we have
\begin{align}
p_{(1,2),(1,1)}^{(2,1)}+p_{(1,2),(1,2)}^{(2,1)}=p_{(2,1),(1,1)}^{(1,2)}+p_{(2,1),(2,1)}^{(1,2)}=p_{(1,2),(1,1)}^{(1,2)}+p_{(1,2),(1,2)}^{(1,2)}.\label{8-3'''}\tag{8.6}
\end{align}

In view of Lemmas \ref{jichu} and \ref{g21},   the graph $(V\Gamma,\Gamma_{2,2})$ is  not connected,	which implies that $(V\Gamma,\{\Gamma_{0,0},\Gamma_{1,2},\Gamma_{2,1},\Gamma_{1,1},\Gamma_{2,2}\})$ is a non-symmetric impritive $4$-class association scheme.
 By Proposition  \ref{main prop doubly},  $\Gamma$ is a  doubly regular $(k,m)$-team semicomplete multipartite digraph of Type I or Type II.
By \eqref{8-3'''} and  Theorem \ref{main thm}, $\Gamma$ has Type II.
\end{proof}

\begin{lemma}\label{g23jiegouxj1}
 Let  $\Gamma_{2,2}\in \Gamma_{1,2}^2$.  Then $\Gamma$ is  isomorphic to  one of the digraphs in Theorem \ref{xushumain4} \ref{main1-fu1}.
\end{lemma}
\begin{proof}
	If $\Gamma_{2,2}\in \Gamma_{1,2}\Gamma_{1,1}$, from Lemma \ref{g2zzjiegou}, then the desired result follows. Suppose $\Gamma_{2,2}\notin \Gamma_{1,2}\Gamma_{1,1}$.
	Pick a partite set $V$ such that $\Gamma_{1,2}(x)\cap V\ne \emptyset$. Since $\Gamma_{2,2}\in \Gamma_{1,2}^2$,   from Lemma \ref{le216fj}, we have  $|\Gamma_{2,1}(x)\cap V|=p_{(2,1),(2,2)}^{(1,2)}=\frac{k_{1,2}}{t}$. In view of Lemmas \ref{main3''} and \ref{g21}, one obtains $p_{(1,2),(2,2)}^{(1,2)}=p_{(2,1),(2,2)}^{(2,1)}=\frac{k_{1,2}}{t}-1$.
	
%

Since $\Gamma_{2,2}\notin \Gamma_{1,2}\Gamma_{1,1}$, from Lemma \ref{g21}, we have $\Gamma_{1,2}\Gamma_{1,1}\subseteq\{\Gamma_{1,1},\Gamma_{1,2},\Gamma_{2,1}\}$. By  Lemma \ref{jb} \ref{jb-2} and \ref{jb-3}, one has
\begin{align}
k_{1,1}=p_{(1,1),(1,2)}^{(1,2)}+p_{(1,1),(2,1)}^{(1,2)}+p_{(1,1),(1,1)}^{(1,2)}.\label{bujia1}\tag{8.7}
\end{align}
In view of Lemma \ref{g21}, we get $k_{1,2}=p_{(1,2),(1,1)}^{(1,2)}+p_{(1,2),(1,2)}^{(1,2)}+p_{(1,2),(2,1)}^{(1,2)}+p_{(1,2),(2,2)}^{(1,2)}+1$ and $k_{1,2}=p_{(2,1),(1,1)}^{(1,2)}+p_{(2,1),(1,2)}^{(1,2)}+p_{(2,1),(2,1)}^{(1,2)}+p_{(2,1),(2,2)}^{(1,2)}$.
Since $p_{(1,2),(2,2)}^{(1,2)}=\frac{k_{1,2}}{t}-1$ and $p_{(2,1),(2,2)}^{(1,2)}=\frac{k_{1,2}}{t}$, one obtains \begin{align}
(1-\frac{1}{t})k_{1,2}&=p_{(1,2),(1,1)}^{(1,2)}+p_{(1,2),(1,2)}^{(1,2)}+p_{(1,2),(2,1)}^{(1,2)},\nonumber\\
(1-\frac{1}{t})k_{1,2}&=p_{(2,1),(1,1)}^{(1,2)}+p_{(2,1),(1,2)}^{(1,2)}+p_{(2,1),(2,1)}^{(1,2)}.\nonumber
\end{align}
By Lemma \ref{jb} \ref{jb-2}, we have
\begin{align}
p_{(1,2),(1,1)}^{(1,2)}+p_{(1,2),(1,2)}^{(1,2)}=p_{(2,1),(1,1)}^{(1,2)}+p_{(2,1),(2,1)}^{(1,2)}=p_{(1,2),(1,1)}^{(2,1)}+p_{(1,2),(1,2)}^{(2,1)}.\label{bujia1'}\tag{8.8}
\end{align}

In view of Lemmas \ref{jichu} and \ref{g21},  the graph $(V\Gamma,\Gamma_{2,2})$ is  not connected,	which implies that $(V\Gamma,\{\Gamma_{0,0},\Gamma_{1,2},\Gamma_{2,1},\Gamma_{1,1},\Gamma_{2,2}\})$ is a non-symmetric impritive $4$-class association scheme.
 By Proposition  \ref{main prop doubly},  $\Gamma$ is a  doubly regular $(k,m)$-team semicomplete multipartite digraph of Type I or Type II.
By	\eqref{bujia1'} and Theorem \ref{main thm}, $\Gamma$ has Type II.
\end{proof}

\begin{lemma}\label{g23jiegouxj2}
	Let $\Gamma_{2,2}\notin \Gamma_{1,2}\Gamma_{1,1}\cup \Gamma_{1,2}^2$. Then $\Gamma$ is  isomorphic to  one of the digraphs in Theorem \ref{xushumain4} \ref{main1-fu7}.
\end{lemma}
\begin{proof}
Since $\Gamma_{2,2}\notin \Gamma_{1,2}\Gamma_{1,1}\cup \Gamma_{1,2}^2$, from Lemma \ref{jb} \ref{jb-2},	one has $\Gamma_{2,1},\Gamma_{1,1}\notin \Gamma_{1,2}\Gamma_{2,2}$ and $\Gamma_{1,2},\Gamma_{2,1}\notin \Gamma_{1,1}\Gamma_{2,2}$. By the first statement in Lemma \ref{jichu2}, we have  $\Gamma_{1,q}\Gamma_{2,2}=\{\Gamma_{1,q}\}$ for $q\in \{1,2\}$.

Let $F=\{\Gamma_{0,0},\Gamma_{2,2}\}$.   Lemma \ref{g21} implies $\tilde{\partial}(\Gamma)=\{(0,0),(1,1),(1,2),(2,1),(2,2)\}$. By Lemma \ref{jichu3}, 	$F$ is closed.   We claim that $(x',y')\in\Gamma_{1,q}$ for $x'\in F(x)$ and $y'\in F(y)$ when $(x,y)\in\Gamma_{1,q}$ with $q\in \{1,2\}$. Without loss of generality, we may assume $x'\neq x$. Since $F=\{\Gamma_{0,0},\Gamma_{2,2}\}$, we have $(x',x)\in\Gamma_{2,2}$. Since  $\Gamma_{1,q}\Gamma_{2,2}=\{\Gamma_{1,q}\}$  and $(x,y)\in \Gamma_{1,q}$, one gets $(x',y)\in \Gamma_{1,q}$. If $y=y'$, then our claim is valid. If $y\neq y'$, then $(y,y')\in\Gamma_{2,2}$, and so $(x',y')\in \Gamma_{1,q}$ since $\Gamma_{1,q}\Gamma_{2,2}=\{\Gamma_{1,q}\}$ again. Thus, our claim is valid.
	
Denoted $\Sigma:=\Gamma/F$. Since $\Gamma$ is semicomplete multipartite,  from Lemma \ref{g21} and the claim, $\Gamma$ is isomorphic to $\Sigma\circ \overline{{K}}_{k_{2,2}+1}$, which implies $\tilde{\partial}_{\Gamma}(x,y)=\tilde{\partial}_{\Sigma}(F(x),F(y))$ for vertices $x,y$ with  $y\notin F(x)$. It follows that  $\tilde{\partial}(\Sigma)=\{(0,0),(1,1),(1,2),(2,1)\}$.
	
	Let $\wz{h}\in\wz{\partial}(\Sigma)$ and $(F(u),F(v))\in\Sigma_{\wz{h}}$. Then $(u,v)\in \Gamma_{\wz{h}}$. For $\wz{i},\wz{j}\in\wz{\partial}(\Sigma)$, we have
	$$P_{\wz{i},\wz{j}}(u,v)=\cup_{w\in P_{\wz{i},\wz{j}}(u,v)}F(w)
	=\cup_{F(w)\in \Sigma_{\wz{i}}(F(u))\cap\Sigma_{\wz{j}^*}(F(v))}F(w).$$
	Since $F=\{\Gamma_{0,0},\Gamma_{2,2}\}$, one gets $|F(w)|=k_{2,2}+1$, which implies $$|\Sigma_{\wz{i}}(F(u))\cap\Sigma_{\wz{j}^*}(F(v))|=\frac{p_{\wz{i},\wz{j}}^{\wz{h}}}{k_{2,2}+1}.$$ Since $\Gamma$ is a commutative weakly distance-regular digraph, $\Sigma$ is a  commutative weakly distance-regular digraph. Thus, $\Gamma$ is  isomorphic to  one of the digraphs in Theorem \ref{xushumain4} \ref{main1-fu7}.

\end{proof}	

\begin{prop}\label{g23mt}
The digraph  $\Gamma$ is  isomorphic to  one of the digraphs in Theorem \ref{xushumain4} \ref{main1-fu7} and \ref{main1-fu1}.
\end{prop}
\begin{proof}
By Lemmas \ref{g2zzjiegou}--\ref{g23jiegouxj2}, the desired result follows.
\end{proof}

\section{The proof of Theorem \ref{xushumain4}}
\begin{lemma}\label{mainfan}
If $\Sigma$ is a  semicomplete weakly distance-regular digraph with girth $3$, then $\Sigma\circ C_4$ is weakly distance-regular.
\end{lemma}
\begin{proof}
By \cite[Proposition 3.5]{YYF1}, $\tilde{\partial}(\Sigma)=\{(0,0),(1,2),(2,1)\}$.  Let $\Gamma=\Sigma\circ C_4$. Since $\tilde{\partial}(C_4)=\{(0,0),(1,3),(3,1),(2,2)\}$, we have
	\begin{align}
	\tilde{\partial}_{\Gamma}((x,u),(y,v))&=\tilde{\partial}_{\Sigma}(x,y)~{\rm for~all}~(x,u),(y,v)\in V\Gamma~{\rm with}~x\neq y\label{nec-1}\tag{9.1},\\
	\tilde{\partial}_{\Gamma}((x,u),(x,v))&=\tilde{\partial}_{C_4}(u,v)~{\rm for~all}~(x,u),(x,v)\in V\Gamma.\label{nec-2}\tag{9.2}
	\end{align}
	Then $\tilde{\partial}(\Gamma)=\tilde{\partial}(\Sigma)\cup \tilde{\partial}(C_4)$.  It suffices to show that $|P_{\wz{i},\wz{j}}((x,u),(y,v))|$ only depends on $\tilde{i},\tilde{j},\tilde{\partial}_{\Gamma}((x,u),(y,v))\in\tilde{\partial}(\Sigma)\cup \tilde{\partial}(C_4)$. Assume that $P_{\wz{i},\wz{j}}((x,u),(y,v))\neq\emptyset$ and $(x,u)\neq(y,v)$.
	
Suppose  $\tilde{\partial}_{\Gamma}((x,u),(y,v))\in \tilde{\partial}(\Sigma)$. Since $\tilde{\partial}(\Sigma)\cap \tilde{\partial}(C_4)=\{(0,0)\}$, from \eqref{nec-1}, we get $x\ne y$.   If $\tilde{i}$ or $\tilde{j}\in \tilde{\partial}(C_4)$, then $\tilde{j}=\tilde{\partial}_{\Sigma}((x,u),(y,v))$ or $\tilde{i}=\tilde{\partial}_{\Sigma}((x,u),(y,v))$ since $P_{\wz{i},\wz{j}}((x,u),(y,v))\neq\emptyset$, which implies
	$P_{\wz{i},\wz{j}}((x,u),(y,v))=\{(x,w)\mid\tilde{\partial}_{C_4}(u,w)=\tilde{i}\}$ or $P_{\wz{i},\wz{j}}((x,u),(y,v))=\{(y,w)\mid\tilde{\partial}_{C_4}(w,v)=\tilde{j}\}$
	by \eqref{nec-2}. If $\tilde{i},\tilde{j}\notin\tilde{\partial}(C_4)$,  from \eqref{nec-1}, then
	$|P_{\wz{i},\wz{j}}((x,u),(y,v))|=4|\{z\in V\Sigma\mid \tilde{\partial}_{\Sigma}(x,z)=\tilde{i}, \tilde{\partial}_{\Sigma}(z,y)=\tilde{j}\}|$ since $\tilde{\partial}(\Sigma)\cap \tilde{\partial}(C_4)=\{(0,0)\}$.
	
	Suppose  $\tilde{\partial}_{\Gamma}((x,u),(y,v))\in \tilde{\partial}(C_4)$. Since  $\tilde{\partial}(\Sigma)\cap \tilde{\partial}(C_4)=\{(0,0)\}$, from \eqref{nec-2}, we get $x=y$.  If $\tilde{i}\in \tilde{\partial}(C_4)$, then   $\tilde{j}\in \tilde{\partial}(C_4)$
	since $P_{\wz{i},\wz{j}}((x,u),(y,v))\neq\emptyset$, which implies
	 $P_{\wz{i},\wz{j}}((x,u),(y,v))=	\{(x,w)\mid \tilde{\partial}_{C_4}(u,w)=\tilde{i}, \tilde{\partial}_{C_4}(w,v)=\tilde{j}\}$ by \eqref{nec-2}. If $\tilde{i}\notin \tilde{\partial}(C_4)$, then $\tilde{j}=\tilde{i}^*$, which implies $|P_{\wz{i},\wz{j}}((x,u),(y,v))|=|\{z\in V\Sigma\mid\tilde{\partial}_{\Sigma}(x,z)=\tilde{i}\}|$ by \eqref{nec-1}.
	
Since $\Sigma$ and $C_4$ are weakly distance-regular, by above argument, $|P_{\wz{i},\wz{j}}((x,u),(y,v))|$ only depends on $\tilde{i},\tilde{j},\tilde{\partial}((x,u_x),(y,v_y))$, which implies that $\hat{\Gamma}$ is weakly distance-regular.
\end{proof}

\begin{proof}[Proof of Theorem \ref{xushumain4}]
	Routinely, all digraphs in Theorem \ref{xushumain4}  are semicomplete multipartite. By \cite[Proposition 2.4 and Theorems 1.1, 3.1]{KSW03} and Lemma \ref{mainfan}, all digraphs in Theorem \ref{xushumain4} are commutative weakly distance-regular digraphs.

	Now we  prove the necessity.   Proposition \ref{g234mprop} implies $|T|=1$, $T=\{3,4\}$ or $T=\{2,3\}$.  If $|T|=1$, by Proposition \ref{Tjishu1},  then $\Gamma$ is  isomorphic to  one of the digraphs in Theorem \ref{xushumain4} \ref{main1-fu7} and \ref{main1-fu5}.
If $T=\{3,4\}$, by Proposition \ref{t34xjia}, then $\Gamma$ is isomorphic to one of the digraphs in Theorem \ref{xushumain4} \ref{main1-fu2} and \ref{main1-fu0}.
If $T=\{2,3\}$, by Proposition \ref{g23mt}, then $\Gamma$ is  isomorphic to  one of the digraphs in Theorem \ref{xushumain4} \ref{main1-fu7} and \ref{main1-fu1}.
\end{proof}

\section*{Acknowledgements}

Y. Yang is supported by NSFC (12101575),  K. Wang is supported by the National Key R$\&$D Program of China (No.~2020YFA0712900) and NSFC (12071039, 12131011).

\section*{Data Availability Statement}

No data was used for the research described in the article.

	\end{document}